\documentclass[11pt]{article}
\usepackage{amsmath,amssymb,amsthm,amscd}

\newtheorem{theorem}{Theorem}[section]
\newtheorem{lemma}[theorem]{Lemma}
\newtheorem{proposition}[theorem]{Proposition}
\newtheorem{corollary}[theorem]{Corollary}

\theoremstyle{plain}

\theoremstyle{definition}
\newtheorem{definition}[theorem]{Definition}
\newtheorem{remark}[theorem]{Remark}

\numberwithin{equation}{section}

\renewcommand{\labelenumi}{\textup{(\theenumi)}}

\newcommand{\Homeo}{\operatorname{Homeo}}

\newcommand{\id}{\operatorname{id}}
\newcommand{\Coker}{\operatorname{Coker}}
\newcommand{\Ker}{\operatorname{Ker}}

\newcommand{\Ad}{\operatorname{Ad}}

\newcommand{\K}{\mathcal{K}}

\newcommand{\N}{\mathbb{N}}
\newcommand{\Z}{\mathbb{Z}}
\newcommand{\T}{\mathbb{T}}
\newcommand{\bbR}{\mathbb{R}}
\newcommand{\Zp}{{\mathbb{Z}}_+}

\title{Asymptotic continuous orbit  equivalence of Smale spaces 
and Ruelle  algebras }
\author{Kengo Matsumoto \\
Department of Mathematics \\
Joetsu University of Education \\
Joetsu, 943-8512, Japan
}

\begin{document}
\date{}
\maketitle


\def\det{{{\operatorname{det}}}}

\begin{abstract}
In the first part of the paper, we introduce notions of 
 asymptotic continuous orbit equivalence 
and asymptotic conjugacy in Smale spaces 
and characterize them in terms of their asymptotic Ruelle algebras with their dual actions.
In the second part, we introduce a groupoid $C^*$-algebra which is an extended version 
of the asymptotic Ruelle algebra from a Smale space
and study the extended Ruelle algebras from the view points of Cuntz--Krieger algebras. 
As a result, the asymptotic Ruelle algebra is realized as a fixed point algebra
of the extended Ruelle algebra under certain circle action.
\end{abstract}



\def\R{{\mathcal{R}}}
\def\RT{\tilde{\mathcal{R}}}

\def\OA{{{\mathcal{O}}_A}}
\def\OB{{{\mathcal{O}}_B}}
\def\RA{{{\mathcal{R}}_A^{s,u}}}
\def\RAC{{{\mathcal{R}}_A^\circ}}
\def\RAR{{{\mathcal{R}}_A^\rho}}
\def\RtA{{{\mathcal{R}}_{A^t}}}
\def\RB{{{\mathcal{R}}_B}}
\def\FA{{{\mathcal{F}}_A}}
\def\FB{{{\mathcal{F}}_B}}
\def\FA{{{\mathcal{F}}_A}}
\def\FtA{{{\mathcal{F}}_{A^t}}}
\def\FB{{{\mathcal{F}}_B}}
\def\DtA{{{\mathcal{D}}_{A^t}}}
\def\FDA{{{\frak{D}}_A^{s,u}}}
\def\FDtA{{{\frak{D}}_{A^t}}}
\def\FFA{{{\frak{F}}_A^{s,u}}}
\def\FFtA{{{\frak{F}}_{A^t}}}
\def\FFB{{{\frak{F}}_B}}
\def\OZ{{{\mathcal{O}}_Z}}
\def\V{{\mathcal{V}}}
\def\E{{\mathcal{E}}}
\def\OtA{{{\mathcal{O}}_{A^t}}}
\def\SOA{{{\mathcal{O}}_A}\otimes{\mathcal{K}}}
\def\SOB{{{\mathcal{O}}_B}\otimes{\mathcal{K}}}
\def\SOZ{{{\mathcal{O}}_Z}\otimes{\mathcal{K}}}
\def\SOtA{{{\mathcal{O}}_{A^t}\otimes{\mathcal{K}}}}
\def\DA{{{\mathcal{D}}_A}}
\def\DB{{{\mathcal{D}}_B}}
\def\DZ{{{\mathcal{D}}_Z}}

\def\SDA{{{\mathcal{D}}_A}\otimes{\mathcal{C}}}
\def\SDB{{{\mathcal{D}}_B}\otimes{\mathcal{C}}}
\def\SDZ{{{\mathcal{D}}_Z}\otimes{\mathcal{C}}}
\def\SDtA{{{\mathcal{D}}_{A^t}\otimes{\mathcal{C}}}}
\def\BC{{{\mathcal{B}}_C}}
\def\BD{{{\mathcal{B}}_D}}
\def\OAG{{\mathcal{O}}_{A^G}}
\def\OBG{{\mathcal{O}}_{B^G}}
\def\O2{{{\mathcal{O}}_2}}
\def\D2{{{\mathcal{D}}_2}}

\newcommand{\mathP}{\mathcal{P}}
\def\OAG{{\mathcal{O}}_{A^G}}
\def\DAG{{\mathcal{D}}_{A^G}}
\def\OBG{{\mathcal{O}}_{B^G}}

\def\OAGP{{\mathcal{O}}_{A^{{[\mathP]}}}} 
\def\DAGP{{\mathcal{D}}_{A^{{[\mathP]}}}} 
\def\OAGHP{{\mathcal{O}}_{{Z^{[\mathP]}}}} 
\def\DAGHP{{\mathcal{D}}_{{Z^{[\mathP]}}}} 
\def\AGP{A^{{[\mathP]}}} 
\def\AGHP{{Z^{[\mathP]}}} 
\def\OAGIP{{\mathcal{O}}_{A_{[\mathP]}}} 
\def\DAGIP{{\mathcal{D}}_{A_{[\mathP]}}} 
\def\OAGHIP{{\mathcal{O}}_{{Z_{[\mathP]}}}} 
\def\DAGHIP{{\mathcal{D}}_{{Z_{[\mathP]}}}} 
\def\AGIP{A_{[\mathP]}} 
\def\AGHIP{{Z_{[\mathP]}}}
\def\CKT{{\mathcal{T}}}
\def\Max{{{\operatorname{Max}}}}
\def\Per{{{\operatorname{Per}}}}
\def\PerB{{{\operatorname{PerB}}}}
\def\Homeo{{{\operatorname{Homeo}}}}
\def\HA{{{\frak H}_A}}
\def\HB{{{\frak H}_B}}
\def\HSA{{H_{\sigma_A}(X_A)}}
\def\Out{{{\operatorname{Out}}}}
\def\Aut{{{\operatorname{Aut}}}}
\def\Ad{{{\operatorname{Ad}}}}
\def\Inn{{{\operatorname{Inn}}}}
\def\Int{{{\operatorname{Int}}}}
\def\Min{{{\operatorname{Min}}}}
\def\interior{{{\operatorname{int}}}}
\def\det{{{\operatorname{det}}}}
\def\exp{{{\operatorname{exp}}}}
\def\cobdy{{{\operatorname{cobdy}}}}
\def\Ker{{{\operatorname{Ker}}}}
\def\ind{{{\operatorname{ind}}}}
\def\id{{{\operatorname{id}}}}
\def\supp{{{\operatorname{supp}}}}
\def\co{{{\operatorname{co}}}}
\def\Sco{{{\operatorname{Sco}}}}
\def\U{{{\mathcal{U}}}}
\def\Porb{{{\operatorname{P}_{\operatorname{orb}}}}}



\begin{enumerate}
\renewcommand{\labelenumi}{\arabic{enumi}}
\item
 Introduction
\item
 Smale spaces and their groupoids
\item
 Asymptotic continuous orbit equivalence
\item
Asymptotic periodic orbits of Smale spaces
\item
 Asymptotic  Ruelle algebras $\R_\phi^a$ with dual actions
\item
Asymptotic conjugacy 
\item
Extended Ruelle algebras $\R_\phi^{s,u}$


\item
Asymptotic continuous orbit equivalence in topological Markov shifts
\item
 Approach from Cuntz--Krieger algebras
\item
K-theory for the asymptotic Ruelle algebras for full shifts
\item
Concluding remarks



\end{enumerate}
\section{Introduction}
D. Ruelle has initiated a study of a basic class of hyperbolic dynamical systems,
called Smale spaces, from a view point of noncommutative operator algebras in \cite{Ruelle1}, \cite{Ruelle2}. 
Smale spaces are, roughly speaking,  hyperbolic dynamical systems
with local product structure.
His definition of Smale space was  motivated by the work of 
S. Smale \cite{Smale}, R. Bowen \cite{Bowen2}, \cite{Bowen3} and others. 
Two-sided subshifts of finite type are typical examples of Smale spaces.
Ruelle  has introduced non commutative algebras
from Smale spaces and studied equilibrium states on them.
After the Ruelle's papers,
 Ian F. Putnam \cite{Putnam1}, \cite{Putnam2}, \cite{Putnam3},
\cite{Putnam4},
 Putnam--Spielberg \cite{PutSp}
and Kaminker--Putnam--Spielberg
\cite{KamPutSpiel}
(cf. K. Thomsen \cite{Thomsen}, etc.) 
have investigated more detail on various kinds of  $C^*$-algebras
associated to
 Smale spaces from the view points of groupoids and structure theory of $C^*$-algebras. 
For a Smale space $(X,\phi)$,  
Putnam has considered the following six kinds of $C^*$-algebras written in \cite{Putnam1}, \cite{Putnam2}
\begin{equation}
S(X,\phi), \quad
U(X,\phi), \quad
A(X,\phi), \quad
S(X,\phi)\rtimes\Z, \quad
U(X,\phi)\rtimes\Z, \quad
A(X,\phi)\rtimes\Z. \label{eq:sixalg}
\end{equation}
The symbols $S, U, A$
correspond to stable, unstable, asymptotic equivalence relations, respectively.  
The second three algebras in the above six algebras
are crossed products of the first three algebras
by $\Z$-actions defined from automorphisms induced by $\phi$,
respectively.
Putnam has written the second three algebras 
as $R_s, R_u, R_a$, respectively and call them
the stable Ruelle algebra, the unstable Ruelle algebra, the asymptotic Ruelle algebra,
respectively (\cite{Putnam2}).
In this paper, we write them as
$\R_\phi^s, \R_\phi^u,\R_\phi^a$,
respectively to emphasize the original homeomorphism $\phi$.
 He pointed out that 
if $(X,\phi)$ is a shift of finite type defined by an irreducible square matrix 
$A$ with entries in $\{0,1\}$, the algebras 
$S(X,\phi)\rtimes\Z$
and $U(X,\phi)\rtimes\Z$
are isomorphic to
the stabilized Cuntz--Krieger algebras
$\SOA$ and $\SOtA$, respectively,
where $\K$ 
denotes the $C^*$-algebra of compact operators on 
a separable infinite dimensional Hilbert space.
Putnam--Spielberg \cite{PutSp}
(cf. Killough--Putnam \cite{KilPut})
also constructed another kinds of $C^*$-algebras
$S(X,\phi,P), U(X,\phi, P)$
and their crossed products
$S(X,\phi,P)\rtimes\Z, U(X,\phi,P)\rtimes\Z
$
from a $\phi$-invariant subset $P\subset X$ of periodic points 
by using \'etale groupoids defined by restricting stable, unstable equivalence
relations to $P$, respectively.
Although there are many different choices of $P$,
they are all Morita equivalent to
$
S(X,\phi),
U(X,\phi)
$
and
$
S(X,\phi)\rtimes\Z,
U(X,\phi)\rtimes\Z,
$
respectively.
In the present paper,  
we will not deal with these $C^*$-algebras
$S(X,\phi,P), U(X,\phi, P), S(X,\phi,P)\rtimes\Z, U(X,\phi,P)\rtimes\Z.
$

In this paper we will mainly focus on the algebra $\R_\phi^a$,
the last one in \eqref{eq:sixalg}.
By Putnam \cite{Putnam1}, 
 the algebra $\R_\phi^a$ is realized as 
the groupoid $C^*$-algebra
$C^*(G_\phi^a\rtimes\Z)$
of an  \'etale groupoid $G_\phi^a\rtimes\Z$.
Its unit space $(G_\phi^a\rtimes\Z)^\circ$
is identified with the original space $X$.
We naturally identify $C(X)$ with a subalgebra of $\R_\phi^a$.
A Smale space $(X,\phi)$ is said to be {\it asymptotically essentially free}\/
if the interior of the set of $n$-asymptotic periodic points
$\{x \in X \mid (\phi^n(x), x) \in G_\phi^a \}$
is empty for every $n\in \Z$ with $n\ne 0$.
If $(X,\phi)$ is irreducible and $X$ is not any finite set,
$(X,\phi)$ is  asymptotically essentially free (Lemma \ref{lem:irreessfree}).
We know that  $(X,\phi)$ is  asymptotically essentially free
if and only if the \'etale groupoid 
$G_\phi^a\rtimes \Z$ is essentially principal
(Lemma \ref{lem:essprin}).
Hence if  $(X,\phi)$ is irreducible and the space $X$ is infinite,
 the  $C^*$-algebra
 $\R_\phi^a$ is simple
(Proposition \ref{prop:simple}),
and  the $C^*$-subalgebra 
$C(X)$ is maximal abelian in $\R_\phi^a$.
Since 
$C^*(G_\phi^a\rtimes\Z)$ is canonically isomorphic to
the crossed product
$C^*(G_\phi^a)\rtimes\Z$
of the groupoid $C^*$-algebra
$C^*(G_\phi^a)$, which is the $C^*$-algebra $A(X,\phi)$
the third one in \eqref{eq:sixalg},
 by the integer group $\Z$ coming from the original transformation
$\phi$ on $X$,
the algebra  $\R_\phi^a$ has the dual action written $\rho^\phi_t$
of the circle group $\T = {\mathbb{R}}/\Z$.
Throughout the paper we assume that the space $X$ is infinite.

In the first part of this paper, 
we introduce a notion of asymptotic continuous orbit equivalence in Smale spaces, 
which will be defined in Section 2.
Roughly speaking, 
two Smale spaces are asymptotically continuous orbit equivalent  
if they are continuous orbit equivalent up to asymptotic equivalence.
We will show that  asymptotic continuous orbit equivalence in Smale spaces
is equivalent to their associated \'etale groupoids being isomorphic.
It corresponds to the fact that continuous orbit equivalence in one-sided topological Markov shifts is equivalent to their associated \'etale groupoids being isomorphic  
(cf. \cite{MMKyoto}, \cite{MatuiPLMS}, \cite{Matui2015}).
If two
Smale spaces  $(X,\phi)$ and $(Y,\psi)$ are 
asymptotically continuous orbit equivalent, 
written
$(X, \phi) \underset{ACOE}{\sim}(Y, \psi),$ 
then 
there exists a homeomorphism
$h:X\longrightarrow Y$
having certain  continuous homomorphisms
$c_\phi: G_\phi^a\rtimes\Z\longrightarrow \Z$
and
$c_\psi: G_\psi^a\rtimes\Z\longrightarrow \Z$.
The continuous homomorphisms define  unitary representations
$U_t(c_\phi)$ on  $l^2(G_\phi^a\rtimes\Z)$
and
$U_t(c_\psi)$ on  $l^2(G_\psi^a\rtimes\Z)$
of $\T$
which give rise to  actions
$\Ad(U_t(c_\phi))$ on  $\R_\phi^a$ of $\T$
and
$\Ad(U_t(c_\psi))$ on  $\R_\psi^a$ of $\T$,
respectively.
In Section 3 and Section 5,
we will show the following.

\begin{theorem}[{Theorem \ref{thm:main1} and Theorem \ref{thm:main3}}]
\label{thm:1.1}
Let
 $(X,\phi)$ and $(Y,\psi)$ be irreducible Smale spaces.
Then the following assertions are equivalent:
\begin{enumerate}
\renewcommand{\theenumi}{\roman{enumi}}
\renewcommand{\labelenumi}{\textup{(\theenumi)}}
\item
$(X, \phi) $ and $(Y, \psi)$
are asymptotically continuous orbit equivalent.
\item
The groupoids 
$G_{\phi}^{a} \rtimes \Z$ and $G_{\psi}^{a} \rtimes \Z$
are isomorphic as \'etale groupoids.
\item
There exists an isomorphism $\Phi:\R_\phi^a \longrightarrow \R_\psi^a$
of $C^*$-algebras such that 
$\Phi(C(X)) = C(Y)$
and 
$$
\Phi\circ \rho^\phi_t = \Ad(U_t(c_\psi))\circ \Phi,
\qquad
\Phi\circ\Ad(U_t(c_\phi)) =\rho^\psi_t\circ \Phi
\quad
\text{ for }
t \in \T
$$
for some continuous homomorphisms
$c_\phi: G_\phi^a\rtimes\Z\longrightarrow \Z$
and
$c_\psi: G_\psi^a\rtimes\Z\longrightarrow \Z$.
\end{enumerate}
\end{theorem}
In Section 4, we will prove that 
stably asymptotic continuous orbit equivalence of Smale spaces
preserves their periodic orbits, so that 
their zeta functions are related to each other by the associated cocycle functions
(Theorem \ref{thm:zeta}).

In Section 5, 
we  study
asymptotic continuous orbit equivalence in Smale spaces 
in terms of the dual actions of the associated Ruelle algebras.

In Section 6, we will introduce  a notion of asymptotic conjugacy 
between Smale spaces
$(X,\phi)$ and $(Y,\psi)$, written
$(X, \phi) \underset{a}{\cong}(Y, \psi).$ 
Roughly speaking, 
two Smale spaces are asymptotically conjugate 
if they are topologically conjugate up to asymptotic equivalences.
It is stronger than asymptotic continuous orbit equivalence but weaker than 
topological conjugacy.
The notion of asymptotic conjugacy in this paper is not the same as the notion of 
eventual conjugacy which is used in dynamical systems (cf.
\cite[Definition 7.7.14]{LM}).
Let 
$d_\phi: G_\phi^a\rtimes\Z\longrightarrow \Z$
and
$d_\psi: G_\psi^a\rtimes\Z\longrightarrow \Z$
be the continuous homomorphisms of \'etale groupoids defined by
\begin{equation*}
d_\phi(x,n,z) =n, \qquad d_\psi(y,m,w) =m 
\quad
\text{ for }
(x,n,z) \in G_\phi^a\rtimes\Z, \,
(y,m,w) \in G_\psi^a\rtimes\Z.
\end{equation*}
We will characterize  asymptotic conjugacy in terms of 
the Ruelle algebras with their dual actions
in the following way. 
\begin{theorem}[{Theorem \ref{thm:mains6}}]
\label{thm:1.2}
Let
$(X,\phi)$ and $(Y,\psi)$
be irreducible Smale spaces.
Then the following assertions are equivalent. 
\begin{enumerate}
\renewcommand{\theenumi}{\roman{enumi}}
\renewcommand{\labelenumi}{\textup{(\theenumi)}}
\item
$(X,\phi)$ and $(Y,\psi)$
are asymptotically conjugate.
\item There exists an isomorphism
$\varphi: G_\phi^a\rtimes\Z\longrightarrow G_\psi^a\rtimes\Z$ 
of \'etale groupoids
such that 
$d_\psi\circ\varphi = d_\phi$.
\item
There exists an isomorphism $\Phi: \R_\phi^a \longrightarrow \R_\psi^a$
of $C^*$-algebras such that 
$\Phi(C(X)) = C(Y)$
and 
$$
\Phi\circ \rho^\phi_t = \rho^\psi_t \circ \Phi
\quad
\text{ for }
t \in \T.
$$
\end{enumerate}
\end{theorem}
The asymptotic Ruelle algebra $\R_\phi^a$
has a translation invariant  faithful tracial state coming from
maximal measure called the Bowen measure on $X$.
Hence the algebra $\R_\phi^a$ is never purely infinite. 
In Section 7, we will introduce a unital, purely infinite version of $\R_\phi^a$.
The introduced $C^*$-algebra is denoted by $\R_\phi^{s,u}$, 
called the extended asymptotic Ruelle algebra.
 It has  a natural $\T^2$-action denoted by $\rho^{s,u}_\phi$.
 The fixed point algebra  $(\R_\phi^{s,u})^{\delta^\phi}$ 
of $\R_\phi^{s,u}$
under the diagonal $\T$-action
defined by $\delta^\phi_z = \rho^{s,u}_{\phi,(z,z)}, \, z \in \T$
 is isomorphic to the original
asymptotic Ruelle algebra $\R_\phi^a$
(Theorem \ref{thm:7.3}).

In Section 8, 9, 
we will apply the above discussions to shifts of finite type, which we call
topological Markov shifts, from the view point of Cuntz--Krieger algebras.
For  an irreducible and not permutation square matrix $A$ with entries in $\{0,1\},$
let us denote by
 $(\bar{X}_A, \bar{\sigma}_A)$
the associated two-sided topological Markov shift.
The dynamical system is a typical example of Smale space
as in \cite{Putnam1}, \cite{Putnam2}, \cite{Ruelle1}.
Consider 
the asymptotic Ruelle algebra
$\R_{\bar{\sigma}_A}^a$
and the extended asymptotic Ruelle algebra
$\R_{\bar{\sigma}_A}^{s,u}$
 for the topological Markov shift $(\bar{X}_A, \bar{\sigma}_A)$,
respectively.
Let $\rho^{A^t}$ and $\rho^A$ be the gauge actions on the Cuntz--Krieger algebras
$\OtA$ and $\OA$, respectively, 
where $A^t$ is the transpose of the matrix $A$.
We put the diagonal gauge action
$\delta^A_r= \rho^{A^t}_r\otimes\rho^A_r, r \in \T$
on $\OtA\otimes\OA$.
\begin{theorem}[{Theorem \ref{thm:RAGR} and Corollary \ref{cor:mains10.9}}]
Let  $(\bar{X}_A, \bar{\sigma}_A)$ 
be the Smale space of the two-sided topological Markov shift
defined by irreducible non-permutation matrix $A$ with entries in $\{0,1\}$.
Then there exists a projection $E_A$ in the tensor product $C^*$-algebra
$\OtA\otimes\OA$ such that 
$\delta^A_r(E_A) = E_A$ for all $r \in \T$
and 
the extended asymptotic Ruelle algebra 
$\R_{\bar{\sigma}_A}^{s,u}$
is isomorphic to the $C^*$-algebra  
$E_A(\OtA\otimes\OA)E_A$ denoted by $\R_{A}^{s,u}.$
The asymptotic Ruelle algebra 
$\R_{\bar{\sigma}_A}^a$ is isomorphic to the fixed point algebra 
${(\R_{A}^{s,u})}^{\delta^{A}}$
of $\R_{A}^{s,u}$ under the diagonal gauge action
$\delta^A.$
\end{theorem}
For the two-sided topological Markov shift $(\bar{X}_A, \bar{\sigma}_A),$
 we denote by 
$\R_{A}^{s,u}$ 
the extended asymptotic Ruelle algebra 
$\R_{\bar{\sigma}_A}^{s,u},$
which is identified with
the $C^*$-algebra  
$E_A(\OtA\otimes\OA)E_A,$
and by 
$\R_A^a$ the asymptotic Ruelle algebra 
$\R_{\bar{\sigma}_A}^a,$
which is identified with  the fixed point algebra 
of $E_A(\OtA\otimes\OA)E_A$ under the diagonal gauge action
$\delta^A$ by the above theorem.

In Section 10,
we will present 
the K-groups of the asymptotic 
Ruelle algebras $\R_\phi^a$ for some topological Markov shifts.
In Putnam \cite{Putnam1} and  Killough--Putnam \cite{KilPut}, 
the K-theory formula for the  asymptotic 
Ruelle algebras $\R_A^a$ for the topological Markov shift
$(\bar{X}_A,\bar{\sigma}_A)$
has been provided.
We will use Putnam's formula in \cite{Putnam1} 
to compute the K-groups of the $C^*$-algebra
$\R_A^a$ 
for 
the $N\times N$ matrix
$
A=
\begin{bmatrix} 
1& \cdots & 1 \\
\vdots &   & \vdots \\
1& \cdots & 1 \\
\end{bmatrix}
$
with all entries being $1$'s,
so that the topological Markov shift
$(\bar{X}_A,\bar{\sigma}_A)$
is the full  $N$-shift.
Let us denote by
$\R_N^a$ the asymptotic Ruelle algebra $\R_A^a$ 
for the matrix $A$.
The $C^*$-algebra $\R_N^a$ is 
a simple A$\mathbb{T}$-algebra of real rank zero 
with a unique tracial state written $\tau_N$.
 We will show that 
$K_0(\R_N^a) \cong K_1(\R_N^a) \cong \Z[\frac{1}{N}]$
(Proposition \ref{prop:K}) and
$\tau_{N*}(K_0(\R_N^a)) =\Z[\frac{1}{N}]$
(Lemma \ref{lem:dimensionrange}).
We then see that 
two algebras 
$\R_N^a$ and $\R_M^a$ are isomorphic if and only if 
$\Z[\frac{1}{N}] =\Z[\frac{1}{M}]$
(Proposition \ref{prop:classRuelle}).  
As a result, we know that 
the two-sided full 2-shift and the two-sided full 3-shift are 
not asymptotically continuous orbit equivalent
(Corollary \ref{coro:acoefull}).

 In Section 11,
 we refer to differences among  
asymptotic continuous orbit equivalence,   asymptotic  conjugacy and  topological conjugacy of Smale spaces, and present an open question related to their differences.

Throughout the paper,
we denote by $\Zp$ and $\N$ the set of nonnegative integers
and the set of positive integers, respectively.

\section{ Smale spaces and their groupoids}

Let $\phi$ be a homeomorphism on a compact metric space $(X, d)$ with metric $d$. 
Let us recall the definition of Smale space following 
D. Ruelle \cite[7.1]{Ruelle1} and I. F, Putnam \cite[Section 1]{Putnam1}.
Our notations are slightly changed from Ruelle's ones and Putnam's ones.
  For $\epsilon>0$, 
we set
\begin{equation*}
\Delta_{\epsilon} := \{(x,y) \in X\times X\mid d(x,y) <\epsilon \}.
\end{equation*}
Suppose that there exist $\epsilon_0>0$ and a continuous map
\begin{equation*}
[\cdot, \cdot]: (x,y) \in \Delta_{\epsilon_0} \longrightarrow [x,y]\in X  
\end{equation*}
having the following three properties called ({\bf SS1}):
\begin{enumerate}
\renewcommand{\theenumi}{\roman{enumi}}
\renewcommand{\labelenumi}{\textup{(\theenumi)}}
\item
$
[x,x] = x
$ for $ x \in X, $
\item
$
[[x,y],z] =[ x, [y,z]] = [x,z]  
$
 for
$ (x, y), (y,z), (x,z), ([x,y],z), ( x, [y,z]) \in \Delta_{\epsilon_0},$
\item 
$[\phi(x), \phi(y)] =\phi([x,y]) 
$
for
$ (x, y), (\phi(x),\phi(y)) \in \Delta_{\epsilon_0}.
$
\end{enumerate}
Put for $0<\epsilon \le \epsilon_0$
\begin{align}
X^s(x,\epsilon) & = \{ y \in X \mid [y,x] = y, \, d(x,y) <\epsilon \}, \\
X^u(x,\epsilon) & = \{ y \in X \mid [x,y] = y, \, d(x,y) <\epsilon \}.
\end{align}
We further require 
that there exists  $0<\lambda_0<1$ such that 
the following two properties called ({\bf SS2}) hold:
\begin{align}
d(\phi(y), \phi(z)) \le \lambda_0 d(y,z) \qquad \text{ for } y,z \in X^s(x,\epsilon), \label{eq:2.2.3}\\
d(\phi^{-1}(y), \phi^{-1}(z)) \le \lambda_0 d(y,z) \qquad \text{ for } y,z \in X^u(x,\epsilon).
\end{align}
\begin{definition}[{Ruelle \cite[7.1]{Ruelle1}}] \label{def:smalespace}
A Smale space is a topological dynamical system
$(X,\phi)$ of a homeomorphism $\phi$ on a compact metric space $X$
with a bracket $[\cdot, \cdot]$ satisfying ({\bf SS1}) and ({\bf SS2})
for suitable $\epsilon_0, \lambda_0$.
\end{definition}
By Ruelle \cite[7.1]{Ruelle1}, Putnam \cite[Section 1]{Putnam1},  
there exists $\epsilon_1$ with $0<\epsilon_1<\epsilon_0$
such that for any $\epsilon$ satisfying
$0<\epsilon <\epsilon_1$, the equalities  
\begin{align}
X^s(x,\epsilon) & = \{ y \in X \mid d(\phi^n(x), \phi^n(y))< \epsilon \text{ for all }
n=0,1,2,\dots \}, \\
X^u(x,\epsilon) & = \{ y \in X \mid d(\phi^n(x), \phi^n(y))< \epsilon \text{ for all }
n=0,-1,-2,\dots \}
\end{align}
hold.
\begin{lemma}[{Putnam \cite[Section 1]{Putnam1}, Ruelle \cite[7.1]{Ruelle1}}]
For $x,y \in X$ with $(x,y) \in \Delta_{\epsilon_0}$ 
and
$d(x,[y,x]), d(y,[y,x]) <\epsilon_1,$
\begin{equation*}
\{ [y,x]\} = X^u(y,\epsilon_1) \cap X^s(x,\epsilon_1).
\end{equation*}
Hence for $0<\epsilon< \epsilon_1$ and $x,y,z \in X$,
the equality
$[y,x]=z$ holds if and only if
$$
d(\phi^{-n}(y), \phi^{-n}(z))< \epsilon, \qquad
d(\phi^n(x), \phi^n(z))< \epsilon \quad \text{ for all }
n=0,1,2,\dots.
$$
\end{lemma}
This means that the bracket $[\cdot, \cdot]$ is determined by the original dynamics
$(X, d,\phi)$ if it exists.
The following lemma is also useful.
\begin{lemma}[{Putnam \cite[Section 1]{Putnam1}, Ruelle \cite[7.1]{Ruelle1}}]
\label{lem:2.1}
For any $\epsilon$ with $0 <\epsilon\le \epsilon_0$ and $x \in X$, we have
\begin{enumerate}
\renewcommand{\theenumi}{\roman{enumi}}
\renewcommand{\labelenumi}{\textup{(\theenumi)}}
\item
$\phi(X^s(x,\epsilon)) \subset X^s(\phi(x),\epsilon)). $
\item
$\phi^{-1}(X^u(x,\epsilon)) \subset X^u(\phi^{-1}(x),\epsilon)). $
\end{enumerate}
\end{lemma}
Following Putnam \cite[Section 1]{Putnam1}, we put for $x \in X$
\begin{align*}
X^s(x) & = \{ y \in X \mid \lim_{n\to{\infty}}d(\phi^n(x), \phi^n(y)) =0\}, \\
X^u(x) & = \{ y \in X \mid \lim_{n\to{\infty}}d(\phi^{-n}(x), \phi^{-n}(y)) =0\}, \\
X^a(x) & = X^s(x) \cap X^u(x).
\end{align*}
We note that the inclusion relations 
$X^s(x,\epsilon_0) \subset X^s(x)$ 
and
$X^u(x,\epsilon_0) \subset X^u(x)$ 
have been shown in \cite{Putnam1}.
The following lemma has been seen in 
\cite{Putnam1}, \cite{Ruelle1}. 
\begin{lemma}[{Putnam \cite[Section 1]{Putnam1}, Ruelle \cite[7.1]{Ruelle1}}]
For any $\epsilon$ with $0 <\epsilon\le \epsilon_0$ and $x \in X$, we have
\begin{enumerate}
\renewcommand{\theenumi}{\roman{enumi}}
\renewcommand{\labelenumi}{\textup{(\theenumi)}}
\item
$X^s(x) = \cup_{n=0}^\infty \phi^{-n}(X^s(\phi^n(x),\epsilon)). $
\item
$X^u(x) = \cup_{n=0}^\infty \phi^{n}(X^s(\phi^{-n}(x),\epsilon)). $
\end{enumerate}
\end{lemma}
Following Putnam \cite[Section 1]{Putnam1}, we put
\begin{align*}
G_{\phi}^{s,0} &= \{ (x,y) \in X \times X \mid y \in X^s(x,\epsilon_0) \}, \\
G_{\phi}^{u,0} &= \{ (x,y) \in X \times X \mid y \in X^u(x,\epsilon_0) \}, \\
G_{\phi}^{a,0} &= G_{\phi}^{s,0} \cap G_{\phi}^{u,0},
\end{align*}
and for $n\in \N$, 
\begin{align*}
G_{\phi}^{s,n} &= (\phi\times\phi)^{-n}(G_{\phi}^{s,0}), \\
G_{\phi}^{u,n} &= (\phi\times\phi)^{n}(G_{\phi}^{u,0}), \\
G_{\phi}^{a,n} &= G_{\phi}^{s,n} \cap G_{\phi}^{u,n}.
\end{align*}
All of them are given the relative topology of $X\times X$.

Since $[y,x] =y$ if and only if $ [x,y] =x$, one sees that 
$y \in X^s(x,\epsilon_0)$ 
if and only if 
$x \in X^s(y,\epsilon_0)$.
Hence $(x,y) \in G_{\phi}^{*,n}$ if and only if 
$(y,x) \in G_{\phi}^{*,n}$ for $* =s,u,a$. 

We note the following lemma, which is well-known and useful. 
\begin{lemma}\label{lem:2.5}
For $x,y \in X$ we have $(x,y) \in G_\phi^{a,0}$
if and only if $x = y$.
Hence we may identify
$G_\phi^{a,0}$ with $X$ as a topological space.
\end{lemma}
\begin{proof}
Take an arbitrary
$(x,y) \in G_\phi^{a,0}$.
As $(x,y) \in G_\phi^{s,0}$, 
we see that $y \in X^s(x,\epsilon_0)$ so that $[y,x] =y$,
and also 
as $(x,y) \in G_\phi^{u,0}$, 
we see that $y \in X^u(x,\epsilon_0)$ so that $[x,y] =y$.
Hence we have
$$
x = [x,x] = [x,[y,x]] = [x,y] = y.
$$
\end{proof}
By Lemma \ref{lem:2.1}, we know that 
\begin{equation}
G_\phi^{*, n} \subset G_\phi^{*, n+1}, \qquad * =s,u,a, \,\, n=0,1,\dots
\label{eq:Gphin}
\end{equation} 
Following \cite[Section 1]{Putnam1}, \cite[Section 3]{Putnam2} and 
\cite[Section 2]{PutSp},
we define several equivalence relations on $X$:
\begin{equation*}
G_{\phi}^{s} = \cup_{n=0}^\infty G_{\phi}^{s,n},\qquad
G_{\phi}^{u} = \cup_{n=0}^\infty G_{\phi}^{u,n},\qquad
G_{\phi}^{a} = \cup_{n=0}^\infty G_{\phi}^{a,n}.
\end{equation*}
By \eqref{eq:Gphin},
the set $G_{\phi}^{*} = \cup_{n=0}^\infty G_{\phi}^{*,n}$
is an inductive system of topological spaces.
Each $G_\phi^*, * =s,u,a$ is  endowed with the inductive limit topology.
The following lemma has been also shown by Putnam. 
\begin{lemma}[Putnam {\cite[Section 1]{Putnam1}}]\label{lem:n2.6}
\hspace{5cm}
\begin{enumerate}
\renewcommand{\theenumi}{\roman{enumi}}
\renewcommand{\labelenumi}{\textup{(\theenumi)}}
\item
$G_{\phi}^{s} = \{ (x,y) \in X\times X
\mid 
\lim\limits_{n\to\infty}d(\phi^n(x),\phi^n(y)) =0 \}.$
\item
$G_{\phi}^{u} = \{ (x,y) \in X\times X
\mid 
\lim\limits_{n\to\infty}d(\phi^{-n}(x),\phi^{-n}(y)) =0 \}.$
\item
$G_{\phi}^{a} = \{ (x,y) \in X\times X
\mid 
  \lim\limits_{n\to\infty}d(\phi^n(x),\phi^n(y)) 
=\lim\limits_{n\to\infty}d(\phi^{-n}(x),\phi^{-n}(y)) =0 \}.$
\end{enumerate}
\end{lemma}
Putnam have studied these three equivalence relations  
$G_{\phi}^{s}, \, G_{\phi}^{u} $ and $G_{\phi}^{a}$
on $X$ by regarding them as principal groupoids.
He pointed out that the third equivalence relation $G_{\phi}^a$ is an \'etale
equivalence relation whereas  
the first two  are not \'etale in general. 
He has also studied the associated groupoid $C^*$-algebras
$C^*(G_{\phi}^{s}), \, C^*(G_{\phi}^{u}) $ and $C^*(G_{\phi}^{a})$
which have been denoted by
$S(X,\phi), \, U(X,\phi)$
 and $A(X,\phi)$, respectively.
He has pointed out that they are all stably AF-algebras
if $(X,\phi)$ is a shift of finite type.
He also studied their semi-direct products by the integer group $\Z$ 
as groupoids
\begin{align*}
G_{\phi}^{s} \rtimes \Z= &\{(x,n,y) \in X \times \Z \times X \mid
(\phi^n(x), y) \in G_{\phi}^{s} \}, \\
G_{\phi}^{u} \rtimes \Z= &\{(x,n,y) \in X \times \Z \times X \mid
(\phi^n(x), y) \in G_{\phi}^{u} \}, \\
G_{\phi}^{a} \rtimes \Z= &\{(x,n,y) \in X \times \Z \times X \mid
(\phi^n(x), y) \in G_{\phi}^{a} \}.
\end{align*}
Since the map
\begin{equation}
\gamma: 
(x,n,y)\in G_{\phi}^{*} \rtimes \Z \rightarrow 
((x,\phi^{-n}(y)),n)\in G_{\phi}^{*} \times \Z
\label{eq:gamma}
\end{equation}
is bijective, the topology of the groupoid
$G_{\phi}^{*} \rtimes \Z$ is defined by the product topology of 
$G_{\phi}^{*} \times \Z$ through the map $\gamma$.
Let us denote by
$(G_{\phi}^{*} \rtimes \Z)^\circ$
the unit space of the groupoid
$G_{\phi}^{*} \rtimes \Z$
which is identified with
that of 
$G_{\phi}^{*}$
and naturally homeomorphic to
the original space $X$
through the correspondence
$(x,0,x) \in (G_{\phi}^{*} \rtimes \Z)^\circ \longrightarrow x \in X$
for $*=s,u,a.$
The range map
$r:G_{\phi}^{*} \rtimes \Z\rightarrow (G_{\phi}^{*} \rtimes \Z)^\circ$
and the source map
$s:G_{\phi}^{*} \rtimes \Z\rightarrow (G_{\phi}^{*} \rtimes \Z)^\circ$
are defined by 
\begin{equation*}
r(x,n,y) = (x,0,x)
\quad
\text{ and }
\quad
s(x,n,y) = (y,0,y).
\end{equation*}
The groupoid operations are defined by
\begin{gather*}
(x, n, y)\cdot (x',m,w) =(x, n+m, w)\quad \text{ if } y = x',\\
(x, n,y)^{-1} = (y,-n,x).
\end{gather*}

Putnam \cite{Putnam1}, \cite{Putnam2} 
and
Putnam--Spielberg \cite{PutSp}
have  also  studied their associated groupoid $C^*$-algebras
$C^*(G_{\phi}^{s}\rtimes \Z), \, C^*(G_{\phi}^{u}\rtimes \Z) $ 
and $C^*(G_{\phi}^{a}\rtimes \Z)$
which have been written 
$R_s, \, R_u$
 and $R_a$, respectively
in their  papers.
In this paper we denote them 
by $\R_\phi^s, \, \R_\phi^u$
 and $\R_\phi^a$, respectively,
to emphasize the homeomorphism
$\phi$.
We remark that Putnam--Spielberg \cite{PutSp}
(cf. Killough--Putnam \cite{KilPut})
also constructed another kinds of $C^*$-algebras
$S(X,\phi,P), U(X,\phi, P)$
and their crossed products
$S(X,\phi,P)\rtimes\Z, U(X,\phi,P)\rtimes\Z
$
from a $\phi$-invariant subset $P\subset X$ of periodic points
by using \'etale groupoids defined by 
restricting the stable equivalence relation $G_\phi^s$, 
unstable equivalence relations $G_\phi^u$ to $P$, respectively.
In the present paper,  
we will not deal with these $C^*$-algebras
$S(X,\phi,P), U(X,\phi, P), S(X,\phi,P)\rtimes\Z, U(X,\phi,P)\rtimes\Z.
$

\section{Asymptotic continuous orbit equivalence}
Let $(X,\phi)$ be a Smale space.
In this section, the symbol $d$ will be used as a two-cocycle.
It does not mean the metric on $X$. 
A sequence $\{ f_n \}_{n \in \Z}$ of integer-valued continuous functions on $X$ 
is called a {\it one-cocycle for $\phi$}\/ 
if they satisfy the identity  
\begin{equation}
f_n(x) + f_m(\phi^n(x)) = f_{n+m}(x),\qquad x \in X, \, n,m \in \Z. \label{eq:onecocycle}
\end{equation}
For a continuous function 
$f:X\longrightarrow \Z$ and $n \in \Z$,
we define
\begin{equation*}
f^n(x) =
\begin{cases}
\sum_{i=0}^{n-1} f(\phi^i(x)) & \text{ for } n>0, \\
0 & \text{ for } n=0,\\
- \sum_{i=n}^{-1}f(\phi^{i}(x)) & \text{ for } n<0.
\end{cases}                           
\end{equation*}
It is straightforward to see the following lemma.
\begin{lemma}
For $n, m\in \Z$, the identity
\begin{equation}
f^n(x) + f^m(\phi^n(x)) = f^{n+m}(x),\qquad x \in X
\label{eq:1cocycle}
\end{equation}
holds. Hence the sequence
$\{f^n\}_{n \in \Z}$ 
is a one-cocycle for $\phi$. 
\end{lemma}
We note that a sequence of functions
satisfying \eqref{eq:onecocycle}  
is determined by only $f_1$.

In what follows we focus on asymptotic equivalence relations
$G_\phi^a$.
A continuous function $d:G_\phi^a \longrightarrow \Z$
is called a {\it two-cocycle}\/ if it satisfies the following equalities:
\begin{equation}
d(x,z) + d(z,w) = d(x,w), \qquad (x,z), (z,w), (x,w) \in G_\phi^a. \label{eq:2cocycle}
\end{equation}

The identity \eqref{eq:2cocycle} means that $d: G_\phi^a \longrightarrow \Z$
is a groupoid homomorphism.
\begin{definition}\label{def:acoe}
Smale spaces
$(X,\phi)$ and $(Y, \psi)$ are said to be
{\it asymptotically continuously orbit equivalent}\/
if there exist a homeomorphism
$h: X\longrightarrow Y$,
continuous functions
$
c_1:X\longrightarrow \Z, \, c_2:Y\longrightarrow \Z,
$
and two-cocycle functions
$
d_1: G_\phi^a \longrightarrow \Z,\, d_2: G_\psi^a \longrightarrow \Z
$
such that 
\begin{enumerate}
\renewcommand{\theenumi}{\arabic{enumi}}
\renewcommand{\labelenumi}{\textup{(\theenumi)}}
\item
$c_1^m(x) + d_1(\phi^m(x), \phi^m(z))
=c_1^m(z) + d_1(x, z), \quad
(x,z) \in G_\phi^a, \, m \in \Z.$
\item
$c_2^m(y) + d_2(\psi^m(y), \psi^m(w))
=c_2^m(w) + d_2(y, w),\quad
(y, w) \in G_\psi^a, \, m \in \Z.$
\end{enumerate}
and,
\begin{enumerate}
\renewcommand{\theenumi}{\roman{enumi}}
\renewcommand{\labelenumi}{\textup{(\theenumi)}}
\item
For each $n \in \Z$, the pair
$( \psi^{c_1^n(x)}(h(x)),  h(\phi^n(x)))$ denoted by $\xi_1^n(x)$
 belongs to $G_\psi^a$ for each $x \in X$,
and
the map
$\xi_1^n: x \in X \longrightarrow \xi_1^n(x) \in G_\psi^a$ is continuous. 
\item
For each $n \in \Z$, the pair
$( \phi^{c_2^n(y)}(h^{-1}(y)),  h^{-1}(\psi^n(y)))$ denoted by $\xi_2^n(y)$
 belongs to $G_\phi^a$ for each $y \in Y$,
and
the map
$\xi_2^n: y \in Y \longrightarrow \xi_2^n(y) \in G_\phi^a$ is continuous. 
\item
The pair 
$(\psi^{d_1(x,z)}(h(x)), h(z))$ denoted by $\eta_1(x,z)$ belongs to $G_\psi^a$
for each $(x,z) \in G_\phi^a$, and the map
$\eta_1:(x,z) \in G_\phi^a\longrightarrow \eta_1(x,z)\in G_\psi^a$ is continuous.
\item
The pair 
$(\phi^{d_2(y,w)}(h^{-1}(y)), h^{-1}(w))$ denoted by $\eta_2(y,w)$ belongs to $G_\phi^a$
for each $(y,w) \in G_\psi^a$, and the map
$\eta_2:(y,w) \in G_\psi^a\longrightarrow \eta_2(y,w)\in G_\phi^a$ is continuous.
\item
$c^{c^n_1(x)}_2(h(x)) + d_2(\psi^{c_1^n(x)}(h(x)), h(\phi^n(x))) = n, \quad x \in X,\, n \in \Z.$
\item
$c^{c^n_2(y)}_1(h^{-1}(y)) + d_1(\phi^{c_2^n(y)}(h^{-1}(y)), h^{-1}(\psi^n(y))) = n, 
\qquad y \in Y,\, n \in \Z.$
\item
$c^{d_1(x,z)}_2(h(x)) + d_2(\psi^{d_1(x,z)}(h(x)), h(z)) = 0, \quad (x,z) \in G_\phi^{a}.$
\item
$c^{d_2(y,w)}_1(h^{-1}(y)) + d_1(\phi^{d_2(y,w)}(h^{-1}(y)), h^{-1}(w)) = 0, 
\qquad (y,w) \in G_\psi^{a}.$
\end{enumerate}
In this situation, we write 
$(X, \phi) \underset{ACOE}{\sim}(Y, \psi).$
\end{definition}

\begin{remark}\label{remark3.3}
{\bf 1.} The condition (1) above is equivalent to the condition
$$
c_1(x) + d_1(\phi(x), \phi(z))
=c_1(z) + d_1(x, z), \quad
(x,z) \in G_\phi^a,
$$
and the condition (2) is similar to (1).

{\bf 2.} The conditions (i), (ii), (iii), (iv) above 
are equivalent to the following conditions respectively,
\begin{enumerate}
\renewcommand{\theenumi}{\roman{enumi}}
\renewcommand{\labelenumi}{\textup{(\theenumi)}}
\item
For each $n \in \Z$, there exits a continuous function
$k_{1,n}:X\longrightarrow \Zp$ such that 
\begin{gather}
( \psi^{k_{1,n}(x) +c_1^n(x)}(h(x)),  \psi^{k_{1,n}(x)}(h(\phi^n(x))))\in G_\psi^{s,0},  
\label{eq:iprime1}\\
( \psi^{-k_{1,n}(x) +c_1^n(x)}(h(x)),  \psi^{-k_{1,n}(x)}(h(\phi^n(x))))\in G_\psi^{u,0}. 
\label{eq:iprime2}
\end{gather}
\item
For each $n \in \Z$, there exits a continuous function
$k_{2,n}:Y\longrightarrow \Zp$ such that 
\begin{gather}
( \phi^{k_{2,n}(y) +c_2^n(y)}(h^{-1}(y)),  
  \phi^{k_{2,n}(y)}(h^{-1}(\psi^n(y))))\in G_\phi^{s,0}, 
\label{eq:iiprime1}\\
( \phi^{-k_{2,n}(y) +c_2^n(y)}(h^{-1}(y)),  
  \phi^{-k_{2,n}(y)}(h^{-1}(\psi^n(y))))\in G_\phi^{u,0}. 
\label{eq:iiprime2}
\end{gather}
\item
There exists a continuous function
$m_1: G_\phi^a\longrightarrow\Zp$ such that 
\begin{gather*}
(\psi^{m_1(x,z) +d_1(x,z)}(h(x)), \psi^{m_1(x,z)}(h(z))) \in G_\psi^{s,0}
\text{ for } (x,z) \in G_\phi^a, \\
(\psi^{-m_1(x,z) +d_1(x,z)}(h(x)), \psi^{-m_1(x,z)}(h(z))) \in G_\psi^{u,0}
\text{ for } (x,z) \in G_\phi^a.
\end{gather*}
\item
There exists a continuous function
$m_2: G_\psi^a\longrightarrow\Zp$ such that 
\begin{gather*}
(\phi^{m_2(y,w) +d_2(y,w)}(h^{-1}(y)), \phi^{m_2(y,w)}(h^{-1}(w))) \in G_\phi^{s,0}
\text{ for } (y,w) \in G_\psi^a, \\
(\phi^{-m_2(y,w) +d_2(y,w)}(h^{-1}(y)), \phi^{-m_2(y,w)}(h^{-1}(w))) \in G_\phi^{u,0}
\text{ for } (y,w) \in G_\psi^a.
\end{gather*}
\end{enumerate}
\end{remark}

In what follows, we will assume that our Smale space is irreducible,
which means that for every ordered pair of open sets $U, V \subset X$,
there exists $K \in \N$ such that 
$\phi^K(U) \cap V \ne \emptyset$.  

\begin{theorem}\label{thm:main1}
Suppose that 
 Smale spaces  $(X,\phi)$ and $(Y,\psi)$ are irreducible.
Then the following assertions are equivalent:
\begin{enumerate}
\renewcommand{\theenumi}{\roman{enumi}}
\renewcommand{\labelenumi}{\textup{(\theenumi)}}
\item
$(X, \phi)$ and $(Y, \psi)$ are asymptotically continuous orbit equivalent.
\item
The groupoids 
$G_{\phi}^{a} \rtimes \Z$ and $G_{\psi}^{a} \rtimes \Z$
are isomorphic as \'etale groupoids.
\end{enumerate}
\end{theorem}
\begin{proof}
(ii) $\Longrightarrow$ (i):
Suppose that 
the groupoids 
$G_{\phi}^{a} \rtimes \Z$ and $G_{\psi}^{a} \rtimes \Z$
are isomorphic as \'etale groupoids.
There exist  homeomorphisms
$h: (G_{\phi}^{a} \rtimes \Z)^\circ \rightarrow  (G_{\psi}^{a} \rtimes \Z)^\circ$ 
and
$\varphi_h: G_{\phi}^{a} \rtimes \Z \rightarrow  G_{\psi}^{a} \rtimes \Z$
which are compatible with their groupoid operations.
Since the unit spaces
$ (G_{\phi}^{a} \rtimes \Z)^\circ$ and $(G_{\psi}^{a} \rtimes \Z)^\circ$
are identified with the original spaces
$X$ and $Y$ as topological spaces
through the identifications
$$
(x,0,x) \in (G_{\phi}^{a} \rtimes \Z)^\circ \longrightarrow x \in X
\quad
\text{ and }
\quad
(y,0,y) \in (G_{\psi}^{a} \rtimes \Z)^\circ \longrightarrow y \in Y,
$$
respectively, 
we have a homeomorphism
from $X$ onto $Y$,
which is still denoted by 
$h: X\rightarrow Y.$
As 
$\varphi_h(x,n,z) \in G_{\psi}^{a} \rtimes \Z$
for $(x,n,z) \in G_{\phi}^{a} \rtimes \Z$,
there exists
$c_1(x,n,z) \in \Z$ such that 
$
\varphi_h(x,n,z) =(h(x), c_1(x,n,z), h(z)). 
$
In particular, we  have
$(x, n, \phi^n(x))\in G_{\phi}^{a} \rtimes \Z$
for $z = \phi^n(x)$,
and may define
$c_{1,n}(x) =c_1(x,n,\phi^n(x))$ so that  
\begin{equation}
(h(x), c_{1,n}(x), h(\phi^n(x)))\in G_{\psi}^{a} \rtimes \Z. \label{eq:c1n}
\end{equation}
Now for $x \in X$ and $n, m\in \Z$, we have
$$
(x,n+m, \phi^{n+m}(x))=
(x,n, \phi^{n}(x))\cdot(\phi^{n}(x), m, \phi^{n+m}(x))
\text{ in }
G_{\phi}^{a} \rtimes \Z
$$
so that 
\begin{align*}
 & (h(x), {c_{1,n+m}(x)}, h(\phi^{n+m}(x))) \\
=& \varphi_h(x, n+m, \phi^{n+m}(x)) \\
=& \varphi_h(x,n, \phi^{n}(x))\varphi_h(\phi^{n}(x), m, \phi^{n+m}(x))\\
=& (h(x),{c_{1,n}(x)}, h(\phi^{n}(x)))
 ( h(\phi^{n}(x)), {c_{1,m}(\phi^{n}(x))},  h(\phi^{n+m}(x)))\\
=& (h(x),{c_{1,n}(x)} +{c_{1,m}(\phi^{n}(x))},  h(\phi^{n+m}(x))).
\end{align*}
Hence we have 
\begin{equation}
c_{1,n+m}(x) ={c_{1,n}(x)} +{c_{1,m}(\phi^{n}(x))}, \label{eq:c1nl}
\end{equation}
so that
the sequence
$\{c_{1,n}\}_{n \in \Z}$ of continuous functions  is a one-cocycle function on $X$.
By putting $c_1(x):=c_{1,1}(x)$, we see that 
$c_1^n(x) = c_{1,n}(x)$
by \eqref{eq:c1nl}.
By \eqref{eq:c1n},
we see that
$(\psi^{c_1^n(x)}(h(x)), h(\phi^n(x))) \in G_\psi^{a}$.
Since the maps below
\begin{align*}
 &((x,x),n) \in  G_{\phi}^a \times \Z \\
\overset{\gamma^{-1}}{\longrightarrow}\hspace{6mm}
& (x,n,\phi^n(x))\in G_{\phi}^{a} \rtimes \Z \\
\overset{\varphi_h}{\longrightarrow}\hspace{6mm}
&(h(x), c_1^n(x), h(\phi^n(x)))\in G_{\psi}^{a} \rtimes \Z \\
\overset{\gamma}{\longrightarrow}\hspace{6mm}
&((h(x), \psi^{-c_1^n(x)}(h(\phi^n(x))),c_1^n(x))  \in G_{\psi}^{a} \times \Z \\
\overset{(\psi^{c_1^n(x)}\times\psi^{c_1^n(x)})\times\id}{\longrightarrow}
&(\psi^{c_1^n(x)}(h(x)), h(\phi^n(x))),c_1^n(x))  \in G_{\psi}^{a} \times \Z 
\label{eq:peta}
\end{align*}
are all continuous,
 the map
$\xi_1^n: x \in X \longrightarrow \xi_1^n(x):=( \psi^{c_1^n(x)}(h(x)),h(\phi^n(x)))
 \in G_\psi^a$ is continuous. 

On the other hand, for $(x,z) \in G_\phi^{a}$
we see that 
$(x,0,z)\in G_{\phi}^{a} \rtimes \Z.
$  
Hence there exists
$d_1(x,z) \in \Z$ such that 
$\varphi_h(x,0,z) = (h(x),d_1(x,z), h(z))$.
Since $\varphi_h:G_{\phi}^{a} \rtimes \Z \longrightarrow G_{\psi}^{a} \rtimes \Z$ 
is continuous,
the function
$d_1: G_\phi^{a}\longrightarrow \Z$
is continuous.
For 
$(x,z), (z,w) \in G_\phi^{a}$,
we have
$(x,0,w) =(x,0,z)(z,0,w) \in G_\phi^{a}$,
and hence
\begin{align*}
 (h(x),d_1(x,w), h(w)) 
= & \varphi_h(x,0,w) \\ 
= & \varphi_h(x,0,z)\cdot\varphi_h(z,0,w) \\ 
= & (h(x),d_1(x,z), h(z))\cdot(h(z),d_1(z,w), h(w)) \\
= & (h(x),d_1(x,z) + d_1(z,w), h(w))
\end{align*}
so that 
$d_1(x,w) =d_1(x,z) + d_1(z,w)$
holds,
and 
$d_1: G_\phi^{a}\longrightarrow \Z$
is a two-cocycle function.
Since the maps below
\begin{align*}
 &((x,z),0) \in  G_{\phi}^a \times \Z \\
\overset{\gamma^{-1}}{\longrightarrow}\hspace{8mm}
& (x,0,z)\in G_{\phi}^{a} \rtimes \Z \\
\overset{\varphi_h}{\longrightarrow}\hspace{8mm}
&(h(x), d_1(x,z), h(z))\in G_{\psi}^{a} \rtimes \Z \\
\overset{\gamma}{\longrightarrow}\hspace{8mm}
&((h(x), \psi^{-d_1(x,z)}(h(z)),d_1(x,z))  \in G_{\psi}^{a} \times \Z \\
\overset{(\psi^{d_1(x,z)}\times\psi^{d_1(x,z)})\times\id}{\longrightarrow}
&(\psi^{d_1(x,z)}(h(x)), h(z)),d_1(x,z))  \in G_{\psi}^{a} \times \Z
\end{align*}
are all continuous,
 the map
$\eta_1: (x,z) \in G_\phi^a
 \longrightarrow \eta_1(x,z):=(\psi^{d_1(x,z)}(h(x)), h(z))
 \in G_\psi^a$ is continuous.

For $(x, n, x'), (x', m, z) \in G_\phi^a\rtimes\Z,$
the identity
$\varphi_h((x, n, x') \cdot (x', m, z) ) =  
\varphi_h(x, n, x') \cdot \varphi_h(x', m, z)$
is equivalent to the identity
\begin{equation*}
c_1^m(\phi^n(x)) + d_1(\phi^{m+n}(x), z)
=c_1^m(x') + d_1(\phi^{n}(x), x') + d_1(\phi^{m}(x'), z),
\end{equation*}
that implies the identity
$$c_1^m(x) + d_1(\phi^{m}(x), \phi^m(z))
=c_1^m(z) + d_1(x, z), \quad
(x, z) \in G_\phi^a, \, m \in \Z.
$$

We similarly have one-cocycle function
$c_2:Y \rightarrow\Z$
and two-cocycle function
$d_2: G_\psi^{a}\longrightarrow \Z$
for the homeomorphism
$\varphi_h^{-1}: G_{\psi}^{a} \rtimes \Z\longrightarrow G_{\phi}^{a} \rtimes \Z$.
Since 
$h^{-1} = (\varphi_h)^{-1}|_{(G_\psi^a\rtimes\Z)^\circ}:
{(G_\psi^a\rtimes\Z)^\circ}=Y \longrightarrow 
{(G_\phi^a\rtimes\Z)^\circ} =X,$
we see that
$\varphi_h^{-1} =\varphi_{h^{-1}}.
 $
By the identity 
\begin{equation*}
(x,n,\phi^n(x)) = (\varphi_h^{-1}\circ\varphi_h)(x,n,\phi^n(x)) 
\qquad
\text{for } x \in X, \, n \in \Z,
\end{equation*}
we have
\begin{align*}
  & (\varphi_h^{-1}\circ\varphi_h)(x,n,\phi^n(x)) \\
= & \varphi_h^{-1}(h(x),c_1^n(x), h(\phi^n(x))) \\
= & \varphi_h^{-1}(h(x),c_1^n(x), \psi^{c_1^n(x)}(h(x)))\cdot 
   \varphi_h^{-1}(\psi^{c_1^n(x)}(h(x)), 0, h(\phi^n(x))) \\
= & (x,c_2^{c_1^n(x)}(h(x)), h^{-1}(\psi^{c_1^n(x)}(h(x)))\cdot 
   (h^{-1}(\psi^{c_1^n(x)}(h(x))), 
   d_2(\psi^{c_1^n(x)}(h(x)), h(\phi^n(x))), \phi^n(x)) \\
= &(x,c_2^{c_1^n(x)}(h(x))+d_2(\psi^{c_1^n(x)}(h(x)), h(\phi^n(x))), \phi^n(x))
\end{align*}
so that
the identity
\begin{equation*}
c_2^{c_1^n(x)}(h(x))+d_2(\psi^{c_1^n(x)}(h(x)), h(\phi^n(x)))
=n
\end{equation*}
holds, and similarly
\begin{equation*}
c^{c^n_2(y)}_1(h^{-1}(y)) + d_1(\phi^{c_2^n(y)}(h^{-1}(y)), h^{-1}(\psi^n(y))) = n, 
\quad y \in Y, \, n \in \Z.
\end{equation*}

For 
$(x,z) \in G_\phi^{a}$,
the identity
$(x,0,z) = (\varphi_h^{-1}\circ \varphi)(x, 0, z)
$ 
holds so that
we have
\begin{align*}
  & (\varphi_h^{-1}\circ\varphi_h)(x, 0, z) \\
= & \varphi_h^{-1}(h(x), d_1(x,z), h(z)) \\
= & \varphi_h^{-1}(h(x), d_1(x,z), \psi^{d_1(x,z)}(h(x)))\cdot 
   \varphi_h^{-1}(\psi^{d_1(x,z)}(h(x)), 0, h(z)) \\
= & (x, c_2^{d_1(x,z)}(h(x)), h^{-1}(\psi^{d_1(x,z)}(h(x)))\cdot 
   (h^{-1}(\psi^{d_1(x,z)}(h(x))), 
   d_2(\psi^{d_1(x,z)}(h(x)), h(z)), z) \\
= &(x,  c_2^{d_1(x,z)}(h(x))+d_2(\psi^{d_1(x,z)}(h(x)), h(z)), z).
\end{align*}
Hence we have 
\begin{equation*}
c^{d_1(x,z)}_2(h(x)) + d_2(\psi^{d_1(x,z)}(h(x)), h(z)) = 0, \qquad (x,z) \in G_\phi^{a},
\end{equation*}
and similarly 
\begin{equation*}
c^{d_2(y,w)}_1(h^{-1}(y)) + d_1(\phi^{d_2(y,w)}(h^{-1}(y)), h^{-1}(w)) = 0, 
\qquad (y,w) \in G_\psi^{a}.
\end{equation*}
Therefore we see that
$(X, \phi) \underset{ACOE}{\sim}(Y, \psi).$

(i) $\Longrightarrow$ (ii):
Assume that
$(X, \phi) \underset{ACOE}{\sim}(Y, \psi)$
and take a homeomorphism
$h: X\longrightarrow Y$,
continuous functions
$
c_1:X\longrightarrow \Z, \,
c_2:Y\longrightarrow \Z, 
$
and two-cocycle functions
$
d_1: G_\phi^a \longrightarrow \Z,\,
d_2: G_\psi^a \longrightarrow \Z
$
 as in Definition \ref{def:acoe}.
Let $(x, n, z) \in G_{\phi}^{a} \rtimes \Z$ be  an arbitrary element, 
 so that
 $(\phi^n(x), z) \in G_\phi^a$ and we have
\begin{equation*}
(x, n, z)  = (x, n, \phi^n(x)) \cdot (\phi^n(x), 0, z). 
\end{equation*}
Put
\begin{equation}
\varphi_h(x,n,z) =
(h(x), c_1^n(x), h(\phi^n(x))\cdot
(h(\phi^n(x)), d_1(\phi^n(x), z), h(z)).\label{eq:3.19}
\end{equation}
By Definition \ref{def:acoe} (i), 
$( \psi^{c_1^n(x)}(h(x)),  h(\phi^n(x)))$ belongs to $G_\psi^a$ 
 so that $(h(x), c_1^n(x), h(\phi^n(x)))$ gives an element of
 $G_\psi^a\rtimes\Z$.
As $(\phi^n(x), z) \in G_\phi^a$,
we see that by Definition \ref{def:acoe} (iii),
$(\psi^{d_1(\phi^n(x), z)}(h(\phi^n(x))),h(z))$
belongs to $G_\psi^a$
so that 
$(h(\phi^n(x)), d_1(\phi^n(x), z), h(z))
$
gives an element of $G_\psi^a\rtimes\Z$.
 Hence 
$ \varphi_h(x,n,z)
$
defines an element of the groupoid
$G_\psi^a\rtimes\Z$ such that 
\begin{equation*}
\varphi_h(x,n,z) =
(h(x), c_1^n(x) + d_1(\phi^n(x), z), h(z)).
\end{equation*}
It is straightforward to see that the equality (1) in  
Definition \ref{def:acoe} implies 
$$
\varphi_h((x, n, x') \cdot (x', m, z) ) =  
\varphi_h(x, n, x') \cdot \varphi_h(x', m, z)
$$
for 
$(x, n, x'), (x',m,z) \in G_\phi^a\rtimes\Z.$

Since 
$x \in X \longrightarrow 
\xi_1^n(x)= ( \psi^{c_1^n(x)}(h(x)),  h(\phi^n(x)))
 \in G_\psi^a$
 is continuous by Definition \ref{def:acoe} (i) and
 \begin{align*}
& \gamma^{-1}\circ {((\psi^{-c_1^n(x)}\times\psi^{-c_1^n(x)})}\times \id)
 (\xi_1^n(x), c_1^n(x)) \\
=&
\gamma^{-1}(h(x), \psi^{-c_1^n(x)}(h(\phi^n(x)), c_1^n(x)) \\
= &(h(x), c_1^n(x), h(\phi^n(x))),
\end{align*}
 the map $\varphi_h^1: G_\phi^a\rtimes\Z\longrightarrow G_\psi^a\rtimes\Z$
 defined by
 \begin{equation*}
\varphi_h^1(x,n,z) := (h(x), c_1^n(x), h(\phi^n(x)))
\end{equation*}
 is continuous.

 And also the map
$\eta_1:(x,z) \in G_\phi^a\longrightarrow
\eta_1(x,z) = (\psi^{d_1(x,z)}(h(x)), h(z))\in G_\psi^a$ is continuous
by Definition \ref{def:acoe} (iii) and
\begin{align*}
& \gamma^{-1}{((\psi^{-d_1(\phi^n(x),z)}\times\psi^{-d_1(\phi^n(x),z)})\times\id)}
(\eta_1(\phi^n(x),z),d_1(\phi^n(x),z)) \\
=&\gamma^{-1}(h(\phi^n(x)), \psi^{-d_1(\phi^n(x),z)},d_1(\phi^n(x),z)) \\
=&(h(\phi^n(x)), d_1(\phi^n(x),z), h(z)).
\end{align*}
 Hence
 the map $\varphi_h^0: G_\phi^a\rtimes\Z\longrightarrow G_\psi^a\rtimes\Z$
 defined by
\begin{equation*}
\varphi_h^0(x,n,z) := (h(\phi^n(x)), d_1(\phi^n(x),z), h(z))
\end{equation*}
 is continuous.
Since
$\varphi_h(x,n,z) = \varphi_h^1(x,n,z)\varphi_h^0(x,n,z)$
by \eqref{eq:3.19},
the map
$\varphi_h: 
G_\phi^a\rtimes\Z\longrightarrow G_\psi^a\rtimes\Z$
is continuous.

Similarly we may define a continuous map
$\varphi_{h^{-1}}: 
G_\psi^a\rtimes\Z\longrightarrow G_\phi^a\rtimes\Z$
from the homeomorphism
$h^{-1}: Y\longrightarrow X$ 
and one-cocycle function
$c_2:Y\longrightarrow \Z$,
two-cocycle function
$d_2:G_\psi^a\longrightarrow \Z$
by setting
\begin{equation*} 
 \varphi_{h^{-1}}(y,m,w)
= (h^{-1}(y),c_2^m(y)+ d_2(\psi^m(y), w), h^{-1}(w))
\qquad
\text{for }(y,m,w)\in G_\psi^a\rtimes\Z.
\end{equation*}
We put
for $(y,m,w)\in G_\psi^a\rtimes\Z$
\begin{align*} 
 (\varphi_{h^{-1}})^1(y,m,w)
= &(h^{-1}(y),c_2^m(y),h^{-1}(\psi^m(y))), \\
 (\varphi_{h^{-1}})^0(y,m,w)
= &   (h^{-1}(\psi^m(y)), d_2(\psi^m(y), w), h^{-1}(w))
\end{align*}
so that 
\begin{equation*} 
 \varphi_{h^{-1}}(y,m,w)
= (\varphi_{h^{-1}})^1(y,m,w)\cdot(\varphi_{h^{-1}})^0(y,m,w)
\qquad
\text{for }(y,m,w)\in G_\psi^a\rtimes\Z.
\end{equation*}
We will next show that 
$\varphi_h $ and $\varphi_{h^{-1}}$ are inverses to each other.
For $x\in X, n \in \Z$, we have
\begin{align*}
  &(\varphi_{h^{-1}}\circ \varphi_h)(x,n, \phi^n(x)) \\
= &\varphi_{h^{-1}}(h(x), c_1^n(x), h(\phi^n(x)))\\
= &\varphi_{h^{-1}}(h(x), c_1^n(x), \psi^{c_1^n(x)}(h(x)))\cdot
           (\psi^{c_1^n(x)}(h(x)), 0, h(\phi^n(x)))   \\
= & (x, c_2^{c_1^n(x)}(h(x)), h^{-1}(\psi^{c_1^n(x)}(h(x))))\cdot
    (h^{-1}(\psi^{c_1^n(x)}(h(x))), d_2(\psi^{c_1^n(x)}(h(x)),h(\phi^n(x))), \phi^n(x)) \\
= &(x, c_2^{c_1^n(x)}(h(x))+ d_2(\psi^{c_1^n(x)}(h(x)),h(\phi^n(x))), \phi^n(x)). 
\end{align*}           
By the condition of Definition \ref{def:acoe} (v),
we have            
\begin{equation*}
(\varphi_{h^{-1}}\circ \varphi_h)(x,n, \phi^n(x))
=(x,n,\phi^n(x)).
\end{equation*}
We also have
for $(x,z) \in G_\phi^a$,
\begin{align*}
 & \varphi_{h^{-1}}\circ \varphi_h(x,0, z) \\
=& \varphi_{h^{-1}}(h(x), d_1(x,z), h(z))  \\
=& (x, c_2^{d_1(x,z)}(h(x)), h^{-1}(\psi^{d_1(x,z)}(h(x))))
    \cdot (h^{-1}(\psi^{d_1(x,z)}(h(x))), d_2(\psi^{d_1(x,z)}(h(x)), h(z)), z)  \\
=& (x, c_2^{d_1(x,z)}(h(x)) + d_2(\psi^{d_1(x,z)}(h(x)), h(z)), z)
\end{align*}           
By the condition of Definition \ref{def:acoe} (vii),
we have            
\begin{equation*}
(\varphi_{h^{-1}}\circ \varphi_h)(x,0, z)
=(x,0,z).
\end{equation*}
Therefore we have for $(x,n,z) \in G_\phi^a\rtimes\Z$
\begin{align*}
  &(\varphi_{h^{-1}}\circ \varphi_h)(x,n, z) \\
= &((\varphi_{h^{-1}}\circ \varphi_h)(x,n,\phi^n(x)))\cdot
   ((\varphi_{h^{-1}}\circ \varphi_h)(\phi^n(x), 0, z))\\
= &(x, n, \phi^n(x))\cdot (\phi^n(x), 0,z) \\ 
= & (x,n,z).
\end{align*}           
Similarly we have
$(\varphi_h\circ\varphi_{h^{-1}})(y,m,w) = (y,m,w)
$
for $(y,m,w) \in G_\psi^a\rtimes\Z$.
Hence we have 
$\varphi_{h^{-1}} = (\varphi_h)^{-1}
$
and $\varphi_h$ gives rise to an isomorphism 
$ G_\phi^a\rtimes\Z\longrightarrow G_\psi^a\rtimes\Z$ 
of  the \'etale groupoids. 
\end{proof}
Smale spaces
$(X,\phi)$ and $(Y, \psi)$ are said to be
{\it stably  continuous orbit equivalent}\/
if in Definition \ref{def:acoe}, we may replace $G_\phi^a, G_\psi^a$
with $G_\phi^s, G_\psi^s$, respectively,
and written
$(X, \phi) \underset{SCOE}{\sim}(Y, \psi).$
Unstably continuous orbit equivalent is similarly defined 
by replacing $G_\phi^a, G_\psi^a$
with $G_\phi^u, G_\psi^u$, respectively,
and written
$(X, \phi) \underset{UCOE}{\sim}(Y, \psi).$
Precise definition of 
{\it stably continuous orbit equivalet}\/
is the following:
\begin{definition}\label{def:scoe}
Smale spaces
$(X,\phi)$ and $(Y, \psi)$ are said to be
{\it stably continuous orbit equivalent}\/
if there exist a homeomorphism
$h: X\longrightarrow Y$,
continuous functions
$
c_1:X\longrightarrow \Z, \,
c_2:Y\longrightarrow \Z, 
$
and two-cocycle functions
$
d_1: G_\phi^s \longrightarrow \Z,\,
d_2: G_\psi^s \longrightarrow \Z
$
such that 
\begin{enumerate}
\renewcommand{\theenumi}{\arabic{enumi}}
\renewcommand{\labelenumi}{\textup{(\theenumi)}}
\item
$c_1^m(x) + d_1(\phi^{m}(x), \phi^m(z))
=c_1^m(z) + d_1(x, z), \quad
(x,z) \in G_\phi^s, \, m \in \Z.$
\item
$c_2^m(y) + d_2(\psi^{m}(y), \psi^m(w))
=c_2^m(w) + d_2(y, w),\quad
(y, w) \in G_\psi^s,  m \in \Z.$
\end{enumerate}
and,
\begin{enumerate}
\renewcommand{\theenumi}{\roman{enumi}}
\renewcommand{\labelenumi}{\textup{(\theenumi)}}
\item
For each $n \in \Z$, there exits a continuous function
$k_{1,n}:X\longrightarrow \Zp$ such that 
\begin{equation*}
( \psi^{k_{1,n}(x) +c_1^n(x)}(h(x)),  \psi^{k_{1,n}(x)}(h(\phi^n(x))))\in G_\psi^{s,0}.  
\end{equation*}
\item
For each $n \in \Z$, there exits a continuous function
$k_{2,n}:Y\longrightarrow \Zp$ such that 
\begin{equation*}
( \phi^{k_{2,n}(y) +c_2^n(y)}(h^{-1}(y)),  
  \phi^{k_{2,n}(y)}(h^{-1}(\phi^n(y))))\in G_\phi^{s,0}.
\end{equation*}
\item
There exists a continuous function
$m_1: G_\phi^s\longrightarrow\Zp$ such that 
\begin{equation*}
(\psi^{m_1(x,z) +d_1(x,z)}(h(x)), \psi^{m_1(x,z)}(h(z))) \in G_\psi^{s,0}
\text{ for } (x,z) \in G_\phi^s.
\end{equation*}
\item
There exists a continuous function
$m_2: G_\psi^s\longrightarrow\Zp$ such that 
\begin{equation*}
(\phi^{m_2(y,w) +d_2(y,w)}(h^{-1}(y)), \phi^{m_2(y,w)}(h^{-1}(w))) \in G_\phi^{s,0}
\text{ for } (y,w) \in G_\psi^s. 
\end{equation*}
\item
$c^{c^n_1(x)}_2(h(x)) + d_2(\psi^{c_1^n(x)}(h(x)), h(\phi^n(x))) = n, 
\quad x \in X,\, n \in \Z.$
\item
$c^{c^n_2(y)}_1(h^{-1}(y)) + d_1(\phi^{c_2^n(y)}(h^{-1}(y)), h^{-1}(\psi^n(y))) = n, 
\qquad y \in Y,\, n \in \Z.$
\item
$c^{d_1(x,z)}_2(h(x)) + d_2(\psi^{d_1(x,z)}(h(x)), h(z)) = 0, 
\qquad (x,z) \in G_\phi^{s}.$
\item
$c^{d_2(y,w)}_1(h^{-1}(y)) + d_1(\phi^{d_2(y,w)}(h^{-1}(y)), h^{-1}(w)) = 0, 
\qquad (y,w) \in G_\psi^{s}.$
\end{enumerate}
\end{definition}
If we replace 
$G_\phi^{s,0}, G_\psi^{s,0},
G_\phi^{s}, G_\psi^{s}
$ with 
$G_\phi^{u,0}, G_\psi^{u,0},
G_\phi^{u}, G_\psi^{u},$ respectively, 
then 
$(X,\phi)$ and $(Y,\psi)$ are said to be 
{\it unstably continuous orbit equivalent}.

We may show the following theorem in a similar fashion to Theorem \ref{thm:main1}. 
\begin{theorem}
Suppose that 
the Smale spaces  $(X,\phi)$ and $(Y,\psi)$ are irreducible.
Then the following conditions are equivalent:
\begin{enumerate}
\renewcommand{\theenumi}{\roman{enumi}}
\renewcommand{\labelenumi}{\textup{(\theenumi)}}
\item
$(X, \phi) \underset{SCOE}{\sim}(Y, \psi)$
(resp.  $(X, \phi) \underset{UCOE}{\sim}(Y, \psi)$)
\item
The groupoids 
$G_{\phi}^{s} \rtimes \Z$ and $G_{\psi}^{s} \rtimes \Z$
(resp. $G_{\phi}^{u} \rtimes \Z$ and $G_{\psi}^{u} \rtimes \Z$)
are isomorphic as topological groupoids.
\end{enumerate}
\end{theorem}
We note that the groupoids $G_\phi^s, G_\psi^s, G_\phi^u, G_\psi^u$ above 
are the non-\'etale groupoids appearing in Lemma \ref{lem:n2.6} 
which had been defined in \cite{Putnam1}.    
We do not know whether or not the corresponding theorem holds for 
\'etale groupoids defined from $\phi$-invariant set of stable or unstable equivalence relations appearing in \cite{PutSp}. 

\section{Asymptotic periodic orbits of Smale spaces}
Let $(X,\phi)$ be an irreducible Smale space.
\begin{definition}
An element $x \in X$ is called  {\it a stably asymptotic periodic point}\/
if there exists $p\in \Z$ with $p \ne 0$ such that 
$(x,\phi^p(x) )\in G_\phi^s$.
We call such $p$ asymptotic period of $x$.
If $|p|$ is the least positive such number, 
it is said to be the least asymptotic period of $x$. 
\end{definition}
We note that the asymptotic period is possibly negative,  
and hence if $p$ is the least asymptotic period, then $-p$ is also
the least asymptotic period.  

Throughout this section,
we assume that
$(X, \phi) \underset{SCOE}{\sim}(Y, \psi)$
and keep 
a homeomorphism
$h:X\longrightarrow Y$,
continuous functions
$c_1, c_2$ and
two-cocycle functions
$d_1, d_2$
which give rise to 
stably asymptotic continuous orbit equivalence
between 
$(X,\phi)$ and $(Y,\psi)$ as in Definition \ref{def:scoe}.
\begin{lemma}
If $x\in X$ is a stably asymptotic periodic point with asymptotic period $p$,
then  $h(x)$ is also a stably asymptotic periodic point with asymptotic period 
$c_1^p(x) + d_1(\phi^p(x), x)$. 
\end{lemma}
\begin{proof}
Since
$(x,\phi^p(x) )\in G_\phi^s$
and hence
 $(x,p,x) \in G_\phi^s\rtimes\Z$,
 we have
\begin{equation*}
\varphi_h(x,p,x) =(h(x), c_1^p(x) +d_1(\phi^p(x),x ), h(x)).
\end{equation*}
As
$(X, \phi) \underset{SCOE}{\sim}(Y, \psi),$
$h(x)$ is  a stably asymptotic periodic point in $Y$ with asymptotic period 
$c_1^p(x) + d_1(\phi^p(x), x)$.
\end{proof}
\begin{lemma}
Let $x\in X$ be  a stably asymptotic periodic point with least asymptotic period $p$.
Let $p'$ be the least asymptotic period of $h(x)$.
Then we have the equality
\begin{equation}
c_2^{np'}(h(x)) + d_2(\psi^{np'}(h(x)), h(x))
=n \cdot (c_2^{p'}(h(x)) + d_2(\psi^{p'}(h(x)), h(x))) \label{eq:aperiod}
\end{equation}
for all
$ n \in \Z.$
 \end{lemma}
\begin{proof}
Suppose that 
$(x,\phi^p(x) )\in G_\phi^s$.
Put $y =h(x)$ and 
$q' =c_1^p(x) + d_1(\phi^p(x), x).$
By the preceding lemma, we know that 
$y $ has asymptotic period $q'$,
so that
$(y,p',y) \in G_\psi^s\rtimes\Z$.
Now suppose that the equality \eqref{eq:aperiod} holds for $n=k$. 
Since 
$(y,p',y)(y,kp',y) = (y,(k+1)p',y)$, we get
\begin{equation}
\varphi_{h^{-1}}((y,p',y)(y,kp',y)) = \varphi_{h^{-1}}(y,(k+1)p',y). \label{eq:ypy}
\end{equation}
The left hand side of  \eqref{eq:ypy} equals
\begin{align*}
  & \varphi_{h^{-1}}(y,p',y)\varphi_{h^{-1}}(y,kp',y) \\
 =& (h^{-1}(y), c_2^{p'}(y) + d_2(\psi^{p'}(y),y), h^{-1}(y))\cdot    
      (h^{-1}(y), c_2^{kp'}(y) + d_2(\psi^{kp'}(y),y), h^{-1}(y)) \\
 =& (h^{-1}(y), 
     c_2^{p'}(y) + d_2(\psi^{p'}(y),y) + c_2^{kp'}(y) + d_2(\psi^{kp'}(y),y), h^{-1}(y)).
\end{align*}
The right hand side of  \eqref{eq:ypy} equals
\begin{equation*}
   \varphi_{h^{-1}}(y,(k+1)p',y) 
 = (h^{-1}(y), c_2^{(k+1)p'}(y) + d_2(\psi^{(k+1)p'}(y),y), h^{-1}(y))
\end{equation*}
so that we have
\begin{equation*}
  c_2^{p'}(y) + d_2(\psi^{p'}(y),y) + c_2^{kp'}(y) + d_2(\psi^{kp'}(y),y) 
=  c_2^{(k+1)p'}(y) + d_2(\psi^{(k+1)p'}(y),y).
\end{equation*}
By induction, we obtain the desired equalities for all $n \in \N$, and hence for all 
$n \in \Z$ in a similar way.
\end{proof}
\begin{lemma}\label{lem:4.5}
If $x\in X$ is a stably asymptotic periodic point with asymptotic period $p$,
then  $h(x)$ is also a stably asymptotic periodic point with asymptotic period 
$c_1^p(x) + d_1(\phi^p(x), x)$. 
If in particular, $p$ is the least asymptotic period of $x$,
then 
$c_1^p(x) + d_1(\phi^p(x), x)$
is the least asymptotic period of $h(x).$
\end{lemma}
\begin{proof}
It suffices to show the ``If in particular'' part.
Suppose that 
$(x,\phi^p(x) )\in G_\phi^s$
and
$p$ is the least asymptotic period of $x$.
We will show that 
$c_1^p(x) + d_1(\phi^p(x), x)$
is the least asymptotic period of $h(x).$
Let $p'$ be the the least asymptotic period of $h(x).$
Put
$q' =c_1^p(x) + d_1(\phi^p(x), x),$
so that $q' = n \cdot p'$ for some $n \in \Z$.
We will prove that $n =\pm 1$.
We have
\begin{align*}
(x, p, x) 
= & (\varphi_{h^{-1}}\circ\varphi_h)(x,p,x) \\
= & \varphi_{h^{-1}}(h(x), q', h(x)) \\
= & (x, c_2^{q'}(h(x)) + d_2(\psi^{q'}(h(x)), h(x)), x). 
\end{align*}
Hence 
$p =c_2^{q'}(h(x)) + d_2(\psi^{q'}(h(x)), h(x)).$
As $q' = n p'$, the preceding lemma tells us that
\begin{equation}
p= n\cdot(c_2^{p'}(h(x)) + d_2(\psi^{p'}(h(x)), h(x))). \label{eq:pnpprime}
\end{equation}
Since $p'$ is (the least) asymptotic period of $h(x)$,
we have $( \psi^{p'}(h(x)),h(x)) \in G_\psi^s$,
so that by 
Definition \ref{def:scoe} (iv), we have
\begin{equation}
(\phi^{d_2( \psi^{p'}(h(x)),h(x))}(h^{-1}(\psi^{p'}(h(x))), h^{-1}(h(x))) \in G_\phi^s.
\label{eq:phid2psip'1}
\end{equation}
By Definition \ref{def:scoe} (ii), we have
$
(\phi^{c_2^{p'}(h(x))}(x), h^{-1}(\psi^{p'}(h(x))) ) \in G_\phi^s
$
and hence
\begin{equation}
(\phi^{c_2^{p'}(h(x)) + d_2(\psi^{p'}(h(x)),h(x))}(x), 
\phi^{d_2(\psi^{p'}(h(x)),h(x))}(h^{-1}(\psi^{p'}(h(x)))) ) \in G_\phi^s.
\label{eq:phid2psip'2}
\end{equation}
By \eqref{eq:phid2psip'1} and \eqref{eq:phid2psip'2},
we have
\begin{equation}
(\phi^{c_2^{p'}(h(x)) + d_2(\psi^{p'}(h(x)),h(x))}(x), x) \in G_\phi^s.
\label{eq:phid2psip'3}
\end{equation}
As $p$ is the least asymptotic period of $x$, we have
\begin{equation}
c_2^{p'}(h(x)) + d_2(\psi^{p'}(h(x)), h(x)) = p\cdot m' \quad
\text{ for some }m' \in \Z.
\label{eq:phid2psip'4}
\end{equation}
By \eqref{eq:pnpprime} and \eqref{eq:phid2psip'4}, we have
\begin{equation*}
p = n\cdot m'\cdot p
\end{equation*}
 so that we conclude that $n=\pm 1$
 and hence
$c_1^p(x) + d_1(\phi^p(x), x)$
is the least asymptotic period of $h(x).$
\end{proof}
For a stably asymptotic periodic point $x \in X$ with asymptotic period $p$,
We put
\begin{equation*}
c_h^p(x) := c_1^p(x) + d_1(\phi^p(x), x) \in \Z.
\end{equation*}
If $p$ is the least asymptotic period, 
the preceding proposition tells us that 
\begin{equation*}
c_h^{np}(x) =n \cdot c_h^p(x) \quad \text{ for } n \in \Z.
\end{equation*}
In this case, as any asymptotic period $q$ of $x$ is written $q = m\cdot p$
for some $m \in \Z$ with $m \ne 0$, we have
$c_h^{nmp}(x) =nm \cdot c_h^p(x)=n \cdot c_h^{mp}(x)$ 
so that 
$c_h^{nq}(x) =n \cdot c_h^q(x).$

For a periodic point $x \in X$,
the finite set 
$\{ \phi^n(x) \mid n \in \Z\}$
is called a periodic orbit.
Let us denote by 
\begin{equation*}
\Porb(X) := \text{the set of periodic orbits of }(X,\phi).
\end{equation*}
A periodic point with period $p$ is called a $p$-periodic point. 
Let $\Per_p(X)$ be the set of $p$-periodic points of $(X,\phi)$.
The following theorem due to R. Bowen 
tells us that  
the set $\Per_p(X)$ is finite for each $p$
because so is $\Per_p(\bar{X}_A)$.
\begin{theorem}[{\cite[Theorem 3.12]{Bowen2}}] \label{thm:Bowen}
Let $(X,\phi)$ be an irreducible Smale space.
Then there exists an irreducible subshift of finite type 
$(\bar{X}_A, \bar{\sigma}_A)$ such that there exists a finite-to-one factor map
$\varphi: (\bar{X}_A, \bar{\sigma}_A) \longrightarrow (X,\phi).$
\end{theorem}
For a periodic orbit $\gamma \in \Porb(X)$, 
take a periodic point  $x \in X$ 
such that
$\gamma =\{ \phi^n(x)\mid n \in \Z\}$.
The cardinality of the set $\{ \phi^n(x)\mid n \in \Z\}$ is
called the length of $\gamma$ and written
$|\gamma|$.
We will show that the periodic orbits 
$\Porb(X)$ and $\Porb(Y)$
are related by their cocycle functions under the condition 
$(X, \phi) \underset{SCOE}{\sim}(Y, \psi).$
We provide the following lemma 
\begin{lemma}\label{lem:periodicpoints}
Suppose that
$(X, \phi) \underset{SCOE}{\sim}(Y, \psi).$
Let $x \in X$ be a periodic point in $X$ such that $\phi^p(x) = x$.
Put $q = |c_1^p(x)|.$ Then we have
\begin{enumerate}
\renewcommand{\theenumi}{\roman{enumi}}
\renewcommand{\labelenumi}{\textup{(\theenumi)}}
\item $ c_1^{kp}(x) = k c_1^p(x) $ for $k \in \Z.$
\item $\psi^q(h(x)) \in Y^s(h(x))$ so that the limit
$\lim_{k\to{\infty}}\psi^{q k}(h(x))$ exists in $Y.$
\item Put 
$\eta_h(x) =\lim_{k\to{\infty}}\psi^{q k}(h(x)).$
Then 
\begin{equation}
\eta_h(\phi^n(x)) = \psi^{c_1^n(x)}(\eta_h(x)) \qquad \text{for } \quad n \in \Z. 
\label{eq:etah}
\end{equation}
In particular, $\eta_h(x)$ is a $q$-periodic point in $Y$.
\item $\eta_h(x) \in Y^s(h(x)).$
\item If $p$ is the least period of $x$, then $c_1^p(x)$ is the least period of $\eta_h(x)$.
\item   $c_2^{c_1^p(x)}(\eta_h(x)) = p$.
\end{enumerate}
\end{lemma}
\begin{proof}
(i) Since $\phi^p(x) = x$, we have 
$d_1(\phi^p(x), x) =0$ so that 
$c_h^p(x) = c_1^p(x) + d_1(\phi^p(x),x) = c_1^p(x)$.
Hence the equality 
 $c_1^p(x)\cdot k = c_1^{kp}(x)$  for $k \in \Z$ 
is immediate.

(ii) 
We have 
$(\psi^{q}(h(x)), h(x))=
(\psi^{c_1^p(x)}(h(x)), h(\phi^p (x)))$
which belongs to $G_\psi^s$ 
because of Definition \ref{def:scoe} (i),
so that 
$ \psi^{q}(h(x)) \in Y^s(h(x)).$
By using \cite[Lemma 5.3]{PutSp},  
the element
$\lim_{k\to{\infty}}\psi^{q k}(h(x))$
exists in $Y$ and is a periodic point with period 
$q.$

(iii)
By Definition \ref{def:scoe} (i) with Lemma \ref{lem:n2.6} ,
we have 
\begin{equation*}
\lim_{k\to{\infty}}\psi^{q  k}(h(\phi^n(x)))
=
\lim_{k\to{\infty}}\psi^{q k}( \psi^{c_1^n(x)}(h(x)))
=
\psi^{c_1^n(x)}(\lim_{k\to{\infty}}\psi^{q k}( h(x)))
\end{equation*}
so that 
the equality \eqref{eq:etah} holds.

(iv)
For each $n \in \Z, $ we have $qn = |c_1^p(x)| n$ by (i)
so that the equality  
\begin{equation*}
  \lim_{k\to{\infty}}\psi^{qk}(\psi^{qn}(h(x))) 
= \lim_{k\to{\infty}}\psi^{qk}(h(\phi^{pn}(x)))
= \lim_{k\to{\infty}}\psi^{qk}(h(x))
\end{equation*}
holds because of  Definition \ref{def:scoe} (i).
It then follows that
\begin{align*}
\lim_{n\to{\infty}}\psi^{qn}(\eta_h(x))
= & \lim_{n\to{\infty}}\psi^{qn}(\lim_{k\to{\infty}}\psi^{qk}(h(x))) \\
= & \lim_{n\to{\infty}}(\lim_{k\to{\infty}}\psi^{qn+ qk}(h(x))) \\
= & \lim_{n\to{\infty}}\left(\lim_{k\to{\infty}} \psi^{qk}(\psi^{qn}(h(x)))\right) \\
= & \lim_{n\to{\infty}}\left(\lim_{k\to{\infty}} \psi^{qk}(h(x))\right) \\
=&\lim_{k\to{\infty}} \psi^{qk}(h(x)) = \eta_h(x) 
\end{align*}
and also
for $j=1,\dots, q-1$,
\begin{align*}
\lim_{n\to{\infty}}\psi^{qn+j}(\eta_h(x))
= &\psi^j\left(\lim_{n\to{\infty}}\psi^{qn}(\eta_h(x)) \right) \\
= &\psi^j\left(\lim_{n\to{\infty}}\psi^{qn}(h(x)) \right) \\
=&\lim_{n\to{\infty}} \psi^{qn+j}(h(x)).
\end{align*}
Hence we have 
 \begin{equation*}
\lim_{n\to{\infty}} d(\psi^n(\eta_h(x)), \psi^n(h(x))) =0
\end{equation*}
where the above $d(\cdot, \cdot)$ is the metric on $Y$, 
so that we obtain that 
$\eta_h(x) \in Y^s(h(x)).$


(v)
Assume that $p$ is the least period of $x$.
We will show that $q=|c^p_1(x)|$ is
the least period of $\eta_h(x).$
Let $q_0$ be the least period of $\eta_h(x),$
such that $q = q_0\cdot m$ for some $m \in \N$ and 
$\psi^{q_0}(\eta_h(x)) =\eta_h(x).$
Hence we have
\begin{equation*}
\lim_{k\to{\infty}}\psi^{q k +j}(\psi^{q_0}(h(x))) 
=\lim_{k\to{\infty}}\psi^{q k+j}(h(x)), 
\quad j=0,1,\dots, q -1
\end{equation*}
so that 
\begin{equation*}
\lim_{n\to{\infty}}\psi^{n}(\psi^{q_0}(h(x))) 
=\lim_{n\to{\infty}}\psi^{n}(h(x)). \label{eq:psinq}
\end{equation*}
By  Lemma \ref{lem:n2.6},
we have 
$(\psi^{q_0}(h(x)), h(x)) \in G_\psi^s$ 
and hence $q_0$ is an asymptotic period of $h(x)$ .
As $q$ is the least asymptotic period of $h(x)$ by Lemma \ref{lem:4.5}
and $q = {q_0}\cdot m,$ we get $m=1$, 
that is, 
 $q$ is the least period of $\eta_h(x).$

(vi)
We will prove $c_2^q(\eta_h(x)) =p $
for the case $q = c_1^p(x).$
The other case $q = -c_1^p(x) = c_1^{-p}(x)$ is similarly shown.
As the function $c_2^q$ is continuous, we have 
\begin{equation*}
c_2^q(\eta_h(x)) = 
\lim_{k\to{\infty}} c_2^q(\psi^{q k}(h(x))).
\label{c2q}
\end{equation*}  
By the cocycle property \eqref{eq:1cocycle} for $c_2$
and 
Definition \ref{def:scoe} (v), we have
\begin{align*}
 c_2^q(\psi^{q k}(h(x))) 
= & c_2^{q+ q k}(h(x)) -  c_2^{q k}(h(x)) \\
= & c_2^{c^{(k+1)p}_1(x)}(h(x)) -  c_2^{ c^{kp}_1(x)}(h(x)) \\ 
= &\left( (k+1)p -d_2(\psi^{c_1^{(k+1)p}(x)}(h(x)), h(\phi^{(k+1)p}(x)))\right) \\
  -& \left( k p -d_2(\psi^{c_1^{k p}(x)}(h(x)), h(\phi^{k p}(x)))\right) \\
=&  p -d_2(\psi^{(k+1)q}(h(x)), h(x)) 
        +d_2(\psi^{k p}(h(x)), h(x)).
\end{align*}
We have
\begin{equation*}
\lim_{n\to{\infty}}\psi^n(\psi^{qk}(h(x))) 
=\lim_{n\to{\infty}}\psi^n(h(\phi^{qk}(x)))
=\lim_{n\to{\infty}}\psi^n(h(x))
 \end{equation*} 
and 
\begin{equation*}
\lim_{n\to{\infty}}d(\psi^n(\psi^{qk}(h(x))), \psi^n(h(x)))
=0,
\end{equation*}
where the above $d(\cdot, \cdot )$ is the metric on $Y$.
This implies that 
$\psi^{qk}(h(x)) \in Y^s(h(x))$ for all $k \in \Z$.
As $\eta_h(x)\in Y^s(h(x))$ by (iv),
we have 
$\psi^{qk}(h(x)) \in Y^s(\eta_h(x))$ for all $k \in \Z,$
so that  there exists $k_0 \in N$ such that for all $k \ge k_0$ and $l \in \N$
\begin{equation*}
d(\psi^{ql}(\psi^{qk}(h(x)) ), \psi^{ql}(\eta_h(x))) < \epsilon_0.
\end{equation*}
Hence for $j=1,\dots,q-1,$ there exists
$k_j \in  \N$ such that for all $k \ge k_j$ and $l \in \N$
\begin{equation*}
d(\psi^{ql +j}(\psi^{qk}(h(x)) ), \psi^{ql+j}(\eta_h(x))) < \epsilon_0.
\end{equation*}
We then find $K\in \N$ such that for all $k \ge K$ and $n \in  \N,$
\begin{equation*}
d(\psi^{n}(\psi^{qk}(h(x)) ), \psi^{n}(\eta_h(x))) < \epsilon_0.
\end{equation*}
This implies 
$\psi^{qk}(h(x)) \in Y^s(\eta_h(x),\epsilon_0)$ for all $k \ge K$.
Since
$$
\lim_{k\to{\infty}}\psi^{(k+1)q}(h(x)) =
\lim_{k\to{\infty}}\psi^{q k}(h(x)) = \eta_h(x),
$$
by the continuity of $d_2$, we see that
$$
\lim_{k\to{\infty}}d_2(\psi^{(k+1)q}(h(x)), h(x)) 
=\lim_{k\to{\infty}}d_2(\psi^{k p}(h(x)), h(x))
=d_2(\eta_h(x), h(x)),
$$ 
thus proving   
$\lim_{k\to{\infty}} c_2^q (\psi^{q k}(h(x))) =p.$
\end{proof}
For a $q$-periodic point $y$ in $Y,$
we put
\begin{equation*}
\eta_{h^{-1}}(y) =\lim_{m\to{\infty}}\phi^{|c^q_2(y)|\cdot m}(h^{-1}(y)).
\end{equation*}
The above limit exists in $X$ by a similar manner to 
Lemma \ref{lem:periodicpoints} (ii),
and $\eta_{h^{-1}}(y)$ is $c_2^q(y)$-periodic point in $X$. 
 \begin{lemma}
For a periodic point $x$ in $X$, 
we have
\begin{equation}
\eta_{h^{-1}}(\eta_h(x)) = \phi^{-d_2(\eta_h(x), h(x))} (x). \label{eq:etaeta}
\end{equation}
Hence 
$\eta_{h^{-1}}(\eta_h(x))$ belongs to the periodic orbit of $x$  under $\phi$.
\end{lemma} 
\begin{proof}
Suppose that 
$\phi^p(x) = x$.
Take the constants
$0<\epsilon_1 < \epsilon_0$ and $0< \lambda_0<1$ for the Smale space 
$(X,\phi)$ as in Definition \ref{def:smalespace} and right after Definition \ref{def:smalespace}.
By using Definition \ref{def:scoe} (ii),
 we know that Lemma \ref{lem:periodicpoints} (iv)
implies that $(\eta_h(x), h(x)) \in G_\psi^s,$
so that 
$$
(\phi^{d_2(\eta_h(x), h(x))}(h^{-1}(\eta_h(x))), x) \in G_\phi^s
$$ 
because of Definition \ref{def:scoe} (iv). 
Hence for $\epsilon >0$ with $0 < \epsilon <\epsilon_1,$
there exists 
$n_0 \in \N$ such that 
\begin{equation*}
d(\phi^n(\phi^{n_0}(\phi^{d_2(\eta_h(x), h(x))}(h^{-1}(\eta_h(x))))), \, 
\phi^n(\phi^{n_0}(x))) < \epsilon
\quad
\text{for } n=0,1,2,\dots
\end{equation*}
where the above $d(\cdot, \cdot)$ is the metric on $X,$
and hence 
$$
\phi^{n_0}(\phi^{d_2(\eta_h(x), h(x))}(h^{-1}(\eta_h(x)))) \in 
X^s(\phi^{n_0}(x), \epsilon).
$$
For any $l \in \N$, we have by \eqref{eq:2.2.3}
\begin{align*}
 & d(\phi^{l}(\phi^{n_0 + d_2(\eta_h(x), h(x))}(h^{-1}(\eta_h(x)))), \, \phi^l(\phi^{n_0}(x))) \\
\le 
&\lambda_0^l d(\phi^{n_0 +d_2(\eta_h(x), h(x))}(h^{-1}(\eta_h(x)))), \, \phi^{n_0}(x))  
\le \lambda_0^l\cdot \epsilon  \\
\end{align*}
so that 
\begin{equation*}
\lim_{n\to{\infty}}\phi^{|p| n}(\phi^{n_0 + d_2(\eta_h(x), h(x))}(h^{-1}(\eta_h(x))))
=\lim_{n\to{\infty}}\phi^{|p|n}(\phi^{n_0}(x)) 
=\phi^{n_0}(x).
\end{equation*}
Since
\begin{align*}
\lim_{n\to{\infty}}\phi^{|p| n}(\phi^{n_0 + d_2(\eta_h(x), h(x))}(h^{-1}(\eta_h(x)))
=& \phi^{n_0 + d_2(\eta_h(x), h(x))}
     \left(\lim_{n\to{\infty}}\phi^{|p| n}(h^{-1}(\eta_h(x))) \right) \\
= & \phi^{n_0 + d_2(\eta_h(x), h(x))}(\eta_{h^{-1}}(\eta_h(x))),
\end{align*}
the equality
\begin{equation*}
\phi^{n_0 + d_2(\eta_h(x), h(x))}(\eta_{h^{-1}}(\eta_h(x)))
=
\phi^{n_0}(x)
\end{equation*}
holds, thus proving \eqref{eq:etaeta}.
\end{proof}

We thus reach the following proposition.
\begin{proposition}\label{prop:4.9}
Suppose that
$(X, \phi) \underset{SCOE}{\sim}(Y, \psi).$
Then there exists a bijective map
$\xi_h: \Porb(X)\longrightarrow \Porb(Y)$
such that 
\begin{equation*}
|\xi_h(\gamma)| = |c_1^{|\gamma|}(x) | 
\quad \text{ for } \gamma \in \Porb(X)
\text{ with }
\gamma =\{ \phi^n(x)\mid n \in \Z\}.
\end{equation*}
 \end{proposition}
\begin{proof}
For $\gamma =\{ \phi^n(x)\mid n \in \Z\}\in \Porb(X)$,
put  $p =|\gamma|$ the positive least period of $x$.
Define
\begin{equation*}
\xi_h(\gamma) = \{ \psi^n(\eta_h(x)) \mid n \in \Z \} 
\end{equation*}
Since 
$\eta_h(x)$ is a periodic point in $Y$ with its least  period $c_1^p(x),$ 
$\xi_h(\gamma)$ is a periodic orbit in $Y$ such that 
$|\xi_h(\gamma)| = |c_1^p(x)|.$
We note the corresponding statement for $h^{-1}$ to Lemma \ref{lem:periodicpoints} (iii) 
holds so that we have
the equality
\begin{equation}
\eta_{h^{-1}}(\psi^n(y)) = \phi^{c_2^n(y)}(\eta_{h^{-1}}(y)), \qquad n \in \Z
 \label{eq:4.16}
\end{equation}
for a periodic point $y \in Y$.
By \eqref{eq:etah} and  \eqref{eq:4.16},
we have
\begin{equation*}
\eta_{h^{-1}}(\psi^n(\eta_h(x))) 
=\phi^{c_2^n(\eta_h(x))}(\eta_{h^{-1}}(\eta_h(x)))
=\phi^{c_2^n(\eta_h(x)) -d_2(\eta_h(x)), h(x))}(x)
\end{equation*}
Hence 
$\eta_{h^{-1}}(\psi^n(\eta_h(x))) 
$ belongs to $\gamma$ so that we have
$\xi_{h^{-1}}(\xi_h(\gamma)) = \gamma.$
Similarly we have
$\xi_h(\xi_{h^{-1}}(\gamma')) = \gamma'$
for $\gamma' \in \Porb(Y).$
We thus conclude that 
the map 
$\xi_h: \Porb(X)\longrightarrow \Porb(Y)$
is bijective and satisfies the desired property.
\end{proof}
The zeta function $\zeta_\phi(t)$ for the dynamical system $(X,\phi)$
is defined by
\begin{equation*}
\zeta_{\phi}(t) :=
\exp\left\{ 
\sum_{n=1}^\infty \frac{t^n}{n} |\Per_n(X) |\right\}
\qquad
(cf. \cite{Bowen2}, \cite{LM}, \cite{PP}, \cite{Ruelle3}),
\end{equation*}
where
$|\Per_n(X) |$ means the cardinal number of the finite set $\Per_n(X). $
Suppose that
$(X, \phi) \underset{SCOE}{\sim}(Y, \psi).$
By Proposition \ref{prop:4.9}, 
there exists a bijective map
$\xi_h: \Porb(X)\longrightarrow \Porb(Y)$
such that 
\begin{equation*}
|\xi_h(\gamma)| = |c_1^{|\gamma|}(x)| 
\quad \text{ for } \gamma \in \Porb(X)
\text{ with }
\gamma =\{ \phi^n(x)\mid n \in \Z\}.
\end{equation*}
We set the two kinds of dynamical zeta functions 
\begin{equation*}
\zeta_{\xi_h}(t) := \prod_{\gamma\in \Porb(X)}( 1 - t^{|\xi_h(\gamma)|})^{-1} 
\end{equation*} 
and
\begin{equation*}
\zeta_{\phi, c_1}(s) :=
\exp\left\{ 
\sum_{n=1}^\infty \frac{1}{n}\sum_{x \in \Per_n(X)}\exp(-s |c_1^n(x)|)\right\}
\qquad(cf.  \cite{PP}, \cite{Ruelle3}).
\end{equation*}
By putting $t =e^{-s}$, we see that
\begin{equation*}
\zeta_{\xi_h}(t) =\zeta_{\phi, c_1}(s)
\end{equation*}
by general theory of dynamical zeta function
(cf. \cite{PP}, \cite{Ruelle3}).
\begin{theorem}\label{thm:zeta}
Suppose that
$(X, \phi) \underset{SCOE}{\sim}(Y, \psi).$
Let $h:X\longrightarrow Y$ be a homeomorphism 
which gives rise to a stably asymptotic continuous orbit equivalence between them.
Then we have
\begin{equation*}
\zeta_\phi(t)=\zeta_{\xi_{h^{-1}}}(t),
\qquad
\zeta_\psi(t)=\zeta_{\xi_h}(t).
\end{equation*}
\end{theorem}
\begin{proof}
There exists a bijection $\xi_h: \Porb(X) \longrightarrow \Porb(Y)$
such that $|\xi_h(\gamma)| = |c_1^{|\gamma|}(x)|$ for 
$\gamma \in \Porb(X)$ with $\gamma =\{\phi^n(x) \mid n \in \Z\}$.
As $\xi_h$ is bijective with   $|\xi_h(\gamma)| = |c_1^{|\gamma|}(x)|$, 
it is direct to see that
$\zeta_\psi(t) = \zeta_{\xi_h}(t),$
and similarly  
$
\zeta_\phi(t)=\zeta_{\xi_{h^{-1}}}(t)$.
\end{proof}
We remark that a similat statement under the condition
$(X, \phi) \underset{UCOE}{\sim}(Y, \psi)$
to the above theorem holds.
\section{ Asymptotic Ruelle algebras $\R_\phi^a$ with dual actions}
Let us recall the construction of the groupoid $C^*$-algebras from  \'etale groupoids
(\cite{Renault}).
Let $G$ be an  \'etale groupoid
with range map $r: G\longrightarrow G^{\circ}$
and source map  $s: G\longrightarrow G^{\circ}$
from $G$ to the unit space $G^\circ$ of $G$.
In \cite{Renault},  ``r-discrete'' was used 
instead of ``\'etale''.

The (reduced) groupoid $C^*$-algebra $C^*_r(G)$ 
for an \'etale groupoid $G$ is defined as in the following way
 (\cite{Renault}).
Let $C_c(G) $ be the set of all continuous functions on
$G$ with compact support that has a natural product structure of 
$*$-algebra given by
\begin{align*}
(f*g)(u) &  = 
 \sum_{r(t) = r(u)}f(t) g(t^{-1}u) 
           = 
 \sum_{ u = t_1 t_2} f(t_1) g(t_2), \\
  f^*(u) & = \overline{f(u^{-1})}, 
  \qquad f,g \in C_c(G), \quad u \in G.     
\end{align*}
Let $C_0(G^{\circ}) $ be the $C^*$-algebra  of all continuous functions on the unit space
$G^{\circ}$ that vanish at infinity. 
The algebra
$C_c(G) $
is a 
$C_0(G^{\circ}) $-right module, endowed with a $C_0(G^{\circ}) $-valued inner product 
given  by
\begin{align*}
 (\xi f )(t) 
 = & \xi(t)f(s(t)), 
 \qquad
   \xi \in C_c(G), \quad f \in C_0(G^{\circ}),
   \quad t \in G,
 \\
  < \xi, \eta >(x) 
 = & \sum_{x=s(t)}
   \overline{ \xi (t)} \eta (t),
   \qquad
   \xi,\eta \in C_c(G), \quad x \in G^\circ.
 \end{align*}
 Let us denote by $l^2(G)$ the completion of the inner product 
$C_0(G^{\circ}) $-module
$C_c(G)$.
It is a Hilbert $C^*$-right module over the commutative $C^*$-algebra 
$C_0(G^{\circ}) $.
We denote by
$B(l^2(G))$
the $C^*$-algebra of all bounded adjointable
$C_0(G^{\circ}) $-module maps on
$l^2(G).$
Let $\pi $ be the $*$-homomorphism of 
$C_c(G)$ into
$B(l^2(G))$
defined by
$
\pi (f)\xi = f * \xi
$
for
$
f, \xi \in C_c(G).
$
Then the closure of $\pi (C_c(G))$ in
$B(l^2(G))$
is called the (reduced) $C^*$-algebra of the groupoid $G$,
that we denote by
$C^*_r(G).$
If we endow 
$C_c(G)$ with the universal $C^*$-norm,
its completion is called the  
the (full) $C^*$-algebra of the groupoid $G$,
that we denote by
$C^*(G).$
By a general theory of groupoid $C^*$-algebras,
$C^*_r(G)$
 is canonically isomorphic to
$C^*(G)$
if the groupoid is amenable
(\cite{Renault}).
An \'etale groupoid $G$ is said to be essentially principal
if the interior of
$G' =\{ \gamma \in G \mid s(\gamma) = r(\gamma) \}$ is $G^{\circ}$
(\cite[Definition 3.1]{Renault2}). 
By Renault \cite[Proposition 4.7]{Renault}, \cite[Proposition 4.2]{Renault2},
$C_0(G^{\circ})$ is maximal abelian in $C^*_r(G)$ if and only if
$G$ is essentially principal.
\begin{definition}
A Smale space $(X,\phi)$ is said to be {\it asymptotically essentially free}\/
if the interior of the set of $n$-asymptotic periodic points
$\{x \in X \mid (\phi^n(x), x) \in G_\phi^a \}$
is empty for every $n\in \Z$ with $n\ne 0$.
\end{definition}
We are always assuming that the space $X$ is infinite.
Recall that 
a Smale space $(X,\phi)$ is said to be irreducible
if for every ordered pair of open sets $U, V \subset X$,
there exists $K \in \N$ such that 
$\phi^K(U) \cap V \ne \emptyset$.  
It is equivalent to the condition that 
for every ordered pair of open sets $U, V \subset X$,
there exists $K \in \N$ such that 
$\phi^{-K}(U) \cap V \ne \emptyset$. 
The referee kindly showed to the author the following lemma with its proof below.
The author deeply thanks the referee.
\begin{lemma}\label{lem:irreessfree}
If a Smale space $(X,\phi)$ is irreducible and $X$ is infinite,
then $(X,\phi)$ is asymptotically essentially free.
\end{lemma}
\begin{proof}
Let $U_n, n \in \N$ be a countable base of open sets of the topology of $X$.
Since $(X,\phi)$ is irreducible, 
the set $\cup_{n=0}^\infty\phi^{-n}(U_m)$ 
is dense in $X$ for every $m \in \N$.
By Baire's category theorem, 
 $\cap_{m=1}^\infty\cup_{n=0}^\infty\phi^{-n}(U_m)$ 
is dense in $X.$ 
The set 
$\cap_{m=1}^\infty\cup_{n=0}^\infty\phi^{-n}(U_m)$
coincides with
the set of points whose forward orbit is dense in $X.$
 Now suppose that for a fixed $n \ne 0$,
the interior of the set of $n$-asymptotic periodic points
$\{x \in X \mid (\phi^n(x), x) \in G_\phi^a \}$
contains  a non-empty open set $U$.
There exists a point $x \in U$ such that the forward orbit of $x$ is dense in $X$.
Since 
$(\phi^n(x),x) \in G_\phi^a$, we have
$$
\lim_{m\to\infty}d(\phi^m(\phi^n(x)), \phi^m(x)) =0
$$
so that  $\phi^n(x) \in X^s(x)$.
By \cite[Lemma 5.3]{PutSp}, 
there exists 
$\lim_{k\to{\infty}}\phi^{kn}(x),$
denoted by $y$, 
in the set of $n$-periodic points 
$\Per_n(X)$.
We note that although \cite[Lemma 5.3]{PutSp}
is considering only mixing Smale spaces,
the assertion \cite[Lemma 5.3]{PutSp} holds in the irreducible Smale space with $X$ being infinite.  
Since $X$ is infinite,
one may find a point $z {\not}{\in} \{y,\phi(y), \dots, \phi^{n-1}(y) \}.$
Put $\epsilon =\frac{1}{4}\Min\{d(z,\phi^i(y)) \mid i=0,1,\dots,n-1\}$.
Let us denote by
$N_\epsilon(z)$ the $\epsilon$-neighborhood of $z$ of open ball.
We put $V =\cup_{i=0}^{n-1}N_\epsilon(\phi^i(y))$
so that we have
$V\cap N_\epsilon(z) =\emptyset.$
Since $X$ is compact,
there exists $\delta>0$ such that 
for all $w_1,w_2 \in X,$  $d(w_1,w_2) <\delta$ implies
 $d(\phi^j(w_1),\phi^j(w_2)) <\epsilon$
for all $j=0,1,\dots,n-1$.
In particular, for $j=0$, we have $\delta <\epsilon$.
Since 
$\lim_{k\to{\infty}}\phi^{kn}(x) =y,$ 
there exists $K \in \N$ such that 
$d(\phi^{k n}(x), y) <\delta$ for all $k \ge K$.
Hence we have
$$
d(\phi^j(\phi^{k n}(x)),\phi^j(y)) <\epsilon
\quad
\text{ for all } j=0,1,\dots,n-1, \, \, k \ge K
$$
so that
$\phi^{k n +j}(x) \in N_\epsilon(\phi^j(y)) $
for all 
$j=0,1,\dots,n-1, \, k \ge K.$
Hence we have
$\phi^m(x) \in V$  for all $m \ge K\cdot n$.
This contradicts the condition that the forward orbit of $x$ is dense in $X$.
We thus conclude that the interior of the set 
$\{x \in X \mid (\phi^n(x), x) \in G_\phi^a \}$
is empty.
\end{proof}
\begin{lemma}\label{lem:essprin}
A Smale space $(X,\phi)$ is  asymptotically essentially free
if and only if the groupoid 
$G_\phi^a\rtimes \Z$ is essentially principal.
\end{lemma}
\begin{proof}
As we have
\begin{align*}
(G_\phi^a\rtimes \Z)'
= & \cup_{n \in \Z}\{(x,n,z) \in G_\phi^a\rtimes \Z \mid x=z\}\\
= & \cup_{n \in \Z}\{(x,n,x) \in X \times \Z \times X  \mid (\phi^n(x),x) \in G_\phi^a\},
\end{align*}
the interior $\interior((G_\phi^a\rtimes \Z)')$
of $G_\phi^a\rtimes \Z$ is 
\begin{equation*}
\interior((G_\phi^a\rtimes \Z)')
=  \cup_{n \in \Z}
\interior(\{(x,n,x) \in X \times \Z \times X  \mid (\phi^n(x),x) \in G_\phi^a\}).
\end{equation*}
For $n=0$,
we see that
$$
\interior(\{(x,0,x) \in X \times \Z \times X  \mid (x,x) \in G_\phi^a\})
 = (G_\phi^a\rtimes\Z)^{\circ}=X.
$$
 Hence  
$\interior((G_\phi^a\rtimes \Z)') = (G_\phi^a\rtimes\Z)^{\circ}$ 
if and only if
$\interior(\{(x,n,x) \in X \times \Z \times X  \mid (\phi^n(x),x) \in G_\phi^a\})$
is empty for all $n \in \Z$ except $n=0$.
This implies that 
$(X,\phi)$ is  asymptotically essentially free
if and only if 
$G_\phi^a\rtimes \Z$ is essentially principal.
\end{proof}
The following proposition as well as Lemma \ref{lem:amenable}
is well-known to experts through \cite[Theorem 3.1]{Putnam1}.
The proof is also direct from 
Renault's result \cite[Proposition 4.6]{Renault}.
\begin{proposition}\label{prop:simple}
If  a Smale space $(X,\phi)$ is irreducible,
then the groupoid $C^*$-algebra
$C^*_r(G_\phi^a\rtimes \Z)$ is simple.
\end{proposition}

 \begin{lemma}[{cf. \cite[Theorem 1.1]{PutSp}}]\label{lem:amenable}
The groupoids $G_\phi^a$ and $G_\phi^a\rtimes\Z$
are both amenable.
\end{lemma}
By Lemma \ref{lem:amenable},
the full groupoid $C^*$-algebras 
$C^*(G_\phi^a), C^*(G_\phi^a\rtimes\Z)$
and the reduced groupoid 
$C^*$-algebras $C^*_r(G_\phi^a), C^*_r(G_\phi^a\rtimes\Z)
$
are canonically isomorphic respectively.
We do not distinguish them
and write them 
$C^*(G_\phi^a), C^*(G_\phi^a\rtimes\Z)$, respectively.

For an irreducible Smale space 
$(X,\phi)$,
the asymptotic Ruelle algebra  
$\R_\phi^a$
is defined as the
groupoid $C^*$-algebras
$C^*(G_{\phi}^{a}\rtimes \Z)
$
of the \'etale groupoid 
$G_{\phi}^{a}\rtimes \Z$.
The algebra  has been written 
 $R_a$
in Putnam's paper \cite{Putnam1}.
In this paper, we denote it
by $\R_\phi^a$.
As in the papers \cite{Putnam1}, \cite{PutSp},
the groupoid 
$G_{\phi}^{a}\rtimes \Z$ 
is the semidirect product of 
the goupoid 
$G_{\phi}^{a}$ by the integer group $\Z$,
one knows that the algebra
$\R_\phi^a$ is naturally isomorphic to the crossed product $C^*$-algebra
$C^*(G_{\phi}^{a})\rtimes \Z$
of the groupoid $C^*$-algebra  
$C^*(G_{\phi}^{a})$
by  $ \Z$.

In the construction of the 
groupoid $C^*$-algebras
$C^*(G_{\phi}^{a}\rtimes \Z),
$
we first define the unitary group $U_t^\phi$ for $t \in \T = {\mathbb{R}}/\Z$
on 
$l^2(G_{\phi}^{a}\rtimes \Z)$
by setting
\begin{equation}
[U_t^\phi\xi](x,n,z) = \exp{(2\pi\sqrt{-1}nt)}\xi(x,n,z) \label{eq:5.1}
\end{equation}
for 
$\xi \in l^2(G_{\phi}^{a}\rtimes \Z), \, (x,n,z) \in G_{\phi}^{a}\rtimes \Z.$
The automorphisms $\Ad(U_t^\phi), t \in \T$ 
on $B(l^2(G_{\phi}^{a}\rtimes \Z))$
leave $\R_\phi^a$ globally invariant,
and yield an action of $\T$ on $\R_\phi^a$.
Let us denote by $\rho^\phi_t$ 
the action
$\Ad(U_t^\phi), t \in \T$ on $\R_\phi^a$.
It is direct to see that the action is exactly corresponds to the dual action
of the crossed product
$C^*(G_{\phi}^{a})\rtimes \Z$.

A continuous function
$f:G_{\phi}^{a}\rtimes \Z\rightarrow \Z$
is called a continuous homomorphism
if it satisfies 
\begin{equation*}
f(\gamma_1 \gamma_2 ) 
= f(\gamma_1) + f(\gamma_2) 
\quad \text{ for } \gamma_1, \gamma_2 \in G_{\phi}^a\rtimes\Z.
\end{equation*}
It defines a one-parameter unitary group
$U_t(f), t \in \T$ on $l^2(G_{\phi}^{a}\rtimes \Z)$
by setting
\begin{equation}
[U_t(f) \xi](x,n,z) = \exp{(2\pi\sqrt{-1}f(x,n,z)t)}\xi(x,n,z)
\end{equation}
for
$\xi \in l^2(G_{\phi}^{a}\rtimes \Z), \, (x,n,z) \in G_{\phi}^{a}\rtimes \Z.$
If in particular 
for the continuous homomorphism
 $d_\phi(x,n,z) = n,$
 we have
$U_t(d_\phi) = U_t^\phi$
by \eqref{eq:5.1}.

Now suppose that
$(X, \phi) \underset{ACOE}{\sim}(Y, \psi).$
Let 
$\varphi_h: G_{\phi}^{a}\rtimes \Z\longrightarrow G_{\psi}^{a}\rtimes \Z
$ 
be the isomorphism of the \'etale groupoids
and
$h:X\longrightarrow Y$ 
the homeomorphism which give rise to the asymptotic continuous orbit equivalence between them.
We define two unitaries
$$
V_h: l^2(G_{\psi}^{a}\rtimes \Z)\longrightarrow  l^2(G_{\phi}^{a}\rtimes \Z),
\qquad
V_{h^{-1}}: l^2(G_{\phi}^{a}\rtimes \Z)\longrightarrow  l^2(G_{\psi}^{a}\rtimes \Z),
$$
by setting
\begin{align}
[V_h \zeta](x,n,z) =& \zeta(\varphi_h(x,n,z)),
\qquad
\zeta \in l^2(G_{\psi}^{a}\rtimes \Z), \, (x,n,z) \in G_{\phi}^{a}\rtimes \Z,
\label{eq:Vh}\\
[V_{h^{-1}} \xi](y,m,w) =& \xi(\varphi_{h^{-1}}(y,m,w)),
\qquad
\xi \in l^2(G_{\phi}^{a}\rtimes \Z), \, (y,m,w) \in G_{\psi}^{a}\rtimes \Z.
\end{align}
Since the unit space
$(G_{\phi}^{a}\rtimes \Z)^\circ$
is identified with the original space $X$
through the correspondence
$(x,0,x) \in (G_{\phi}^{a}\rtimes \Z)^\circ\longrightarrow 
x \in X$
and
$(G_{\phi}^{a}\rtimes \Z)^\circ$
is clopen in 
$G_{\phi}^{a}\rtimes \Z$,
we regard $C(X)$  as a subalgebra of 
$\R_\phi^a$.
Similarly 
 $C(Y)$ is regarded  as a subalgebra of 
$\R_\psi^a.$
\begin{proposition}\label{prop:5.1}
Suppose that
$(X, \phi) \underset{ACOE}{\sim}(Y, \psi),$
and keep the above notation.
Let 
$\varphi_h:G_{\phi}^{a}\rtimes \Z\longrightarrow G_{\psi}^{a}\rtimes \Z$
be the isomorphism of the \'etale groupoids
giving rise to $(X, \phi) \underset{ACOE}{\sim}(Y, \psi).$
Let
$f:G_{\phi}^{a}\rtimes \Z\longrightarrow \Z$
and 
$g:G_{\psi}^{a}\rtimes \Z\longrightarrow \Z$
be continuous homomorphisms satisfying 
$f = g\circ\varphi_h$.
Then there exists an isomorphism
$\Phi: \R_\phi^a \longrightarrow \R_\psi^a$
of $C^*$-algebras
such that 
$\Phi(C(X)) = C(Y)$
and 
\begin{equation}
\Phi\circ \Ad(U_t(f)) = \Ad(U_t(g))\circ \Phi,
\qquad
\text{ for }
t \in \T. \label{eq:prop5.1}
\end{equation}
\end{proposition}
\begin{proof}
We set 
$\Phi = \Ad(V_{h^{-1}})$.
It satisfies 
$\Phi(C_c(G_\phi^a\rtimes\Z)) =C_c(G_\psi^a\rtimes\Z)$
and hence
$\Phi(\R_\phi^a) = \R_\psi^a,$
and
$\Phi(C(X))= C(Y).$
For 
$\zeta \in l^2(G_{\psi}^{a}\rtimes \Z)$,
$(y,m,w) \in G_{\psi}^{a}\rtimes \Z$
and
$a \in C_c(G_{\phi}^{a}\rtimes \Z)$,
we have the following equalities:
\begin{align*}
 &[(\Phi\circ\Ad(U_t(f)))(a)\zeta](y,m,w) \\
=& [V_{h^{-1}}U_t(f) a U_t(f)^* V_h \zeta](y,m,w) \\
=&[U_t(f) a U_t(f)^* V_h \zeta](\varphi_h^{-1}(y,m,w)) \\
=&  \exp{(2\pi\sqrt{-1}f(\varphi_h^{-1}(y,m,w))t)}
      [a * (U_t(f)^* V_h \zeta)](\varphi_h^{-1}(y,m,w)) \\
=&  \exp{(2\pi\sqrt{-1}f(\varphi_h^{-1}(y,m,w))t)}
      \left( \sum_{r(\gamma) =r(\varphi_h^{-1}(y,m,w))}
      a(\gamma) (U_t(f)^* V_h \zeta)(\gamma^{-1}\varphi_h^{-1}(y,m,w)) \right) \\
=&  \exp{(2\pi\sqrt{-1}f(\varphi_h^{-1}(y,m,w))t)}\\
 &    \left(\sum_{r(\gamma) =h^{-1}(y)}
      a(\gamma)
       \exp{(-2\pi\sqrt{-1}f(\gamma^{-1}\varphi_h^{-1}(y,m,w))t)}
      (V_h \zeta)(\gamma^{-1}\varphi_h^{-1}(y,m,w)) \right) \\
=&  \exp{(2\pi\sqrt{-1}f(\varphi_h^{-1}(y,m,w))t)}\\
 &    \left(\sum_{r(\gamma) =h^{-1}(y)}
      a(\gamma)
       \exp{(-2\pi\sqrt{-1}(f(\gamma^{-1}) + f(\varphi_h^{-1}(y,m,w)))t)}
      (V_h \zeta)(\gamma^{-1}\varphi_h^{-1}(y,m,w)) \right) \\
=&        \sum_{r(\gamma) = h^{-1}(y)}
      a(\gamma)
       \exp{(-2\pi\sqrt{-1}f(\gamma^{-1})t)}
       \zeta(\varphi_h(\gamma^{-1})\cdot (y,m,w)) 
\end{align*}
and
\begin{align*}
 &[(\Ad(U_t(g))\circ \Phi)(a)\zeta](y,m,w) \\
= & [U_t(g) V_{h^{-1}} a V_h U_t(g)^* \zeta](y,m,w) \\
=& \exp{(2\pi\sqrt{-1}g(y,m,w)t)}
     [V_{h^{-1}} a V_h U_t(g)^* \zeta](y,m,w) \\
=& \exp{(2\pi\sqrt{-1}g(y,m,w)t)}
     [ a V_h U_t(g)^* \zeta](\varphi_h^{-1}(y,m,w)) \\
=& \exp{(2\pi\sqrt{-1}g(y,m,w)t)}
      \left(\sum_{r(\gamma) =r(\varphi_h^{-1}(y,m,w))}
      a(\gamma) (V_h U_t(g)^* \zeta)(\gamma^{-1}\varphi_h^{-1}(y,m,w)) \right) \\
=& \exp{(2\pi\sqrt{-1}g(y,m,w)t)}
      \left(\sum_{r(\gamma) =h^{-1}(y)}
      a(\gamma) ( U_t(g)^* \zeta)(\varphi_h(\gamma^{-1}\varphi_h^{-1}(y,m,w))) \right) \\
=& \exp{(2\pi\sqrt{-1}g(y,m,w)t)} \\
  &   \left( \sum_{r(\gamma) =h^{-1}(y)}
      a(\gamma)
       \exp{(-2\pi\sqrt{-1}g(\varphi_h(\gamma^{-1}) \cdot (y,m,w))t)}
      \zeta(\varphi_h(\gamma^{-1}) \cdot (y,m,w)) \right) \\
=&       \sum_{r(\gamma) =h^{-1}(y)}
      a(\gamma)
       \exp{(-2\pi\sqrt{-1}g(\varphi_h(\gamma^{-1}))t)}
      \zeta(\varphi_h(\gamma^{-1})\cdot (y,m,w)). 
\end{align*}
By assumption, we see that 
$f(\gamma^{-1})=g(\varphi_h(\gamma^{-1}))$,
so that we obtain
$\Phi\circ\Ad(U_t(f)) =\Ad(U_t(g))\circ \Phi$.
\end{proof}

We are assuming that 
$(X, \phi) \underset{ACOE}{\sim}(Y, \psi).$
Let $h:X\longrightarrow Y$ 
be a homeomorphism which gives rise to the asymptotic continuous orbit equivalence between them.
Take the continuous functions
$
c_1:X\longrightarrow \Z, \,
c_2:Y\longrightarrow \Z
$
and two-cocycle functions
$
d_1: G_\phi^a \longrightarrow \Z,\, 
d_2: G_\psi^a \longrightarrow \Z
$ 
satisfying
 Definition \ref{def:acoe}
of asymptotic continuous orbit equivalence.
We set
two functions
\begin{align}
c_\phi(x,n,z) = & c_1^n(x)+ d_1(\phi^n(x), z),
 \qquad
 (x,n,z) \in G_{\phi}^a\rtimes\Z, \label{eq:cphixnz}\\
c_\psi(y,m,w) = & c_2^m(y)+ d_2(\psi^m(y), w),
 \qquad
 (y,m,w) \in G_{\psi}^a\rtimes\Z.
\end{align}
They satisfy
\begin{align*}
c_\phi(\gamma_1 \gamma_2 ) 
= & c_\phi(\gamma_1) + c_\phi(\gamma_2) 
\quad \text{ for } \gamma_1, \gamma_2 \in G_{\phi}^a\rtimes\Z,\\  
c_\psi(\gamma'_1 \gamma'_2 ) 
= & c_\psi(\gamma'_1) + c_\psi(\gamma'_2) 
\quad \text{ for } \gamma'_1, \gamma'_2 \in G_{\psi}^a\rtimes\Z,
\end{align*}
and hence they are continuous homomorphisms
satisfying
\begin{align*}
\varphi_h(x,n, z) & = (h(x), c_\phi(x,n,z), h(z)),
\qquad (x,n,z) \in G_{\phi}^a\rtimes\Z, \\
\varphi_{h^{-1}}(y,m,w) & =(h^{-1}(y),c_\psi(y,m,w), h^{-1}(w)),
\qquad
 (y,m,w) \in G_{\psi}^a\rtimes\Z.
\end{align*}
We note that the following identities hold:
\begin{gather}
d_\psi(\varphi_h(x,n,z)) = c_\phi(x,n,z), \qquad
d_\phi(\varphi_h^{-1}(y,m,w)) = c_\psi(y,m,w), \label{eq:5.10}\\
c_\psi(\varphi_h(x,n,z)) =d_\phi(x,n,z) = n, \qquad
c_\phi(\varphi_h^{-1}(y,m,w)) =d_\psi(y,m,w) = m. \label{eq:5.11}
\end{gather}

\begin{theorem}\label{thm:main3}
Suppose that 
Smale spaces  $(X,\phi)$ and $(Y,\psi)$ are irreducible.
Then the following assertions are equivalent:
\begin{enumerate}
\renewcommand{\theenumi}{\roman{enumi}}
\renewcommand{\labelenumi}{\textup{(\theenumi)}}
\item
$(X, \phi)$ and $(Y, \psi)$
are asymptotically continuous orbit equivalent.
\item
There exists an isomorphism $\R_\phi^a \longrightarrow \R_\psi^a$
of $C^*$-algebras such that 
$\Phi(C(X)) = C(Y)$
and 
$$
\Phi\circ \rho^\phi_t = \Ad(U_t(c_\psi))\circ \Phi,
\qquad
\Phi\circ\Ad(U_t(c_\phi)) =\rho^\psi_t\circ \Phi
\qquad
\text{ for }
t \in \T
$$
for some continuous homomorphisms
$c_\phi:G_\phi^a\rtimes\Z\longrightarrow\Z$ and
$c_\psi:G_\psi^a\rtimes\Z\longrightarrow\Z.$
\end{enumerate}
\end{theorem}
\begin{proof}
(i) $\Longrightarrow$ (ii):
Take $f = d_\phi, g= c_\psi$
in the equality \eqref{eq:prop5.1}.
We then have
$\Ad(U_t(d_\phi)) = \rho^\phi_t$
and
$c_\psi\circ \varphi_h = d_\phi$ 
by
\eqref{eq:5.11}.
Hence by \eqref{eq:prop5.1}, we obtain
\begin{equation}
\Phi\circ \rho^\phi_t = \Ad(U_t(c_\psi))\circ \Phi, \qquad t \in \T. 
\label{eq:5.13}
\end{equation}
Take $f = c_\phi, g= d_\psi$
in the equality \eqref{eq:prop5.1}.
We then have
$\Ad(U_t(d_\psi)) = \rho^\psi_t$
and
$c_\phi\circ (\varphi_h)^{-1} = d_\psi$ 
by
\eqref{eq:5.11}.
Hence by \eqref{eq:prop5.1}, we obtain
\begin{equation}
\Phi\circ\Ad(U_t(c_\phi)) =\rho^\psi_t\circ \Phi, \qquad t \in \T.
\label{eq:5.14}
\end{equation}

(ii) $\Longrightarrow$ (i):
Since the Smale spaces 
$(X,\phi)$ and $(Y,\psi)$ are both asymptotically essentially free,
the \'etale groupoids
$G_{\phi}^{a} \rtimes \Z$ and $G_{\psi}^{a} \rtimes \Z$
are both essentially principal by Lemma \ref{lem:essprin}.
By Renault \cite[Proposition 4.11]{Renault2},
an isomorphism 
$\R_\phi^a \longrightarrow \R_\psi^a$
of $C^*$-algebras such that 
$\Phi(C(X)) = C(Y)$
yields an isomorphism of the underlying
\'etale groupoids
$G_{\phi}^{a} \rtimes \Z$ and $G_{\psi}^{a} \rtimes \Z$.
Hence by Theorem \ref{thm:main1},
we see the implication
(ii) $\Longrightarrow$ (i).
\end{proof}
\begin{remark}
Similar discussions to 
Theorem \ref{thm:main3}
 for topological Markov shifts  with continuous orbit equivalence 
 are seen in  several papers
(cf. \cite{CR}, \cite{CRS}, \cite{MaPacific}, \cite{MaJOT2015}, \cite{MaProcAMS2017}, \cite{MaMZ2017}, \cite{MatuiPLMS},
etc. )    

\end{remark}

\section{Asymptotic conjugacy}
In this section, we will introduce a notion of asymptotic conjugacy between Smale spaces,
and describe the asymptotic conjugacy 
in terms of the Ruelle algebras with its dual actions.
 Smale spaces $(X,\phi)$
 in this section are assumed to be 
irreducible and the space $X$ is infinite. 
\begin{definition}
Smale spaces  $(X,\phi)$ and $(Y,\psi)$ are said to be
 {\it asymptotically  conjugate}\/ if they are 
asymptotically  continuously orbit equivalent 
such that  we may take their cocycle functions
such as
$c_1 \equiv 1,\, c_2 \equiv 1$ and $d_1 \equiv 0, \, d_2 \equiv 0$
in Definition \ref{def:acoe}. 
\end{definition}
In this situation, we write
$(X, \phi) \underset{a}{\cong}(Y, \psi).$
Namely we have
$(X,\phi)$ and $(Y,\psi)$ are said to be
 asymptotically  conjugate if and only if
there exists a homeomorphism
$h: X\longrightarrow Y$ 
which satisfies
the following four conditions:
\begin{enumerate}
\renewcommand{\theenumi}{\roman{enumi}}
\renewcommand{\labelenumi}{\textup{(\theenumi)}}
\item
There exists a continuous function $k_{1,n}:X\longrightarrow \Zp$ for each $n \in \Z$
such that
\begin{gather*}
(\psi^{k_{1,n}(x) + n}(h(x)), \psi^{k_{1,n}(x)}(h(\phi^n(x)))) \in G_\psi^{s,0},\\
(\psi^{-k_{1,n}(x) + n}(h(x)), \psi^{-k_{1,n}(x)}(h(\phi^n(x)))) \in G_\psi^{u,0}.
\end{gather*}
\item
There exists a continuous function $k_{2,n}:Y\longrightarrow \Zp$ for each $n \in \Z$
such that
\begin{gather*}
(\phi^{k_{2,n}(y) + n}(h^{-1}(y)), \phi^{k_{2,n}(y)}(h^{-1}(\psi^n(y)))) \in G_\phi^{s,0},\\
(\phi^{-k_{2,n}(y) + n}(h^{-1}(y)), \phi^{-k_{2,n}(y)}(h^{-1}(\psi^n(y)))) \in G_\phi^{u,0}.
\end{gather*}
\item
There exists a continuous function $m_1;G_\phi^a \longrightarrow \Zp$ 
such that
\begin{gather*}
(\psi^{m_1(x,z)}(h(x)), \psi^{m_1(x,z)}(h(z))) \in G_\psi^{s,0}
\quad\text{ for } (x,z) \in G_\phi^a,\\
(\psi^{-m_1(x,z)}(h(x)), \psi^{-m_1(x,z)}(h(z)))\in G_\psi^{u,0}
\quad\text{ for } (x,z) \in G_\phi^a.
\end{gather*}
\item
There exists a continuous function $m_2;G_\psi^a \longrightarrow \Zp$ 
such that
\begin{gather*}
(\phi^{m_2(y,w)}(h^{-1}(y)), \phi^{m_2(y,w)}(h^{-1}(w))) \in G_\phi^{s,0}
\quad\text{ for } (y,w) \in G_\psi^a,\\
(\phi^{-m_2(y,w)}(h^{-1}(y)), \phi^{-m_2(y,w)}(h^{-1}(w))) \in G_\phi^{u,0}
\quad\text{ for } (y,w) \in G_\psi^a.
\end{gather*}
\end{enumerate}
Recall that the Ruelle algebra
$\R_\phi^a$ is defined as the groupoid $C^*$-algebra
$C^*(G_\phi^a\rtimes\Z)$ 
of the \'etale groupoid $G_\phi^a\rtimes\Z$.
It is naturally isomorphic to the crossed product
$C^*(G_\phi^a)\rtimes\Z$
of the $C^*$-algebra $C^*(G_\phi^a)$ 
by the automorphism
$\phi^*$ on $C^*(G_\phi^a)$
induced by the formula
\begin{equation*}
\phi^*(f)(x,z) = f(\phi(x), \phi(z))
\qquad
\text{ for }
f \in C_c(G_\phi^a), \, (x,z) \in G_\phi^a.
\end{equation*}
Define the unitary
$U_\phi$ on $l^2(G_\phi^a\rtimes\Z)$
by setting
\begin{equation}
(U_\phi\xi)(x,n,z) = \xi(\phi(x),n-1,z) \qquad \text{ for } \xi \in l^2(G_\phi^a\rtimes\Z),
\, (x,n,z) \in G_\phi^a\rtimes\Z. \label{eq:Uphi}
\end{equation}
It is direct to see that 
\begin{equation*}
U_\phi f U_\phi^* = \phi^*(f)  \qquad 
\text{ for } f \in C_c(G_\phi^a),
\end{equation*}
where 
$$
f(x,m,z) =
\begin{cases} 
f(x,z) & \text{ if } m=0,\\
0 & \text{ otherwise }
\end{cases}
\qquad \text{ for } (x,z) \in G_\phi^a.
$$
Now we are assuming that 
$(X,\phi)$ is irreducible,
so that the $C^*$-algebra
$\R_\phi^a$ is simple by Proposition \ref{prop:simple}.
Hence we know that 
$\R_\phi^a$ is isomorphic to the $C^*$-algebra
$C^*(C^*(G_\phi^a), U_\phi)$
generated by the its subalgebra
$C^*(G_\phi^a)$ and the unitary
$U_\phi$.
The  following lemma is directly seen from J. Renault's result
\cite[Proposition 4.11]{Renault2}.
\begin{lemma}\label{lem:groupoidC*6.2}
Let $(X, \phi)$ and $(Y, \psi)$
be irreducible  Smale spaces.
The following assertions are equivalent:
\begin{enumerate}
\renewcommand{\theenumi}{\roman{enumi}}
\renewcommand{\labelenumi}{\textup{(\theenumi)}}
\item
There exists an isomorphism
$
\varphi: G_\phi^a\rtimes\Z \longrightarrow G_\psi^a\rtimes\Z
$  
of \'etale groupoids such that 
$
\varphi(G_\phi^a)=G_\psi^a
$  
and
$
\varphi(G_\phi^{a,0})=G_\psi^{a,0}.
$  
\item
There exists an isomorphism 
$\Phi:\R_\phi^a \longrightarrow \R_\psi^a$
of $C^*$-algebras such that 
$\Phi(C^*(G_\phi^a))=C^*(G_\psi^a)
$
and
$\Phi(C(X)) = C(Y)$
 \end{enumerate}
\end{lemma}
\begin{proof}
By Lemma \ref{lem:2.5},
the spaces
$G_\phi^{a,0}, G_\psi^{a,0}$
are identified with $X,Y$ respectively
as topological spaces.
They are also identified with the unit spaces 
$(G_\phi^a\rtimes\Z)^\circ, (G_\psi^a\rtimes\Z)^\circ,$
respectively.
Since $(X, \phi)$ and $(Y, \psi)$ are irreducible and 
hence asymptotically essentially free,
the \'etale groupoids 
$G_\phi^a\rtimes\Z$ and $G_\psi^a\rtimes\Z$
are both essentially principal by Lemma \ref{lem:essprin}.
The implication
(i) $\Longrightarrow$ (ii) 
is direct.
By Renault \cite[Proposition 4.11]{Renault2},
an isomorphism 
$\R_\phi^a \longrightarrow \R_\psi^a$
of $C^*$-algebras such that 
$\Phi(C(X)) = C(Y)$
yields an isomorphism $\varphi$ of the underlying
\'etale groupoids
$G_{\phi}^{a} \rtimes \Z$ and $G_{\psi}^{a} \rtimes \Z$.
By the construction of the isomorphism $\varphi$
of 
the \'etale groupoids,
we see that
$
\varphi(G_\phi^{a})=G_\psi^{a}
$  
by
the additional condition
$\Phi(C^*(G_\phi^a))=C^*(G_\psi^a),
$
thus proving the implication
(ii) $\Longrightarrow$ (i).
\end{proof}

\begin{proposition}\label{prop:main6.1}
Let
$(X,\phi)$ and $(Y,\psi)$
be irreducible Smale spaces.
Suppose that
there exists an isomorphism $\varPhi: \R_\phi^a \longrightarrow \R_\psi^a$
of $C^*$-algebras such that 
$\varPhi(C(X)) = C(Y)$
and 
$$
\varPhi\circ \rho^\phi_t = \rho^\psi_t \circ \varPhi,
\qquad t \in \T.
$$
Then there exists a homeomorphism
$h:X\longrightarrow Y$ 
which gives rise to an asymptotic continuous orbit equivalence
between  
$(X,\phi)$ and $(Y,\psi),$
such that its cocycle functions satisfy 
\begin{equation*}
c_1\equiv 1, \qquad
c_2\equiv 1, \qquad
d_1\equiv 0, \qquad
d_2\equiv 0.
\end{equation*}
Namely,
$(X, \phi)$ and $(Y, \psi)$
are asymptotically conjugate.
\end{proposition}
\begin{proof}
Suppose that 
there exists an isomorphism $\varPhi: \R_\phi^a \longrightarrow \R_\psi^a$
of $C^*$-algebras such that 
$\varPhi(C(X)) = C(Y)$
and 
$
\varPhi\circ \rho^\phi_t = \rho^\psi_t \circ \varPhi,
t \in \T.
$
We will first show that
$d_1\equiv 0, \,
d_2\equiv 0.
$
Since 
the fixed point algebra
${(\R_\phi^a)}^{\rho^\phi}$ of
 $\R_\phi^a$ under $\rho^\phi$ 
is canonically isomorphic to the groupoid $C^*$-subalgebra
$C^*(G_\phi^a)$,
the isomorphism 
$\varPhi: \R_\phi^a \longrightarrow \R_\psi^a$
satisfies
$\varPhi(C^*(G_\phi^a))=C^*(G_\psi^a).
$
By Lemma \ref{lem:groupoidC*6.2},
we then  find  
a homeomorphism
$h:X\longrightarrow Y$ and
a
groupoid isomorphism
$\varphi_h:G_{\phi}^{a} \rtimes \Z\longrightarrow G_{\psi}^{a} \rtimes \Z$
such that 
$\varphi_h(G_\phi^a) = G_\psi^a,
\varphi_h|_{G_\phi^{a,0}} =h$
and
$\varPhi(f) = f\circ h^{-1}$ for $f \in C(X)$.
For $(x,z) \in G_\phi^a$, we have
\begin{equation*}
\varphi_h(x,0,z)
= (h(x),c_\phi(x,0,z),h(z))
= (h(x), d_1(x,z),h(z)).
\end{equation*}
As 
$\varphi_h(x,0,z)
 \in G_\psi^a$,
we know that
$d_1(x,z)=0$,
and $d_2(y,w) =0$ for $(y,w) \in G_\psi^a$
similarly.


We will second show that 
$c_1\equiv 1, \,
c_2\equiv 1.
$
Since
the isomorphism
$\varPhi: \R_\phi^a \longrightarrow \R_\psi^a$
satisfies
$\varPhi(C(X)) = C(Y)$,
the groupoid isomorphism
$\varphi_h:G_{\phi}^{a} \rtimes \Z\longrightarrow G_{\psi}^{a} \rtimes \Z$
with homeomorphism
$h:X\longrightarrow Y$ yields
an asymptotic continuous orbit equivalence
between them.
They also satisfy the equalities
\begin{align}
\varphi_h(x,n, z) & = (h(x), c_\phi(x,n,z), h(z)), \qquad (x,n,z) \in G_{\phi}^a\rtimes\Z, 
\label{eq:6.3.1}\\
\varphi_{h^{-1}}(y,m,w) & =(h^{-1}(y),c_\psi(y,m,w), h^{-1}(w)),
\qquad
(y,m,w) \in G_{\psi}^a\rtimes\Z.
\end{align}
Let
$V_h$ be the unitary defined in \eqref{eq:Vh}.
As in the proof of Proposition \ref{prop:5.1},
by putting
$\Phi_h = \Ad(V_{h^{-1}})$,
we have
an isomorphism
$\Phi_h: \R_\phi^a \longrightarrow \R_\psi^a$
such that 
$\Phi_h(C(X)) = C(Y)$
and 
$$
\Phi_h\circ \rho^\phi_t = \Ad(U_t(c_\psi))\circ \Phi_h,
\qquad
\Phi_h\circ\Ad(U_t(c_\phi)) =\rho^\psi_t\circ \Phi_h.
$$
Let $U_\phi$ be the unitary defined in \eqref{eq:Uphi},
which corresponds to the implementing unitary of the positive generator of 
the group representation of $\Z$ in the crossed product
$C^*(G_\phi^a)\rtimes\Z$.
It satisfies the 
equality
$U_\phi f U_\phi^* = f\circ \phi$ for $f \in C(X)$.
For $f \in C(X)$,
as $\varPhi(f) = \Phi_h(f)$, we see that 
\begin{equation*}
\varPhi(U_\phi) \Phi_h(f) \varPhi(U_\phi^*) 
= \varPhi( f\circ \phi)= \Phi_h( f\circ \phi)
= \Phi_h(U_\phi) \Phi_h(f) \Phi_h(U_\phi^*),
\end{equation*}
so that 
\begin{equation*}
\Phi_h^{-1}(\varPhi(U_\phi)) f \Phi_h^{-1}(\varPhi(U_\phi^*)) 
= U_\phi f U_\phi^*.
\end{equation*}
Hence we have
\begin{equation*}
U_\phi^*\Phi_h^{-1}(\varPhi(U_\phi)) f  
=  f U_\phi^*\Phi_h^{-1}(\varPhi(U_\phi))
\quad
\text{ for all }
f \in C(X).
\end{equation*}
Since 
$(X,\phi)$ is irreducible and hence asymptotically essentially free,
the groupoid 
$G_\phi^a\rtimes\Z$
is essentially principal 
by Lemma \ref{lem:essprin}. 
By \cite[Proposition 4.7]{Renault} or \cite[Proposition 4.2]{Renault2}, 
$C(X)=C((G_\phi^a\rtimes\Z)^\circ)$
is a maximal abelian $C^*$-subalgebra of $\R_\phi^a.$
Hence there exists a unitary $f_0 \in C(X)$ such that 
$U_\phi^*\Phi_h^{-1}(\varPhi(U_\phi)) = f_0$,
so that 
\begin{equation}
\varPhi(U_\phi) = \Phi_h(U_\phi f_0).\label{eq:PhiUphi}
\end{equation}
Since
$\varPhi\circ\rho^\phi_t=\rho^\psi_t \circ \varPhi $
and
$
\Phi_h\circ\Ad(U_t(c_\phi)) =\rho^\psi_t\circ \Phi_h,
$
we get by the equality \eqref{eq:PhiUphi}
\begin{equation}
\varPhi\circ \rho^\phi_t(U_\phi) 
= (\Phi_h\circ\Ad(U_t(c_\phi)))(U_\phi f_0).
\label{eq:PhiUphi2}
\end{equation}
As $\rho^\phi_t(U_\phi) = \exp(2\pi\sqrt{-1}t)U_\phi$,
the equality \eqref{eq:PhiUphi2}
goes to
\begin{equation}
\exp{(2\pi\sqrt{-1}t)}\Phi(U_\phi) 
= \Phi_h(U_t(c_\phi) U_\phi f_0 U_t(c_\phi)^*).
\label{eq:PhiUphi3}
\end{equation}
As $f_0 \in C(X)$ and
$U_t(c_\phi)^* = U_t(-c_\phi)$,
we have  
for $\xi \in l^2(G_\phi^a\rtimes\Z)$ and 
$(x,n,z) \in G_\phi^a\rtimes\Z$,
\begin{align*}
[U_t(-c_\phi)f_0\xi](x,n,z)
& =\exp{(2\pi\sqrt{-1} (-c_\phi(x,n,z))t)}[f_0\xi](x,n,z) \\
& =f_0(x) \exp{(2\pi\sqrt{-1} (-c_\phi(x,n,z))t)}\xi(x,n,z) \\
& =[f_0(x) U_t(-c_\phi)\xi](x,n,z) 
\end{align*}
so that 
$U_t(c_\phi)^*f_0 =f_0 U_t(c_\phi)^*$.
Hence the equation
\eqref{eq:PhiUphi3} implies
\begin{equation}
\exp{(2\pi\sqrt{-1}t)}\varPhi(U_\phi) 
= \Phi_h(U_t(c_\phi) U_\phi  U_t(c_\phi)^*f_0),
\label{eq:PhiUphi4}
\end{equation}
which goes to by \eqref{eq:PhiUphi}
\begin{equation}
\exp{(2\pi\sqrt{-1}t)}\Phi_h(U_\phi) 
= \Phi_h(U_t(c_\phi) U_\phi  U_t(c_\phi)^*)
\label{eq:PhiUphi5}
\end{equation}
so that
\begin{equation}
\exp{(2\pi\sqrt{-1}t)}U_\phi 
= U_t(c_\phi) U_\phi  U_t(c_\phi)^*.
\label{eq:PhiUphi6}
\end{equation}
For $\xi \in l^2(G_\phi^a\rtimes\Z)$ and 
$(x,n,z) \in G_\phi^a\rtimes\Z$,
we have the equalities
\begin{align*}
  & [U_t(c_\phi) U_\phi  U_t(c_\phi)^*\xi](x,n,z) \\
=&\exp{(2\pi\sqrt{-1} (c_\phi(x,n,z))t)}[U_\phi  U_t(-c_\phi)\xi](x,n,z) \\
=&\exp{(2\pi\sqrt{-1} (c_\phi(x,n,z))t)}[U_t(-c_\phi)\xi](\phi(x),n-1,z) \\
=&\exp{(2\pi\sqrt{-1} (c_\phi(x,n,z) -c_\phi(\phi(x), n-1,z))t)}\xi(\phi(x),n-1,z). 
\end{align*}
On the other hand,
\begin{equation*}
[\exp{(2\pi\sqrt{-1}t)}U_\phi\xi](x,n,z) 
= \exp{(2\pi\sqrt{-1}t)}\xi(\phi(x),n-1,z).
\end{equation*}
By \eqref{eq:PhiUphi6},
we have
\begin{equation*}
c_\phi(x,n,z) -c_\phi(\phi(x), n-1,z) =1.
\end{equation*}
By \eqref{eq:cphixnz}, we see that
\begin{align*}
  &c_\phi(x,n,z) -c_\phi(\phi(x), n-1,z) \\
=&\{c_1^n(x) + d_1(\phi^n(x),z)\} -\{c_1^{n-1}(\phi(x)) + d_1(\phi^{n-1}(\phi(x)), z) \}\\
=&c_1(x).
\end{align*}
Therefore we have
$c_1(x) =1$ for all $x \in X$,
and $c_2(y) =1$ for all $y \in Y$ similarly.
\end{proof}
Recall that a continuous homomorphism
$d_\phi: G_\phi^a\rtimes\Z\longrightarrow \Z$ is defined 
by $d_\phi(x,n,z) = n$ for $(x,n,z) \in G_\phi^a\rtimes\Z.$ 
\begin{theorem}\label{thm:mains6}
Let
$(X,\phi)$ and $(Y,\psi)$
be irreducible Smale spaces.
Then the following assertions are equivalent:
\begin{enumerate}
\renewcommand{\theenumi}{\roman{enumi}}
\renewcommand{\labelenumi}{\textup{(\theenumi)}}
\item
$(X,\phi)$ and $(Y,\psi)$
are asymptotically conjugate:
$(X, \phi) \underset{a}{\cong}(Y, \psi).$
\item 
There exists an isomorphism
$\varphi: G_\phi^a\rtimes\Z\longrightarrow G_\psi^a\rtimes\Z$
of \'etale groupoids
such that 
$d_\psi\circ\varphi = d_\phi$.
\item
There exists an isomorphism $\Phi: \R_\phi^a \longrightarrow \R_\psi^a$
of $C^*$-algebras such that 
$\Phi(C(X)) = C(Y)$
and 
$$
\Phi\circ \rho^\phi_t = \rho^\psi_t \circ \Phi,
\qquad t \in \T.
$$
\end{enumerate}
\end{theorem}
\begin{proof}
The implication (iii) $\Longrightarrow$ (i)
follows from
Proposition \ref{prop:main6.1}.
In the proof of 
Proposition \ref{prop:main6.1},
we have shown that there exists an isomorphism of groupoids
$\varphi_h: G_\phi^a\rtimes\Z\longrightarrow G_\psi^a\rtimes\Z$
such that
$c_1\equiv 1, \,
c_2\equiv 1
$
and
$d_1\equiv 0, \,
d_2\equiv 0.
$
Hence we have
\begin{equation*}
c_\phi(x,n,z) = c_1^n(x) + d_1(\phi^n(x),z) =n \quad
\text{ for } (x,n,z) \in G_\phi^a\rtimes\Z
\end{equation*}
and 
$c_\psi(y,m,w) =m$, similarly.
This implies that 
$c_\phi = d_\phi$ and $c_\psi = d_\psi.$
By \eqref{eq:6.3.1}, 
we obtain
$d_\psi\circ\varphi = d_\phi$.
This argument shows that the implications
(iii) $\Longrightarrow$ (ii) $\Longrightarrow$ (i)
hold.

 We will show 
the implication (i) $\Longrightarrow$ (iii).
Suppose that 
$(X,\phi)$ and $(Y,\psi)$
are asymptotically conjugate.
Take a homeomorphism
$h:X\longrightarrow Y$ 
which gives rise to the asymptotic conjugacy.
In the proof of (i) $\Longrightarrow$ (ii) of Theorem \ref{thm:main3},
we know that 
$c_\phi = d_\phi$ and $c_\psi = d_\psi$
because 
$c_\phi(x,n,z)  =n
$
for $ (x,n,z) \in G_\phi^a\rtimes\Z$
and
$c_\psi(y,m,w) =m$
for $(y,m,w) \in G_\psi^a\rtimes\Z,$  similarly,
which come from the conditions
$
c_1\equiv 1, \,
c_2\equiv 1, \,
d_1\equiv 0, \,
d_2\equiv 0.
$
Hence we have
$
\Ad(U_t(c_\phi)) = \Ad(U_t(d_\phi)) = \rho^\phi_t
$
and 
$
\Ad(U_t(c_\psi)) = \Ad(U_t(d_\psi)) = \rho^\psi_t.
$
We thus obtain the equality
$\Phi_h\circ\rho^\phi_t=\rho^\psi_t \circ \Phi_h $
by \eqref{eq:5.13} or \eqref{eq:5.14}.
\end{proof}

\section{Extended Ruelle algebras $\R_\phi^{s,u}$}
In this section, we will introduce an extended Ruelle algebra $\R_\phi^{s,u}$
from a certain amenable \'etale groupoid of a Smale space
$(X,\phi)$.
The introduced $C^*$-algebra contains the asymptotic Ruelle algebra $\R^a_\phi$ 
as a fixed point subalgebra under some circle action.  
The extended Ruelle algebras will be useful
in the following sections to investigate the asymptotic Ruelle algebra
$\R^a_\phi$ for topological Markov shifts from the view points of Cuntz--Krieger algebras.

We first introduce the following groupoid
$G_{\phi}^{s,u} \rtimes \Z^2$
 for a Smale space $(X,\phi)$ 
which will be  proved to be \'etale and amenable. 
\begin{equation}
G_{\phi}^{s,u} \rtimes \Z^2
= \{(x, p,q,y) \in X \times \Z\times \Z \times X \mid
(\phi^p(x), y) \in G_{\phi}^{s}, \, (\phi^q(x), y) \in G_{\phi}^{u} \}.
\end{equation}
The following lemma is straightforward.
\begin{lemma}
For
$(x, p,q,y), (x', p',q',y') \in G_{\phi}^{s,u} \rtimes \Z^2$,
we have
\begin{enumerate}
\renewcommand{\theenumi}{\roman{enumi}}
\renewcommand{\labelenumi}{\textup{(\theenumi)}}
\item
$(x, p+p', q + q', y') \in G_{\phi}^{s,u} \rtimes \Z^2$ if $y =x'$.
\item
$(y,-p,-q,x)  \in G_{\phi}^{s,u} \rtimes \Z^2$.
\end{enumerate}
\end{lemma}
Two elements
$(x, p,q,y), (x', p',q',y') \in G_{\phi}^{s,u} \rtimes \Z^2$
are composable 
if and only if $y=x'$.
The multiplication and the inverse 
in $G_{\phi}^{s,u} \rtimes \Z^2$
are given by
\begin{gather*}
(x, p,q,y)\cdot (x', p',q',y') =(x, p+p', q + q', y')\quad \text{ if } y = x',\\
(x, p,q,y)^{-1} = (y,-p,-q,x).
\end{gather*}
We write the unit space 
$(G_{\phi}^{s,u} \rtimes \Z^2)^\circ$
of $G_{\phi}^{s,u} \rtimes \Z^2$
as
\begin{equation*}
(G_{\phi}^{s,u} \rtimes \Z^2)^\circ
= \{ (x,0,0, x) \mid x \in X\}
\end{equation*} 
which is identified with $X$.
Define the range map, source map  
$r, s :G_{\phi}^{s,u} \rtimes \Z^2\longrightarrow (G_{\phi}^{s,u} \rtimes \Z^2)^\circ$
by
\begin{equation*}
r(x,p,q,y) = (x,0,0,x),\qquad
s(x,p,q,y) = (y,0,0,y).
\end{equation*} 
For $p,q \in \Z$ and $n=0,1,\dots$,
we set
\begin{align*}
G_\phi^{s,u,n}(p,q)
= &\{ (x,y) \in X\times X \mid 
(\phi^p(x),y) \in G_\phi^{s,n}, 
(\phi^q(x),y) \in G_\phi^{u,n}\}, \\
G_\phi^{s,u}(p,q)
= &\{ (x,y) \in X\times X \mid 
(\phi^p(x),y) \in G_\phi^{s}, 
(\phi^q(x),y) \in G_\phi^{u}\}. 
\end{align*}
For each $n$, the set
$G_\phi^{s,u,n}(p,q)$ is endowed with the relative topology from $X\times X$.
Since
$G_\phi^{*,n}\subset G_\phi^{*,n+1}$
for
$*=s,u$ and $n=0,1,\dots$,
we have
\begin{equation}
G_\phi^{s,u,n}(p,q)\subset 
G_\phi^{s,u,n+1}(p,q)
\quad
\text{ and }
\quad
G_\phi^{s,u}(p,q) = \cup_{n=0}^\infty G_\phi^{s,u,n}(p,q). \label{eq:Gphisu}
\end{equation}
We may endow $G_\phi^{s,u}(p,q)$ with inductive limit topology
from the inductive system \eqref{eq:Gphisu} of the topological spaces
$\{ G_\phi^{s,u,n}(p,q)\}_{n\in \Zp}$.
Since we may identify
$G_{\phi}^{s,u} \rtimes \Z^2$ with the disjoint union
$\sqcup_{(p,q)\in \Z^2}G_{\phi}^{s,u}(p,q)$,
the groupoid
$G_{\phi}^{s,u} \rtimes \Z^2$
has the topology
defined from the topology of the disjoint union
 $\sqcup_{(p,q)\in \Z^2}G_{\phi}^{s,u}(p,q)$.
We then have
\begin{proposition}
$G_{\phi}^{s,u} \rtimes \Z^2$ is an \'etale groupoid. 
\end{proposition}
\begin{proof}
We will show that  
the range map
$r :(x,p,q,y)\in G_{\phi}^{s,u} \rtimes \Z^2
\longrightarrow 
(x,0,0,x) \in (G_{\phi}^{s,u} \rtimes \Z^2)^\circ$
is a local homeomorphism.
Take an arbitrary point 
$(x,p,q,y) \in G_{\phi}^{s,u} \rtimes \Z^2.$
Since
$G_{\phi}^{s,u} \rtimes \Z^2 =\sqcup_{(p,q)\in \Z^2}G_{\phi}^{s,u}(p,q)$
and
$G_{\phi}^{s,u}(p,q) =
\cup_{n=0}^\infty G_\phi^{s,u,n}(p,q),$
we may assume that
$(x,y) $ belongs to $G_\phi^{s,u,N}(p,q)$ 
for some $N \in \Zp$,
so that 
$$
(\phi^p(x),y) \in G_\phi^{s,N}, \qquad
(\phi^q(x),y) \in G_\phi^{u,N}
$$
which imply that 
\begin{equation}
(\phi^{N+p}(x), \phi^N(y)) \in G_\phi^{s,0}, \qquad
(\phi^{-(N-q)}(x),\phi^{-N}(y)) \in G_\phi^{u,0} \label{eq:7.2.1}
\end{equation}
and
\begin{equation}
d(\phi^{N+p+n}(x), \phi^{N+n}(y)) < \epsilon_0,\qquad
d(\phi^{-(N-q+n)}(x),\phi^{-(N+n)}(y)) <\epsilon_0
\end{equation}
for all 
$n=0,1,2,\dots.$
Take $z \in X$ such that $d(x,z)$ is small enough
and
$d(\phi^N(y), \phi^{N+p}(z)) <\epsilon_0$,
so that 
$(\phi^N(y), \phi^{N+p}(z)) \in \Delta_{\epsilon_0}$.
Hence 
the point
$[\phi^N(y), \phi^{N+p}(z)]$ defines an element of $X$,
and we have an element
$\phi^{-N}([\phi^N(y), \phi^{N+p}(z)])$
in $X$.
Since we may assume that 
$[\phi^N(y), \phi^{N+p}(z)] \in X^u(\phi^N(y),\epsilon_0)$,
we have
\begin{align*}
d(y,\phi^{-N}([\phi^N(y), \phi^{N+p}(z)]))
= & d(\phi^{-N}(\phi^N(y)),\phi^{-N}([\phi^N(y), \phi^{N+p}(z)])) \\
< & \lambda_0^N d(\phi^N(y),[\phi^N(y), \phi^{N+p}(z)]) \\
< & \lambda_0^N \epsilon_0.
\end{align*}

Similarly we have an element
$[\phi^{-(N-q)}(z), \phi^{-N}(y)] \in X$
and
$\phi^N([\phi^{-(N-q)}(z), \phi^{-N}(y)] )\in X$
such that 
\begin{equation*}
d(y,\phi^N([\phi^{-(N-q)}(z), \phi^{-N}(y)] )) < \lambda_0^N \epsilon_0.
\end{equation*}
We may also assume that $\lambda_0^N  < \frac{1}{2}$
by taking $N$ large enough, so that
\begin{align*}
  & d(\phi^N([\phi^{-(N-q)}(z), \phi^{-N}(y)] ), \phi^{-N}([\phi^N(y), \phi^{N+p}(z)])) \\
< & 
d(y,\phi^N([\phi^{-(N-q)}(z), \phi^{-N}(y)] ))
+
d(y,\phi^{-N}([\phi^N(y), \phi^{N+p}(z)])) \\
< & 2 \lambda_0^N \epsilon_0 < \epsilon_0.
\end{align*}
Hence we have
\begin{equation*}
(\phi^N([\phi^{-(N-q)}(z), \phi^{-N}(y)] ), \phi^{-N}([\phi^N(y), \phi^{N+p}(z)]))
\in \Delta_{\epsilon_0}
\end{equation*}
so that the element
\begin{equation}
\gamma(z):=[\phi^N([\phi^{-(N-q)}(z), \phi^{-N}(y)] ), \phi^{-N}([\phi^N(y), \phi^{N+p}(z)])]
\label{eq:gammaz}
\end{equation}
is defined in $X$.
The map $\gamma$ is defined on a small neighborhood of $x$
and gives rise to a continuous map on the neighborhood. 
The conditions \eqref{eq:7.2.1}
imply
\begin{equation*}
[\phi^N(y),\phi^{N+p}(x)] =\phi^N(y), \qquad
[\phi^{-(N-q)}(x),\phi^{-N}(y)] =\phi^{-N}(y).
\end{equation*}
Hence 
for $z=x$ in \eqref{eq:gammaz}, we have
\begin{align*}
\gamma(x)
= & [\phi^N([\phi^{-(N-q)}(x), \phi^{-N}(y)] ), \phi^{-N}([\phi^N(y), \phi^{N+p}(x)])] \\
= & [\phi^N(\phi^{-N}(y)), \phi^{-N}(\phi^N(y))] \\
= & [y,y] = y.
\end{align*}

We will next show that $\gamma$ is injective.
Suppose that 
$\gamma(z) = \gamma(z')$ for $z, z'$ in a small neighborhood of $x$. 
Since
\begin{align}
\phi^{-N}(\gamma(z))
= &[[\phi^{-(N-q)}(z), \phi^{-N}(y)], \phi^{-2N}([\phi^N(y), \phi^{N+p}(z)])] \\
= &[\phi^{-(N-q)}(z),  \phi^{-2N}([\phi^N(y), \phi^{N+p}(z)])]  \label{eq:phigammaz}
\intertext{and similarly}
\phi^{-N}(\gamma(z'))
= &[\phi^{-(N-q)}(z'),  \phi^{-2N}([\phi^N(y), \phi^{N+p}(z')])]. \label{eq:phigammaz2}
\end{align}
We then have by \eqref{eq:phigammaz}, \eqref{eq:phigammaz2}
\begin{align*}
\phi^{-(N-q)}(z)
= &[[\phi^{-(N-q)}(z), \phi^{-2N}([\phi^N(y), \phi^{N+p}(z)])], \phi^{-(N-q)}(z)] \\
= &[\phi^{-N}(\gamma(z)), \phi^{-(N-q)}(z)] \\
= &[\phi^{-N}(\gamma(z')), \phi^{-(N-q)}(z)] \\
= &[[\phi^{-(N-q)}(z'), \phi^{-2N}([\phi^N(y), \phi^{N+p}(z')])], \phi^{-(N-q)}(z)]\\
= &[\phi^{-(N-q)}(z'), \phi^{-(N-q)}(z)]\\ 
\end{align*}
so that
\begin{equation}
\phi^{-(N-q)}(z)
= [\phi^{-(N-q)}(z'), \phi^{-(N-q)}(z)].  \label{eq:7.2.2}
\end{equation}
Since
$z, z'$ are in a small neighborhood of $x$,
we may assume that
\begin{equation}
d(\phi^{-(N-q)}(z'),\phi^{-(N-q)}(z)) < \epsilon_0. \label{eq:7.2.3}
\end{equation}
By \eqref{eq:7.2.2} and \eqref{eq:7.2.3},
we know 
\begin{equation*}
\phi^{-(N-q)}(z') \in X^u(\phi^{-(N-q)}(z), \epsilon_0)
\end{equation*}
so that 
\begin{equation}
d(\phi^{-(N-q+n)}(z'), \phi^{-(N-q+n)}(z)) <\epsilon_0
\quad
\text{ for all } 
n=0,1,2,\dots. \label{eq:7.2.4}
\end{equation}
As $z, z'$ are in a small neighborhood of $x$,
we may assume that 
\begin{equation*}
d(\phi^{-n}(x), \phi^{-n}(z)) <\frac{\epsilon_0}{2},\qquad
d(\phi^{-n}(x), \phi^{-n}(z')) <\frac{\epsilon_0}{2}
\quad
\text{ for all } 
n=0,1,\dots, N-q. 
\end{equation*}
Hence we have
\begin{equation}
d(\phi^{-n}(z'), \phi^{-n}(z)) <\epsilon_0
\quad
\text{ for all } 
n=0,1,\dots N-q. 
 \label{eq:7.2.5}
\end{equation}
By \eqref{eq:7.2.4} and \eqref{eq:7.2.5}, we obtain
\begin{equation*}
d(\phi^{-n}(z'), \phi^{-n}(z)) <\epsilon_0
\quad
\text{ for all } 
n=0,1,2,\dots
\end{equation*}
and hence
\begin{equation*}
z' \in X^u(z,\epsilon_0).
\end{equation*}

We similarly observe that
\begin{equation*}
d(\phi^{n}(z'), \phi^{n}(z)) <\epsilon_0 \quad
\text{ for all } 
n=0,1,2,\dots
\end{equation*}
and hence
\begin{equation*}
z' \in X^s(z,\epsilon_0)
\end{equation*}
so that
\begin{equation*}
z' \in X^u(z,\epsilon_0)\cap X^s(z,\epsilon_0)
\end{equation*}
and 
$z' = [z,z] = z$.
This shows that $\gamma$ is injective on a small neighborhood of $x$
and locally a homeomorphism
by the definition of $\gamma$.
As  a consequence, the groupoid  
$G_{\phi}^{s,u} \rtimes \Z^2$ is  \'etale.
\end{proof}

\begin{lemma}
The \'etale groupoid  
$G_{\phi}^{s,u} \rtimes \Z^2$ is amenable.
\end{lemma}
\begin{proof}
Consider the groupoid homomorphism 
$\eta: (x,p,q,y)\in G_{\phi}^{s,u} \rtimes \Z^2 \longrightarrow (p,q) \in \Z^2$.
The kernel is $G_\phi^s \cap G_\phi^u = G_\phi^a$
which is amenable by Lemma \ref{lem:amenable}.
Hence by \cite[Proposition 5.1.2]{AnanRenault},
we conclude that 
$G_{\phi}^{s,u} \rtimes \Z^2$ is amenable.
\end{proof}


\begin{definition}
A Smale space $(X,\phi)$ is said to be {\it (s,u)-essentially free}\/
if the interior of the set 
$\{x \in X \mid (\phi^p(x), x) \in G_\phi^s, (\phi^q(x), x) \in G_\phi^u \}$
is empty for each $(p,q)\in \Z\times \Z$ with $(p,q) \ne(0,0)$.
\end{definition}
The following lemma, which is kindly suggested by the referee, 
is proved in a similar way to Lemma \ref{lem:irreessfree} 
\begin{lemma}\label{lem:irresuessfree}
If $(X,\phi)$ is irreducible and $X$ is infinite,
then $(X,\phi)$ is (s,u)-essentially free.
\end{lemma}
\begin{proof}
Suppose that 
the set
$$
\interior\{x \in X \mid (\phi^p(x), x) \in G_\phi^s, (\phi^q(x), x) \in G_\phi^u \}
$$
contains a non-empty open set $U$ 
for a fixed  $(p,q)\in \Z\times \Z$ with $(p,q) \ne(0,0)$.
We may assume that $p \ne 0$.
Since
\begin{equation*}
\interior\{x \in X \mid (\phi^p(x), x) \in G_\phi^s, (\phi^q(x), x) \in G_\phi^u \}
\subset
\interior\{x \in X \mid (\phi^p(x), x) \in G_\phi^s \},
\end{equation*}
We have a non-empty open set $U$ 
such that 
\begin{equation*}
U
\subset
\interior\{x \in X \mid (\phi^p(x), x) \in G_\phi^s \}
\end{equation*}
for a fixed $p\ne 0$.
By the same argument as the proof of Lemma \ref{lem:irreessfree},
we have a contradiction, thus proving
$(X,\phi)$ is (s,u)-essentially free. 
\end{proof}

\begin{lemma}\label{lem:suessprin}
A Smale space $(X,\phi)$ is  (s,u)-essentially free
if and only if the \'etale groupoid 
$G_\phi^{s,u}\rtimes \Z^2$ is essentially principal.
\end{lemma}	
\begin{proof}
As we have
\begin{align*}
(G_\phi^{s,u}\rtimes \Z^2)'
= & \cup_{p,q \in \Z}\{(x,p,q,y) \in G_\phi^{s,u}\rtimes \Z^2 \mid x=y\}\\
= & \cup_{p,q \in \Z}\{(x,p,q,x) \in X \times \Z\times \Z \times X  
       \mid (\phi^p(x),x) \in G_\phi^s, (\phi^q(x),x) \in G_\phi^u\},
\end{align*}
the interior $\interior((G_\phi^{s,u}\rtimes \Z^2)')$
of $G_\phi^{s,u}\rtimes \Z^2$ is 
\begin{equation*}
\interior((G_\phi^{s,u}\rtimes \Z^2)')
=  \cup_{p,q \in \Z} \interior(\{(x,p,q,x) \in X \times \Z\times \Z \times X  
       \mid (\phi^p(x),x) \in G_\phi^s, (\phi^q(x),x) \in G_\phi^u\}).
\end{equation*}
For $p=q=0$,
we see that
$$
\interior(\{(x,0,0,x) \in X \times \Z\times \Z \times X  
       \mid (x,x) \in G_\phi^s, (x,x) \in G_\phi^u\})
= (G_\phi^{s,u}\rtimes\Z^2)^{\circ} =X.
$$
 Hence  
$\interior((G_\phi^{s,u}\rtimes \Z^2)') = (G_\phi^{s,u}\rtimes\Z^2)^{\circ}$ 
if and only if
the interior of 
$\{(x,p,q,x) \in X \times \Z\times \Z \times X  
       \mid (\phi^p(x),x) \in G_\phi^s, (\phi^q(x),x) \in G_\phi^u\}
$
is empty for all $p,q \in \Z$ except $p=q=0$.
This implies that 
$(X,\phi)$ is  (s,u)-essentially free
if and only if 
$G_\phi^{s,u}\rtimes \Z^2$ is essentially principal.
\end{proof}
\begin{definition}
The groupoid $C^*$-algebra 
$C^*(G_\phi^{s,u}\rtimes\Z^2)$ of the
\'etale amenable groupoid $G_\phi^{s,u}\rtimes\Z^2$
for a Smale space $(X,\phi)$ is called the {\it extended asymptotic Ruelle algebra}\/
or simply the {\it extended Ruelle algebra}\/ and written
$\R_\phi^{s,u}$.
\end{definition}
Since
$G_{\phi}^{s,u} \rtimes \Z^2$ is amenable,
the $C^*$-algebra
$\R_\phi^{s,u}$
is  identified with the reduced 
groupoid $C^*$-algebra
$C_r^*(G_{\phi}^{s,u} \rtimes \Z^2)$
on
$l^2(G_{\phi}^{s,u} \rtimes \Z^2)$
in a canonical way.

Similarly to Proposition \ref{prop:simple},
we obtain the following. 
\begin{proposition}\label{prop:susimple}
If a Smale space $(X,\phi)$ is irreducible
and $X$ is infinite,
then the $C^*$-algebra
$\R_\phi^{s,u}$ is simple.
\end{proposition}
We note that the above proposition also follows from 
\cite[Theorem 1.4]{PutSp} through Proposition \ref{prop:extendedtensor}
which will be shown later. 

Let $U_{z_1,z_2}, \, (z_1,z_2) \in 
\T^2 = \{ (z_1,z_2) \in \mathbb{C}\times \mathbb{C} \mid |z_i| =1 \}$ 
be an action of $\T^2$ to the  unitary group of
$B(l^2(G_{\phi}^{s,u} \rtimes \Z^2))$ defined by
$$
(U_{z_1,z_2} \xi)(x,p,q,y) =z_1^{p} z_2^{-q}\xi(x,p,q,y) \quad \text{ for} \quad
\xi \in l^2(G_{\phi}^{s,u} \rtimes \Z^2), (x,p,q,y) \in 
G_{\phi}^{s,u} \rtimes \Z^2.
$$
It is easy to see that the automorphisms
$Ad(U_{z_1,z_2})$
of 
$B(l^2(G_{\phi}^{s,u} \rtimes \Z^2))$
for 
$(z_1,z_2) \in \T^2 
$
leave
 $\R_\phi^{s,u}$  globally invariant.
They  give rise to an action of $\T^2$ on $\R_\phi^{s,u},$
denoted by $\rho_{\phi}^{s,u}$.
Let us denote by
$\delta^\phi_z=\rho_{\phi,(z, z)}^{s,u}, z \in \T$
the action of $\T$, called the diagonal action.
Recall that the asymptotic Ruelle algebra 
$\R_\phi^a$ is defined 
by
the groupoid $C^*$-algebra
$C^*(G_{\phi}^{a} \rtimes \Z)$
of the \'etale groupoid
$G_{\phi}^{a} \rtimes \Z$.
We then have 
\begin{theorem}\label{thm:7.3}
Assume that 
a Smale space $(X,\phi)$ is irreducible
and $X$ is infinite.
Then the fixed point algebra  
$(\R_\phi^{s,u})^{\delta^\phi}$
of
$\R_\phi^{s,u}$ under the diagonal action $\delta^\phi$
is isomorphic to the asymptotic Ruelle algebra
$\R_\phi^a$.
\end{theorem}
\begin{proof}
The 
\'etale groupoid
$G_{\phi}^{a} \rtimes \Z$
is identified with the subgroupoid 
\begin{equation}
 \{(x, p,p,y) \in X \times \Z\times \Z \times X \mid
(\phi^p(x), y) \in G_{\phi}^{s}, \, (\phi^p(x), y) \in G_{\phi}^{u} \}
\subset G_{\phi}^{s,u} \rtimes \Z^2
\end{equation}
of
$G_{\phi}^{s,u} \rtimes \Z^2$,
which is written $(G_{\phi}^{s,u} \rtimes \Z^2)^D$.
Since $(G_{\phi}^{s,u} \rtimes \Z^2)^D$
is clopen in $G_{\phi}^{s,u} \rtimes \Z^2$,
we have a natural inclusion relation
$C_c((G_{\phi}^{s,u} \rtimes \Z^2)^D)
\subset
C_c(G_{\phi}^{s,u} \rtimes \Z^2)$
of the algebras.
For $f \in C_c((G_{\phi}^{s,u} \rtimes \Z^2)^D),
$
we put
\begin{equation*}
{\mathcal{E}}_\phi(f) (x,p,q,y)
=
\begin{cases}
f(x,p,p,y) & \text{ if } p=q,\\
0          & \text{ if } p\ne q.
\end{cases}
\end{equation*}
Then 
${\mathcal{E}}_\phi$ 
defines a continuous linear map from 
$
C_c(G_{\phi}^{s,u} \rtimes \Z^2)
$
to
$C_c((G_{\phi}^{s,u} \rtimes \Z^2)^D)
$
and extends to 
$\R_\phi^{s,u}$ by the formula
\begin{equation*}
{\mathcal{E}}_\phi(f) 
= \int_{\mathbb{T}}\delta^{\phi}_z(f) dz \quad
\text{ for }
f \in \R_\phi^{s,u}
\end{equation*}
so that we have a conditional expectation from
$\R_\phi^{s,u}$ onto $\R_\phi^a$.
It is routine to check that
${\mathcal{E}}_\phi(\R_\phi^{s,u})$
is the fixed point algebra $(\R_\phi^{s,u})^{\delta^\phi}$
of $\R_\phi^{s,u}$ under the diagonal action $\delta^\phi$.
\end{proof}
The author would like to thank the referee who kindly suggested the following proposition
 to the author.
\begin{proposition}\label{prop:extendedtensor}
The extended Ruelle algebra $\R_\phi^{s,u}$
is stably isomorphic to the tensor product $\R^s_\phi \otimes \R^u_\phi$
between the stable Ruelle algebra $\R^s_\phi$ 
and the unstable Ruelle algebra $\R^u_\phi$.
\end{proposition}
\begin{proof}
It is easy to see that the correspondence
\begin{equation*}
(x,p,z) \times (y,q,w) 
\in (G_\phi^s\rtimes\Z)\times (G_\phi^u\rtimes\Z) \longrightarrow
((x,z)\times(y,w)), (p,q) )
\in 
 (G_\phi^s\times G_\phi^u)\rtimes\Z^2
\end{equation*}
yields an isomorphism of \'etale groupoids
between $ (G_\phi^s\rtimes\Z)\times (G_\phi^u\rtimes\Z)$
and
$ (G_\phi^s\times G_\phi^u)\rtimes\Z^2.$
Hence we have
\begin{equation*}
\R^s_\phi \otimes \R^u_\phi
=  C^*((G_\phi^s\rtimes\Z) \times (G_\phi^u\rtimes\Z))
\cong C^*((G_\phi^s \times G_\phi^u)\rtimes\Z^2).  
\end{equation*}
As in the proof of \cite[Theorem 3.1]{Putnam1},
the diagonal 
$\Delta = \{((x,z)\times(x,z),(p,q)) \in (G_\phi^s \times G_\phi^u)\rtimes\Z^2\}$
is an abstract transversal in the sense of Muhly--Renault-Williams \cite{MRW}.
Since the reduction of 
 $(G_\phi^s \times G_\phi^u)\rtimes\Z^2$ to $\Delta$ is clearly isomorphic to
$G_\phi^{s,u}\rtimes\Z^2$ 
as \'etale groupoids,
we see by \cite[Theorem 2.8]{MRW}
that 
$C^*(G_\phi^{s,u}\rtimes\Z^2)$ 
is stably isomorphic to 
$C^*((G_\phi^s \times G_\phi^u)\rtimes\Z^2)$, so that 
The extended Ruelle algebra $\R_\phi^{s,u}$
is stably isomorphic to the tensor product $\R^s_\phi \otimes \R^u_\phi$.
\end{proof}

\section{Asymptotic continuous orbit equivalence in topological Markov shifts}
In the first part of this section, we will deal with
topological Markov shifts, which are often called shifts of finite type,
 as examples of Smale spaces.
They have been studied by Ruelle, Putnam and Putnam-Spielberg, etc.
from the view point of Smale spaces.
The following description follows from Putnam's lecture note
\cite[Section 1]{Putnam2}.

Let
 $A =[A(i,j)]_{i,j=1}^N$ be an $N \times N$ matrix with entries $A(i,j)$ in $\{0,1\}$
for $i,j=1,\dots,N$ such that none of its rows or columns is zero.
We assume that $N \ge 2$ and the matrix $A$ is irreducible and not any permutation matrix.
Let us denote by
$\bar{X}_A$  the shift space of the two-sided topological Markov shift 
$(\bar{X}_A, \bar{\sigma}_A)$, which is defined by 
\begin{equation*}
\bar{X}_A =\{ (x_n)_{n \in \Z} \in \{1,\dots,N\}^{\Z} \mid A(x_n, x_{n+1}) =1 \text{ for all } n \in \Z\} 
\end{equation*}
with shift transformation $\bar{\sigma}_A$ defined by
$\bar{\sigma}_A((x_n)_{n\in \Z}) = (x_{n+1})_{n\in \Z}.$
We note that the assumption that $A$ is irreducible and not a permutation matrix
implies the shift space $\bar{X}_A$ is infinite and hence homeomorphic to a Cantor discontinuum.

Take and fix an arbitrary real number $\lambda_0$ with $0<\lambda_0<1$.
The space $\bar{X}_A$ is endowed with the following metric $d$ defined  by 
\begin{equation*}
d((x_n)_{n\in \Z}, (y_n)_{n\in \Z}) =
\begin{cases}
0 & \text{ if } (x_n)_{n\in \Z} = (y_n)_{n\in \Z}, \\
1 & \text{ if } x_0 \ne y_0,\\
(\lambda_0)^{k+1} & \text{ if } k = \Max\{|n| \mid x_i = y_i \text{ for all } i ; |i| \le n\}.
\end{cases}
\end{equation*}
With the above metric $d$,
 the space $\bar{X}_A$ is a compact Hausdorff space such that 
the topological dynamical system  $(\bar{X}_A, \bar{\sigma}_A)$
is called the two-sided topological Markov shift
defined by $A$. 
For $k \in \Zp$, we set
\begin{equation*}
B_k(\bar{X}_A)
=
\{(x_n)_{n=1}^k \in \{1,\dots,N\}^k \mid A(x_n, x_{n+1}) =1, n=1,\dots,k-1\}
\end{equation*}
and
$B_*(\bar{X}_A) = \cup_{k=0}^\infty B_k(\bar{X}_A),$
where 
$B_0(\bar{X}_A)$ denotes the empty word $\emptyset$.
Each member of $B_k(\bar{X}_A)$ is called an admissible word of length $k$.

 We will view the topological Markov shift
as a Smale space in the following way. 
Take $\epsilon_0 = 1$,
so that we have
$(x,y) \in \Delta_{\epsilon_0}$ if and only if 
$x_0 = y_0$.
Hence
the bracket $[x,y] =([x,y]_n)_{n \in \Z} \in \bar{X}_A$ for $(x,y) \in \Delta_{\epsilon_0}$ 
may be defined by
\begin{equation*}
[x,y]_n  = 
\begin{cases}
x_n & \text{ if } n \le 0, \\
y_n & \text{ if } n\ge 0.
\end{cases}
\end{equation*}
Since $x_0 = y_0$, 
$([x,y]_n)_{n \in \Z}$ defines an element of $\bar{X}_A.$
We then have
\begin{align*}
\bar{X}_A^s(x,\epsilon_0) 
&= \{y \in \bar{X}_A \mid y_n = x_n \text{ for } n=0,1,2,\dots \}, \\
\bar{X}_A^u(x,\epsilon_0)
&= \{y \in \bar{X}_A \mid y_n = x_n \text{ for } n=0,-1,-2,\dots \}. 
\end{align*}

As in  Putnam's lecture note
\cite[Section 1]{Putnam2},
the two-sided topological Markov shift
$(\bar{X}_A,\bar{\sigma}_A)$ with the above metric $d$
becomes  a Smale space for $\epsilon_0 =1$ and $\lambda_0$ itself.

We write for $ n=0,1,2,\dots$
\begin{equation*}
G_A^{s,n} = G_{\bar{\sigma}_A}^{s,n}, \qquad
G_A^{u,n} = G_{\bar{\sigma}_A}^{u,n}, \qquad
G_A^{a,n} = G_{\bar{\sigma}_A}^{a,n}.
\end{equation*}
Since
\begin{align*}
G_{A}^{s,0} &= \{ (x,y) \in \bar{X}_A\times \bar{X}_A \mid 
        y_i = x_i \text{ for all } i =0, 1, 2,\dots \},\\
G_{A}^{u,0} &= \{ (x,y) \in \bar{X}_A\times \bar{X}_A \mid 
        y_i = x_i \text{ for all } i =0, -1, -2,\dots \},\\
G_{A}^{a,0} &= G_{A}^{s,0} \cap G_{A}^{u,0} 
                = \{(x,y)  \in \bar{X}_A\times \bar{X}_A \mid x=y \},
\end{align*}
 we know
for $n=0,1,2,\dots$
\begin{align*}
G_{A}^{s,n} &= \{ (x,y) \in \bar{X}_A\times \bar{X}_A \mid 
        y_i = x_i \text{ for all } i =n, n+1,n+2,\dots \},\\
G_{A}^{u,n} &= \{ (x,y) \in \bar{X}_A\times \bar{X}_A \mid 
        y_i = x_i \text{ for all } i =-n, -n-1,-n-2,\dots \},\\
G_{A}^{a,n} &= G_{A}^{s,n} \cap G_{A}^{u,n}
                = \{ (x,y) \in \bar{X}_A\times \bar{X}_A \mid 
        y_i = x_i \text{ for all } |i| =n, n+1, n+2,\dots \}.
\end{align*}
All of them are given the relative topology of $\bar{X}_A\times \bar{X}_A$.
Each of them defines an equivalence relation on $\bar{X}_A$.
We set
\begin{equation*}
G_{A}^{s} = \cup_{n=0}^\infty G_{A}^{s,n},\qquad
G_{A}^{u} = \cup_{n=0}^\infty G_{A}^{u,n},\qquad
G_{A}^{a} = \cup_{n=0}^\infty G_{A}^{a,n},
\end{equation*}
and  they are endowed with the inductive limit topology, respectively.
Putnam have studied these three equivalence relations  
$G_{A}^{s}, \, G_{A}^{u} $ and $G_{A}^{a}$
on $\bar{X}_A$ by regarding them as topological groupoids.
He has  studied the associated groupoid $C^*$-algebras
$C^*(G_{A}^{s}), \, C^*(G_{A}^{u}) $ and $C^*(G_{A}^{a})$
which have been denoted by
$S(\bar{X}_A, \bar{\sigma}_A), \, U(\bar{X}_A, \bar{\sigma}_A)$
 and $A(\bar{X}_A, \bar{\sigma}_A)$, respectively.
He has pointed out that they are all stably AF-algebras.
He investigated their semi-direct products  
as groupoids
\begin{align*}
G_{A}^{s} \rtimes \Z= &\{(x,n,y) \in \bar{X}_A \times \Z \times \bar{X}_A \mid
(\bar{\sigma}_A^n(x), y) \in G_{A}^{s} \}, \\
G_{A}^{u} \rtimes \Z= &\{(x,n,y) \in \bar{X}_A \times \Z \times \bar{X}_A \mid
(\bar{\sigma}_A^n(x), y) \in G_{A}^{u} \}, \\
G_{A}^{a} \rtimes \Z= &\{(x,n,y) \in \bar{X}_A \times \Z \times \bar{X}_A \mid
(\bar{\sigma}_A^n(x), y) \in G_{A}^{a} \}.
\end{align*}
Putnam has  also deeply studied the associated groupoid $C^*$-algebras
$C^*(G_{A}^{s}\rtimes \Z), \, C^*(G_{A}^{u}\rtimes \Z) $ 
and $C^*(G_{A}^{a}\rtimes \Z)$
which have been written 
$R_s, \, R_u$
 and $R_a$, respectively
in his papers.
In this paper, we denote them 
by $\R_A^s, \, \R_A^u$
 and $\R_A^a$, respectively,
to emphasize the matrix $A$.
We note that the
irreducibility of 
the 
Smale space
$(\bar{X}_A,\bar{\sigma}_A)$ 
corresponds to the irreducibility of the matrix $A$,
and the condition that $\bar{X}_A$ is infinite
corresponds to the property that the matrix $A$ is not any permutation matrix.
  
\medskip

In the second part of this section, we study 
asymptotic continuous orbit equivalence defined for Smale spaces
in Section 3 focusing on  topological Markov shifts.

Let
$(\bar{X}_A,\bar{\sigma}_A)$
and
$(\bar{X}_B,\bar{\sigma}_B)$
be topological Markov shifts.
We will regard them as Smale spaces
and consider conditions under which they become asymptotic continuous orbit equivalence. 
\begin{lemma}\label{lem:7.1}
Each of the conditions (i) and (ii) in Remark \ref{remark3.3} are equivalent to the following conditions (i) and (ii), respectively.
 \begin{enumerate}
\renewcommand{\theenumi}{\roman{enumi}}
\renewcommand{\labelenumi}{\textup{(\theenumi)}}
\item
There exits a continuous function
$k_{1}: \bar{X}_A\longrightarrow \Zp$ such that 
\begin{gather}
( \bar{\sigma}_B^{k_{1}(x) +c_1(x)}(h(x)),  
\bar{\sigma}_B^{k_{1}(x)}(h(\bar{\sigma}_A(x))))\in G_B^{s,0},
\label{eq:ipprime1} \\
( \bar{\sigma}_B^{-k_{1}(x) +c_1(x)}(h(x)),  
\bar{\sigma}_B^{-k_{1}(x)}(h(\bar{\sigma}_A(x))))\in G_B^{u,0}.
\label{eq:ipprime2}
\end{gather}
\item
There exits a continuous function
$k_{2}: \bar{X}_B\longrightarrow \Zp$ such that 
\begin{gather}
( \bar{\sigma}_A^{k_{2}(y) +c_2(y)}(h^{-1}(y)),  
\bar{\sigma}_A^{k_{2}(y)}(h^{-1}(\bar{\sigma}_B(y))))\in G_A^{s,0},
\label{eq:iipprime1} \\
( \bar{\sigma}_A^{-k_{2}(y) +c_2(y)}(h^{-1}(y)),  
\bar{\sigma}_A^{-k_{2}(y)}(h^{-1}(\bar{\sigma}_B(y))))\in G_A^{u,0}. 
\label{eq:iipprime2}
\end{gather}
\end{enumerate}
\end{lemma}
\begin{proof}
(i)
We will prove that the equality \eqref{eq:ipprime1}
implies \eqref{eq:iprime1} by putting $k_{1,n}(x) =k_1^n(x).$
Suppose that
there exits a continuous function
$k_{1}: \bar{X}_A\longrightarrow \Zp$
satisfying 
the equality \eqref{eq:ipprime1}.
Since 
$$
G_{B}^{s,0} = \{ (x,y) \in \bar{X}_B\times \bar{X}_B \mid 
        y_i = x_i \text{ for all } i =0, 1, 2,\dots \},
$$
$
G_{B}^{s,0}
$ is an equivalence relation in $\bar{X}_B\times \bar{X}_B. $
In the equality \eqref{eq:ipprime1}, we have 
\begin{equation*}
\bar{\sigma}_B^{k_{1}(x)}(h(\bar{\sigma}_A(x)))
\in \bar{X}_B(\bar{\sigma}_B^{k_{1}(x) +c_1(x)}(h(x)), \epsilon_0)
\end{equation*} 
so that  by Lemma \ref{lem:2.1}, for any $m \in \N$,
\begin{equation}
\bar{\sigma}_B^m(\bar{\sigma}_B^{k_{1}(x)}(h(\bar{\sigma}_A(x))))
\in \bar{X}_B(\bar{\sigma}_B^m(\bar{\sigma}_B^{k_{1}(x) +c_1(x)}(h(x))), \epsilon_0)
\label{eq:8.4'}
\end{equation} 
and hence
\begin{equation*}
( \bar{\sigma}_B^{m+ k_{1}(x) +c_1(x)}(h(x)),  
\bar{\sigma}_B^{m+ k_{1}(x)}(h(\bar{\sigma}_A(x))))\in G_B^{s,0}.
\end{equation*}
Take $m =k_1(\bar{\sigma}_A(x)) + c_1(\bar{\sigma}_A(x))$ so that we have
\begin{equation*}
( \bar{\sigma}_B^{k_1(\bar{\sigma}_A(x)) + c_1(\bar{\sigma}_A(x))+ k_{1}(x) +c_1(x)}(h(x)),  
\bar{\sigma}_B^{k_1(\bar{\sigma}_A(x)) + c_1(\bar{\sigma}_A(x))+ k_{1}(x)}(h(\bar{\sigma}_A(x))))\in G_B^{s,0},
\end{equation*}
that is
\begin{equation}
( \bar{\sigma}_B^{k_{1}^2(x) +c_1^2(x)}(h(x)),  
\bar{\sigma}_B^{ c_1(\bar{\sigma}_A(x))+ k_{1}^2(x)}(h(\bar{\sigma}_A(x))))\in G_B^{s,0}.
\label{eq:ipprime1pp}
\end{equation}
By replacing $x$ with
$\bar{\sigma}_A(x)$ in the equality \eqref{eq:ipprime1} and \eqref{eq:8.4'},
 we have the following 
\begin{gather*}
( \bar{\sigma}_B^{k_{1}(\bar{\sigma}_A(x)) +c_1(\bar{\sigma}_A(x))}(h(\bar{\sigma}_A(x))),  
\bar{\sigma}_B^{k_{1}(\bar{\sigma}_A(x))}(h(\bar{\sigma}_A^2(x))))\in G_B^{s,0}, \\
\bar{\sigma}_B^m(\bar{\sigma}_B^{k_{1}(\bar{\sigma}_A(x))}(h(\bar{\sigma}_A^2(x))))
\in 
\bar{X}_B(\bar{\sigma}_B^m(
 \bar{\sigma}_B^{k_{1}(\bar{\sigma}_A(x)) +c_1(\bar{\sigma}_A(x))}(h(\bar{\sigma}_A(x)))),  
\epsilon_0),
\end{gather*}
so that 
\begin{equation*}
( \bar{\sigma}_B^{m+k_{1}(\bar{\sigma}_A(x)) +c_1(\bar{\sigma}_A(x))}(h(\bar{\sigma}_A(x))),  
\bar{\sigma}_B^{m+k_{1}(\bar{\sigma}_A(x))}(h(\bar{\sigma}_A^2(x))))\in G_B^{s,0}.
\end{equation*}
Take $m =k_1(x)$ so that we have
\begin{equation*}
( \bar{\sigma}_B^{k_1(x)+k_{1}(\bar{\sigma}_A(x)) +c_1(\bar{\sigma}_A(x))}(h(\bar{\sigma}_A(x))),  
\bar{\sigma}_B^{k_1(x)+k_{1}(\bar{\sigma}_A(x))}(h(\bar{\sigma}_A^2(x))))\in G_B^{s,0},
\end{equation*}
that is
\begin{equation}
( \bar{\sigma}_B^{k_1^2(x) +c_1(\bar{\sigma}_A(x))}(h(\bar{\sigma}_A(x))),  
\bar{\sigma}_B^{k_1^2(x)}(h(\bar{\sigma}_A^2(x))))\in G_B^{s,0}.
\label{eq:ipprime1ppp}
\end{equation}
By \eqref{eq:ipprime1pp} and \eqref{eq:ipprime1ppp},
we have
\begin{equation*}
( \bar{\sigma}_B^{k_{1}^2(x) +c_1^2(x)}(h(x)),
\bar{\sigma}_B^{k_1^2(x)}(h(\bar{\sigma}_A^2(x))))\in G_B^{s,0}.
\end{equation*}
This proves \eqref{eq:iprime1} for $n=2$.
We may inductively prove \eqref{eq:iprime1} for general $n$
in a similar fashion,
and we can see (i).
The other assertion (ii) is  shown  in a similar way to  (i).
\end{proof}

For $x =(x_n)_{n \in \Z} \in \bar{X}_A$,
we put
$$
x_- = (x_{-n})_{n=0}^\infty,\qquad
x_+ = (x_n)_{n=0}^\infty.
$$
Hence we have
$(x,z) \in G_{A}^{s,0}$ (resp. $(x,z) \in G_{A}^{u,0}$)
if and only if 
$x_+ = z_+$ (resp. 
$x_- = z_-$).


By Remark \ref{remark3.3} with Lemma \ref{lem:7.1},
we may reformulate 
asymptotic continuous orbit equivalence in topological Markov shifts
in the following way.
\begin{proposition}\label{prop:acoesft}
Topological Markov shifts
$(\bar{X}_A,\bar{\sigma}_A)$ and 
$(\bar{X}_B,\bar{\sigma}_B)$
are
asymptotically continuous orbit equivalent
if and only if
 there exist a homeomorphism
$h: \bar{X}_A\longrightarrow \bar{X}_B$,
continuous functions
$
c_1:\bar{X}_A\longrightarrow \Z, \,
c_2:\bar{X}_B\longrightarrow \Z,
$
and two-cocycle functions
$
d_1: G_A^a \longrightarrow \Z,\,
d_2: G_B^a \longrightarrow \Z,
$
such that 
\begin{enumerate}
\renewcommand{\theenumi}{\arabic{enumi}}
\renewcommand{\labelenumi}{\textup{(\theenumi)}}
\item
$c_1^m(x) + d_1(\bar{\sigma}_A^m(x), \bar{\sigma}_A^m(z))
=c_1^m(z) + d_1(x, z), \quad
(x,z) \in G_A^a, \, m \in \Z.$
\item
$c_2^m(y) + d_2(\bar{\sigma}_B^m(y), \bar{\sigma}_B^m(w))
=c_2^m(w) + d_2(y, w),\quad
(y, w) \in G_B^a,  m \in \Z.$
\end{enumerate}
and,

\begin{enumerate}
\renewcommand{\theenumi}{\roman{enumi}}
\renewcommand{\labelenumi}{\textup{(\theenumi)}}
\item
There exits a continuous function
$k_{1}: \bar{X}_A\longrightarrow \Zp$ such that 
\begin{gather*}
 \bar{\sigma}_B^{k_{1}(x) +c_1(x)}(h(x))_+ =    
\bar{\sigma}_B^{k_{1}(x)}(h(\bar{\sigma}_A(x)))_+,
\\
 \bar{\sigma}_B^{-k_{1}(x) +c_1(x)}(h(x))_- =  
\bar{\sigma}_B^{-k_{1}(x)}(h(\bar{\sigma}_A(x)))_-.
\end{gather*}
\item
There exits a continuous function
$k_{2}: \bar{X}_B\longrightarrow \Zp$ such that 
\begin{gather*}
\bar{\sigma}_A^{k_{2}(y) +c_2(y)}(h^{-1}(y))_+ =  
\bar{\sigma}_A^{k_{2}(y)}(h^{-1}(\bar{\sigma}_B(y)))_+,
\\
 \bar{\sigma}_A^{-k_{2}(y) +c_2(y)}(h^{-1}(y))_- =  
\bar{\sigma}_A^{-k_{2}(y)}(h^{-1}(\bar{\sigma}_B(y)))_-. 
\end{gather*}
\item
There exists a continuous function
$m_1: G_A^a\longrightarrow\Zp$ such that 
\begin{gather*}
\bar{\sigma}_B^{m_1(x,z) +d_1(x,z)}(h(x))_+ 
=
\bar{\sigma}_B^{m_1(x,z)}(h(z))_+ 
\quad \text{ for } (x,z) \in G_A^a, \\
\bar{\sigma}_B^{-m_1(x,z) +d_1(x,z)}(h(x))_-
=
\bar{\sigma}_B^{-m_1(x,z)}(h(z))_- 
\quad \text{ for } (x,z) \in G_A^a.
\end{gather*}
\item
There exists a continuous function
$m_2: G_B^a\longrightarrow\Zp$ such that 
\begin{gather*}
\bar{\sigma}_A^{m_2(y,w) +d_2(y,w)}(h^{-1}(y))_+
=
\bar{\sigma}_A^{m_2(y,w)}(h^{-1}(w))_+ 
\quad \text{ for } (y,w) \in G_B^a, \\
\bar{\sigma}_A^{-m_2(y,w) +d_2(y,w)}(h^{-1}(y))_-
=
\bar{\sigma}_A^{-m_2(y,w)}(h^{-1}(w))_- 
\quad \text{ for } (y,w) \in G_B^a.
\end{gather*}
\item
$c^{c^n_1(x)}_2(h(x)) 
+ d_2(\bar{\sigma}_B^{c_1^n(x)}(h(x)), h(\bar{\sigma}_A^n(x))) = n, \qquad x \in \bar{X}_A, 
\quad n \in \Z.$
\item
$c^{c^n_2(y)}_1(h^{-1}(y)) 
+ d_1(\bar{\sigma}_A^{c_2^n(y)}(h^{-1}(y)), h^{-1}(\bar{\sigma}_B^n(y))) = n, 
\qquad y\in \bar{X}_B, \quad n \in \Z.$
\item
$c^{d_1(x,z)}_2(h(x)) + d_2(\bar{\sigma}_B^{d_1(x,z)}(h(x)), h(z)) = 0, 
\qquad (x,z) \in G_A^{a}.$
\item
$c^{d_2(y,w)}_1(h^{-1}(y)) + d_1(\bar{\sigma}_A^{d_2(y,w)}(h^{-1}(y)), h^{-1}(w)) = 0, 
\qquad (y,w) \in G_B^{a}.$
\end{enumerate}
\end{proposition}



\section{Approach from Cuntz--Krieger algebras}
Let $A =[A(i,j)]_{i,j=1}^N$ be an irreducible square matrix with entries in $\{0,1\}$.
We assume that $A$ is not any permutation matrix.
Let
$\{ S_i \mid i=1,\dots,N \}$ be the canonical generating partial isometries 
of the Cuntz--Krieger algebra $\OA$ defined by the matrix $A$,
and similarly  
$\{ T_j \mid  j=1,\dots,N \}$ be the canonical generating partial isometries 
of the Cuntz--Krieger algebra $\OtA$ defined by the transposed matrix 
 $A^t$ of $A$ (\cite{CK}).
They are the universal unique $C^*$-algebras subject to the following operator relations, respectively
\begin{align*}
\sum_{j=1}^N S_j S_j^* = 1,& \qquad
S_i^* S_i =\sum_{j=1}^N A(i,j)S_j S_j^*, \qquad i=1,\dots,N,\\
\sum_{j=1}^N T_j T_j^* = 1,& \qquad
T_i^* T_i=\sum_{j=1}^N A^t(i,j)T_j T_j^*,
\qquad i=1,\dots,N.
\end{align*}
In the algebra $\OA$,
the automorphisms  $\rho^A_t \in \Aut(\OA), t \in \T={\mathbb{R}}/\Z$
defined by $\rho^A_t(S_i) = e^{2 \pi \sqrt{-1}t}S_i, i=1,\dots,N$
yield an action of $\T$ on $\OA$ is called the gauge action.
It is well-known that the fixed point algebra
$(\OA)^{\rho^A}$ of $\OA$ 
under the gauge action $\rho^A$ is an AF-algebra 
written $\FA$, whose maximal abelian $C^*$-subalgebra 
consisting of diagonal elements is written $\DA$.  
For an admissible word 
$\mu =(\mu_1,\dots,\mu_m) \in B_m(\bar{X}_A),$
we denote by $S_\mu$ the partial isometry
$S_{\mu_1}\cdots S_{\mu_m}$.
The $C^*$-algebra $\FA$
is generated by partial isometries  of the form
$S_\mu S_\nu^*$ for
$\mu,\nu \in B_m(\bar{X}_A), m=1,2,\dots$,
and 
the $C^*$-algebra $\DA$
is generated by projections of the form
$S_\mu S_\mu^*$ for
$\mu\in B_*(\bar{X}_A)$.
Let $X_A$ be the shift space of the right one-sided
topological Markov shift
$(X_A,\sigma_A)$,
which is defined by the compact Hausdorff space
$$
X_A = \{ (x_n)_{n \in \N} \in \{1,\dots,N\}^\N \mid
A(x_n, x_{n+1}) =1, \, n\in \N \}
$$
with shift transformation
$\sigma_A(  (x_n)_{n \in \N}) = (x_{n+1})_{n \in \N}.$
As in \cite[Section 7]{CK},  
the $C^*$-algebra $\DA$
is canonically isomorphic to the commutative $C^*$-algebra
$C(X_A)$ of all continuous functions on $X_A$.

We similarly write the partial isometry
$T_{\bar{\xi}} = T_{\xi_k}\cdots T_{\xi_1}$ 
for $\bar{\xi} = (\xi_k,\dots,\xi_1) \in B_k(\bar{X}_{A^t})$
and the $C^*$-subalgebras
$\FtA, \DtA$ of $\OtA$
for the transposed matrix $A^t$, respectively.

Let us consider the tensor product $C^*$-algebra
 $\OtA\otimes\OA.$
In the algebra $\OtA\otimes\OA$, we define the projections
\begin{equation*}
E_A = \sum_{j=1}^N T_j T_j^*\otimes S_j^* S_j,
\qquad
E_{A^t} = \sum_{j=1}^N  T_j^*T_j\otimes S_j S_j^*.
\end{equation*}
The projection $E_A$ has appeared in Kaminker--Putnam \cite[Section 4]{KamPut}
to study K-theoretic duality between $\OA$ and $\OtA$.
\begin{lemma}[{cf. \cite[Section 4]{KamPut}}]
\begin{equation}
E_A  =E_{A^t}.  \label{eq:EAETA}
\end{equation}
\end{lemma}
\begin{proof}
We have
\begin{align*}
E_A 
=& \sum_{i=1}^N T_i T_i^*\otimes S_i^* S_i 
= \sum_{i=1}^N T_i T_i^*\otimes (\sum_{j=1}^N A(i,j) S_j S_j^*)  
= \sum_{i=1}^N \sum_{j=1}^N A(i,j) T_iT_i^*\otimes S_j S_j^*,  \\
 E_{A^t}
=& \sum_{j=1}^N  T_j^*T_j \otimes S_j S_j^*
= \sum_{j=1}^N   (\sum_{i=1}^N A^t(j,i) T_iT_i^*)\otimes S_j S_j^*
= \sum_{i=1}^N 
     \sum_{j=1}^N A^t(j,i) T_iT_i^*\otimes S_j S_j^*,
\end{align*}
thus proving  \eqref{eq:EAETA}.
\end{proof}
\begin{definition}[{The extended Ruelle algebra for topological Markov shift}] 
We define the $C^*$-algebra $\RA$ by 
$$
\RA =E_A(\OtA\otimes\OA)E_A
$$
as a $C^*$-subalgebra of the tensor product $C^*$-algebra 
$\OtA\otimes\OA$.
\end{definition}
We also define 
$C^*$-subalgebras
\begin{equation*}
\FDA = E_A(\DtA\otimes\DA)E_A, \qquad
\FFA = E_A(\FtA\otimes\FA)E_A.
\end{equation*}
Therefore we have $C^*$-subalgebras of $\R_A^{s,u}$
$$
\FDA \subset \FFA\subset \RA.
$$

For an admissible word 
$\xi =(\xi_1,\dots,\xi_k)\in B_k(\bar{X}_A)$,
we denote by
$\bar{\xi}$
the admissible word 
$(\xi_k,\dots,\xi_1)$
in $\bar{X}_{A^t}$ obtained by reversing 
the symbols  of the word
$(\xi_1,\dots,\xi_k)$.
\begin{lemma}
For $\mu =(\mu_1,\dots,\mu_m), \nu =(\nu_1,\dots,\nu_n) \in B_*(\bar{X}_A)$   
and
$\bar{\xi} = (\xi_k,\dots,\xi_1),
\bar{\eta} = (\eta_l,\dots,\eta_1)
\in B_*(\bar{X}_{A^t}),$
the following two conditions are equivalent:
\begin{enumerate}
\renewcommand{\theenumi}{\roman{enumi}}
\renewcommand{\labelenumi}{\textup{(\theenumi)}}
\item
$E_A( T_{\bar{\xi}}T_{\bar{\eta}}^*\otimes S_\mu S_\nu^* )E_A =
 T_{\bar{\xi}}T_{\bar{\eta}}^* \otimes S_\mu S_\nu^*. 
$
\item
$A(\xi_k,\mu_1) =A(\eta_l,\nu_1) = 1.$ 
\end{enumerate}
\end{lemma}
\begin{proof}
We have the following equalities:
\begin{align*}
E_A(T_{\bar{\xi}}T_{\bar{\eta}}^*\otimes S_\mu S_\nu^*)
=& \sum_{i=1}^N  T_i^*T_i T_{\bar{\xi}}T_{\bar{\eta}}^* \otimes S_i S_i^* S_\mu S_\nu^* \\
=&  T_{\mu_1}^*T_{\mu_1} T_{\bar{\xi}}T_{\bar{\eta}}^* 
\otimes S_{\mu_1} S_{\mu_1}^* S_\mu S_\nu^* \\
=&  A^t(\mu_1,\xi_k)T_{\bar{\xi}}T_{\bar{\eta}}^*\otimes S_\mu S_\nu^*  \\
=&  A(\xi_k, \mu_1)T_{\bar{\xi}}T_{\bar{\eta}}^*\otimes S_\mu S_\nu^*.
\end{align*}
Similarly we have  
\begin{equation*}
(T_{\bar{\xi}}T_{\bar{\eta}}^*\otimes S_\mu S_\nu^*) E_A
= A(\eta_l,\nu_1) T_{\bar{\xi}}T_{\bar{\eta}}^*\otimes S_\mu S_\nu^*.
\end{equation*}
Hence 
the equality
$E_A(T_{\bar{\xi}}T_{\bar{\eta}}^*\otimes S_\mu S_\nu^* )E_A =
T_{\bar{\xi}}T_{\bar{\eta}}^*\otimes S_\mu S_\nu^*
$
holds
if and only if 
$A(\xi_k,\mu_1) =A(\eta_l,\nu_1) = 1.$
\end{proof}
Let us denote by
$\RAC$ the $*$-subalgebra of $\RA$ linearly spanned  
by the operators of the form
\begin{equation}
T_{\bar{\xi}}T_{\bar{\eta}}^* \otimes S_\mu S_\nu^*  
\quad
\text{ for }
\quad
A(\xi_k,\mu_1) =A(\eta_l,\nu_1) = 1 \label{eq:10.2}
\end{equation}
where
 $\mu =(\mu_1,\dots,\mu_m), \nu =(\nu_1,\dots,\nu_n) \in B_*(\bar{X}_A)$   
and
$\bar{\xi} = (\xi_k,\dots,\xi_1),
\bar{\eta} = (\eta_l,\dots,\eta_1)
\in B_*(\bar{X}_{A^t}).$

\begin{lemma}
$\RAC$ is dense in $\RA$. 
\end{lemma}
\begin{proof}
Let ${\mathcal{P}}_A$ be the $*$-algebra linearly spanned by
the operators of the form
$S_\mu S_\nu^*$ for $\mu, \nu \in B_*(\bar{X}_A)$.
As in \cite[Section 2]{CK}, the algebra 
 ${\mathcal{P}}_A$
becomes a dense $*$-subalgebra of $\OA$.
We denote by 
${\mathcal{P}}_{A^t} \otimes {\mathcal{P}}_A$
the linear span of elements 
\begin{equation*}
T_{\bar{\xi}}T_{\bar{\eta}}^* \otimes S_\mu S_\nu^* 
\quad
\text{ for }
\mu, \nu \in B_*(\bar{X}_A),
\,
\bar{\xi}, \bar{\eta}
\in B_*(\bar{X}_{A^t}).
\end{equation*}
It becomes a dense $*$-subalgebra of the $C^*$-algebra of tensor products
$\OtA \otimes \OA$.
For any $Y \in \RA \subset\OtA \otimes \OA$,
take $Y_n \in {\mathcal{P}}_{A^t}\otimes {\mathcal{P}}_A$ 
such that 
$\| Y - Y_n \| \longrightarrow 0$ as $n \longrightarrow \infty$.
Since
$$
\| Y - E_A Y_n E_A  \| 
=  
\|E_A Y E_A  - E_A Y_n E_A  \|
\le \| Y -  Y_n  \|
\longrightarrow 0
$$ 
as $n \longrightarrow \infty$,
and 
$ E_A Y_n E_A $ belongs to $\RAC$,
we conclude that 
$\RAC$ is dense in $\RA$.
\end{proof}
\begin{lemma}
$\FDA $ is canonically isomorphic to $C(\bar{X}_A)$.
\end{lemma}
\begin{proof}
For 
 $\mu =(\mu_1,\dots,\mu_m), \xi =(\xi_1,\dots,\xi_k) \in B_*(\bar{X}_A)$
with $A(\xi_k,\mu_1)=1$,
denote by
$\xi \mu$ the admissible word
$(\xi_1,\dots,\xi_k,\mu_1,\dots,\mu_m)\in B_*(\bar{X}_A)$.
Let 
$U_{\xi\mu}$ be the cylinder set of $\bar{X}_A$
defined by
$$
U_{\xi\mu} =\{ (x_n)_{n \in \Z} \in \bar{X}_A \mid
x_{-(k-1)} =\xi_1,\dots,x_{-1} =\xi_{k-1}, x_0 = \xi_k,\, x_1 =\mu_1,\dots,x_m = \mu_m\}.
$$
Since
$\FDA = E_A(\DtA\otimes\DA)E_A$
and
$$
\DA = C^*(S_\mu S_\mu^*\mid \mu \in B_*(X_A)),\qquad
\DtA = C^*(T_{\bar{\xi}} T_{\bar{\xi}}^* \mid \xi \in B_*(X_{A^t})),
$$ 
it is straightforward to see that the correspondence
$$
T_{\bar{\xi}} T_{\bar{\xi}}^*\otimes S_\mu S_\mu^* \in \FDA
\longrightarrow 
\chi_{U_{\xi\mu}} \in C(\bar{X}_A)
$$
yields an isomorphism
between 
$\FDA$ and $C(\bar{X}_A).$
\end{proof}
Consider the automorphisms
$\gamma^A_{(r,s)}= \rho^{A^t}_r\otimes\rho^A_s, \, (r,s) \in \T^2$
on
$\OtA \otimes \OA$
for the gauge actions
$\rho^{A^t}$ on $\OtA$ and $\rho^A$ on $\OA$.
Since $\gamma^A_{(r,s)}(E_A) = E_A$,
we have an action
 $\gamma^A$ of $\T^2$ on $\R^{s,u}_A.$ 
The diagonal action 
$\delta^A_t, t \in \T$ on  $\R^{s,u}_A$
is defined by 
$\delta^A_t =\gamma^A_{(t,t)}, t \in \T.$
On the other hand, the groupoid $C^*$-algebra
$\R_{\bar{\sigma}_A}^{s,u} =C^*(G_A^{s,u}\rtimes\Z^2)$
of the \'etale amenable groupoid $G_A^{s,u}\rtimes\Z^2$
has an action   $\rho^{s,u}_{\bar{\sigma}_A}$
of $\T^2$ defined in the paragraph right before 
Theorem \ref{thm:7.3}.
Its diagonal action $\delta^{\bar{\sigma}_A}$
of $\T$ on
$\R_{\bar{\sigma}_A}^{s,u}$
is defined by
$\delta^{\bar{\sigma}_A}_t =\rho^{s,u}_{\bar{\sigma}_A,(t,t)}.$
Its fixed point algebra
$(\R_{\bar{\sigma}_A}^{s,u})^{\delta^{\bar{\sigma}_A}}$
is isomorphic to the asymptotic Ruelle algebra
$\R_{\bar{\sigma}_A}^{a}$ written $\R_A^a$.
For the structure of the algebra $\RA$, we have
\begin{theorem}\label{thm:RAGR}
Let $A$ be an irreducible and not permutation matrix with entries in $\{0,1\}$. 
Then the $C^*$-algebra $\RA$ 
is a unital, simple, purely infinite, nuclear $C^*$-algebra 
  isomorphic to the  groupoid $C^*$-algebra 
$\R_{\bar{\sigma}_A}^{s,u}$
of the \'etale groupoid $G_A^{s,u}\rtimes\Z^2$.
More precisely, there exists an isomorphism
$\Phi: \RA\longrightarrow \R_{\bar{\sigma}_A}^{s,u}$
of $C^*$-algebras such that 
\begin{equation}
\Phi(\FDA) = C(\bar{X}_A)
\quad
\text{ and }
\quad
\Phi\circ\gamma^A_{(r,s)} = \rho^{s,u}_{\bar{\sigma}_A,(r,s)}\circ\Phi, \quad
(r,s) \in \T^2. \label{eq:Phi9.61}
\end{equation}
In particular, 
we have
$
\Phi\circ\delta^A_t = \delta^{\bar{\sigma}_A}_t\circ\Phi
$ for $t \in \T.$ 
\end{theorem}
\begin{proof}
Since $A$ is irreducible and not permutation matrix,
the Cuntz--Krieger algebras
$\OA, \OtA$ are both unital, simple, purely infinite
and nuclear (\cite[Theorem 2.14]{CK}). 
Hence so is the algebra
$E_A(\OtA\otimes\OA)E_A=\RA$.
We will construct an isomorphism 
$\Phi: \RA\longrightarrow \R_{\bar{\sigma}_A}^{s,u}$
having the desired properties \eqref{eq:Phi9.61}.
As in \cite{MMKyoto}, \cite{Renault}, \cite{Renault2}, \cite{Renault3}, 
the right one-sided
topological Markov shift
$(X_A,\sigma_A)$ 
gives rise to  an \'etale groupoid
$G_A$, which is defined by 
\begin{equation*}
G_A =\{
((x_i)_{i=1}^\infty, n, (y_j)_{j=1}^\infty) \in X_A\times \Z\times X_A \mid
n= l-k, x_{i+k} = y_{i+l}, i=1,2,\dots \}.
\end{equation*}
We have the groupoid 
$G_{A^t}$ for the transposed matrix $A^t$ in a similar way.
It is wel-known that 
the groupoids $G_A, G_{A^t}$ are amenable and \'etale such that 
their $C^*$-algebras 
$C^*(G_A), C^*(G_{A^t})$ are isomorphic to the 
the Cuntz--Krieger algebras
$\OA, \OtA$, respectively.
Let $G_{A^t} \times G_A$ be the direct product of the groupoids
so that 
$C^*(G_{A^t} \times G_A)$
is isomorphic to the tensor product
$C^*(G_{A^t})\otimes C^*(G_A)$
of the groupoid $C^*$-algebras.
Hence we have a natural isomorphism
$\Phi: \OtA\otimes\OA\longrightarrow C^*(G_{A^t} \times G_A).$
For elements
\begin{gather*}
((x_i)_{i=1}^\infty, n, (y_i)_{i=1}^\infty) \in G_A \text{ with }
n=l-k, \, x_{i+k} = y_{i+l} \text{ for } i\in \N, \\
((x'_j)_{j=1}^\infty, n', (y'_j)_{j=1}^\infty) \in G_A \text{ with }
n'=l'-k', \, x'_{j+k'} = y'_{j+l'} \text{ for } j\in \N
\end{gather*}
of the groupoid $G_A$,
we assume that
$A(x'_1, x_1) = A(y'_1,y_1) =1$.
Put 
$x =(x_i)_{i=1}^\infty, y=(y_i)_{i=1}^\infty$
and
$x' =(x'_j)_{j=1}^\infty, y'=(y'_j)_{j=1}^\infty.$
We define a bi-infinite sequence
$\pi(x',x)=(\pi(x',x)_i)_{i\in \Z}$ by setting
\begin{equation*}
\pi(x',x)_i
= 
\begin{cases}
x_i & \text{ if } i\ge 1, \\
x'_{-i+1} & \text{ if } i\le 0.
\end{cases}
\end{equation*}
Then  
$\pi(x',x)$ and similarly 
$\pi(y',y)$ belong to $\bar{X}_A$.
Put $N = \Max\{l+1, l'\}$ and
$p=-n,q = n'$.
Since 
\begin{align*}
\bar{\sigma}_A^p(\pi(x',x))_i 
=& \pi(y',y)_i, \qquad i\ge N, \\  
\bar{\sigma}_A^q(\pi(x',x))_i 
=& \pi(y',y)_i, \qquad i\le -N, 
\end{align*}
we have
\begin{equation*}
(\pi(x',x), p,q,\pi(y',y)) \in G_A^{s,u}\rtimes\Z^2.
\end{equation*}
Define the subgroupoid
$G_{A^t}\times_{A}G_A$ 
of $G_{A^t}\times G_A$
by
\begin{equation*}
G_{A^t}\times_{A}G_A =
\{((x',n',y'),(x,n,y)) \in G_{A^t} \times G_A \mid
A(x'_1,x_1)=A(y'_1,y_1) = 1\}.
\end{equation*}
It is easy to see that the correspondence
\begin{equation*}
((x',n',y'),(x,n,y)) \in G_{A^t}\times_{A}G_A 
\longrightarrow
(\pi(x',x), -n, n',\pi(y',y)) \in G_A^{s,u}\rtimes\Z^2
\end{equation*}
yields an isomorphism of \'etale groupoids, so that 
we may identify
$G_{A^t}\times_{A}G_A$ and $G_A^{s,u}\rtimes\Z^2$
as \'etale groupoids through the above correspondence.
Since
$G_{A^t}\times_{A}G_A$
is a clopen subset of $G_{A^t}\times G_A,$
the characteristic function
$\chi_{G_{A^t}\times_{A}G_A}$ of
$G_{A^t}\times_{A}G_A$
on
$G_{A^t}\times G_A$
belongs to the $C^*$-algebra
$C^*(G_{A^t} \times G_A)$,
which is denote by $P_A$.
It then follows that 
the isomorphism
$\Phi: \OtA\otimes\OA\longrightarrow C^*(G_{A^t} \times G_A)$
satisfies
$\Phi(E_A) = P_A$.
Hence the restriction of $\Phi$
to the subalgebra
$E_A(\OtA\otimes\OA)E_A$
gives rise to an isomorphism
$E_A(\OtA\otimes\OA)E_A\longrightarrow P_A C^*(G_{A^t} \times G_A)P_A$
which is still denoted by $\Phi$.
As 
 $P_A C^*(G_{A^t} \times G_A)P_A$
 is identified with
 $C^*(G_A^{s,u}\rtimes\Z^2)$,
 we have an isomorphism
$\Phi: \RA \longrightarrow \R_{\bar{\sigma}_A}^{s,u}.$
It is also described in the following way. 
For 
 $\mu =(\mu_1,\dots,\mu_m), \nu =(\nu_1,\dots,\nu_n) \in B_*(\bar{X}_A)$   
and
$\bar{\xi} = (\xi_k,\dots,\xi_1),
\bar{\eta} = (\eta_l,\dots,\eta_1)
\in B_*(\bar{X}_{A^t})$
with 
$A(\xi_k,\mu_1) = A(\eta_l,\nu_1) =1$,
we know that
\begin{equation*}
(\phi^{m-n}(x),y) \in G_A^{s,|m-n|},
\qquad
(\phi^{l-k}(x),y) \in G_A^{u,|l-k|} 
\quad \text{ for } 
x \in U_{\xi\mu}, \, y \in U_{\eta\nu}.
\end{equation*}
Let
$\chi_{\xi\mu,\eta\nu}
\in C_c(G_A^{s,u}\rtimes\Z^2)
$
be the characteristic function
of the clopen set
\begin{align*}
U_{\xi\mu,\eta\nu}
=\{& (x,m-n, l-k,y) \in G_A^{s,u}\rtimes\Z^2 \mid 
x \in U_{\xi\mu},  y \in U_{\eta\nu}, \\
& (\bar{\sigma}_A^m(x),\bar{\sigma}_A^n(y)) \in G_A^{s,0},
(\bar{\sigma}_A^{-k}(x),\bar{\sigma}_A^{-l}(y)) \in G_A^{u,0}
\}.
\end{align*}
It is not difficult to see that the correspondence
\begin{equation}
T_{\bar{\xi}} T_{\bar{\eta}}^* \otimes S_\mu S_\nu^* \in \RA 
\longrightarrow 
\chi_{\xi\mu,\eta\nu} \in C_c(G_A^{s,u}\rtimes\Z^2) \label{eq:Phi9.62}
\end{equation}
gives rise to the 
isomorphism $\Phi: \RA\longrightarrow 
C^*(G_A^{s,u}\rtimes\Z^2)
(=\R_{\bar{\sigma}_A}^{s,u}).$
By \eqref{eq:Phi9.62},
we easily know that
$\Phi$ satisfies 
\eqref{eq:Phi9.61}.
\end{proof}

\begin{corollary}\label{cor:mains10.9}
The fixed point algebra $(\RA)^{\delta^A}$
of $\RA$ under the diagonal gauge action $\delta^A$
is isomorphic to the asymptotic Ruelle algebra $\R_A^a$.
\end{corollary}
\begin{proof}
The fixed point algebra
$(\R_{\bar{\sigma}_A}^{s,u})^{\delta^{\bar{\sigma}_A}}$
of 
$\R_{\bar{\sigma}_A}^{s,u}$
under
$
\delta^{\bar{\sigma}_A}
$ 
is isomorphic 
to the asymptotic Ruelle algebra
$\R_A^a$ by Theorem \ref{thm:7.3}.
Hence the assertion follows from Theorem \ref{thm:RAGR}.
\end{proof}
%
Put
$U_i =  T_i^*\otimes S_i $ in $\OtA \otimes \OA$
for $i=1,\dots,N$.
We set
$U_A = \sum_{i=1}^N U_i$ in 
$\OtA \otimes \OA$.
\begin{lemma}
$U_A$ is a unitary in $\RA$, that is,
$U_A U_A^* = U_A^* U_A =E_A$. 
\end{lemma}
\begin{proof}
We have 
\begin{equation*}
E_A U_i 
 = (\sum_{j=1}^N T_j^*T_j \otimes S_j S_j^*)U_i 
 = \sum_{j=1}^N  T_j^* T_jT_i^*\otimes S_j S_j^* S_i  
 =   T_i^*\otimes S_i  
\end{equation*}
and similarly
$U_i E_A = U_i$,
so that 
we have $U_i \in \RA$.
Since we have
$U_i U_i^* =  T_i^*T_i\otimes S_i S_i^* $
and
$U_i^* U_i =  T_i T_i^*\otimes S_i^* S_i$,
we see that
\begin{equation*}
U_i U_i^* \cdot U_j U_j^* =
U_i^* U_i \cdot U_j^* U_j = 0 \quad \text{ if } i \ne j. 
\end{equation*}
It then follows that
\begin{equation*}
U_A^* U_A = \sum_{i=1}^N U_i^* U_i 
= \sum_{i=1}^N  T_i T_i^*\otimes  S_i^* S_i 
=E_A.
\end{equation*}
We have 
$U_A U_A ^*= E_A$
similarly.
\end{proof}
Define the inner automorphism $\alpha_A$ of $\RA$ by setting
$\alpha_A = \Ad(U_A)$.
\begin{proposition}
Let $\Phi: \RA\longrightarrow 
\R_{\bar{\sigma}_A}^{s,u}(=C^*(G_A^{s,u}\rtimes\Z^2))$
be the isomorphism defined in Theorem \ref{thm:RAGR}.
Then the restriction $\Phi|_{\FDA}: \FDA\rightarrow C(\bar{X}_A)$
of $\Phi$ to the commutative $C^*$-subalgebra 
$\FDA$ satisfies the relation:
\begin{equation*}
\Phi\circ \alpha_A = \bar{\sigma}_A^*\circ \Phi
\end{equation*}
where 
$\bar{\sigma}_A^*(f) = f\circ\bar{\sigma}_A$
 for $f \in C(\bar{X}_A)$.
\end{proposition}
\begin{proof}
For 
 $\mu =(\mu_1,\dots,\mu_m), \xi =(\xi_1,\dots,\xi_k) \in B_*(\bar{X}_A)$
with $A(\xi_k, \mu_1)=1$,
We have
\begin{align*}
U_A( T_{\bar{\xi}} T_{\bar{\xi}}^*\otimes S_\mu S_\mu^* ) U_A^*
=& 
\sum_{i,j=1}^N  T_i^* T_{\bar{\xi}} T_{\bar{\xi}}^* T_j\otimes S_i S_\mu S_\mu^* S_j^*  \\
=&  T_{\xi_1}^* T_{\xi_1} 
T_{(\xi_{k-1},\dots,\xi_1)} T_{(\xi_{k-1},\dots,\xi_1)}^* 
T_{\xi_1}^* T_{\xi_1}\otimes S_{\xi_k\mu} S_{\xi_k\mu}^*  \\
=&    
T_{(\xi_{k-1},\dots,\xi_1)} T_{(\xi_{k-1},\dots,\xi_1)}^*
\otimes S_{\xi_k\mu} S_{\xi_k\mu}^*. 
\end{align*}
This shows that the equality 
$\Phi\circ \alpha_A = \bar{\sigma}_A^*\circ \Phi$
holds on $\FDA$.
\end{proof}
We note that 
the unitary $\Phi(U_A)$ in $\R_{\bar{\sigma}_A}^{s,u}$
belongs to the asymptotic Ruelle algebra $\R_{\bar{\sigma}_A}^a$
and it is nothing but the unitary
$U_{\bar{\sigma}_A}$ for $(X,\phi) = (\bar{X}_A,\bar{\sigma}_A)$
defined in \eqref{eq:Uphi}.

\begin{remark}
In \cite[Proposition 6.7]{Holton},
C. G. Holton proved that if two primitive matrices $A$ and $B$ are shift equivalent
(cf. \cite{LM}), 
then the asymptotic Ruelle algebras 
$\R_A^a$ and $\R_B^a$ are isomorphic by showing that 
 the automorphism $\alpha_A$ induced by the original transformation $\bar{\sigma}_A$
 on the AF-algebra $C^*(G_A^a)$ has the Rohlin property.
\end{remark}

\section{K-theory for the  asymptotic Ruelle algebras for full shifts}
In this final section,
we will compute  the K-groups and its trace values  of the asymptotic 
Ruelle algebras $\R_A^a$
for some topological Markov shifts.
In \cite{Putnam1}(cf. \cite{KilPut}), the K-theory formula for the  asymptotic 
Ruelle algebras $\R_A^a$ for the topological Markov shift
$(\bar{X}_A,\bar{\sigma}_A)$
has been provided.
In particular, ring and module structure of the K-groups have been 
deeply studied in \cite{KilPut}.
We will see the K-groups of the $C^*$-algebra
$\R_A^a$ in a concrete way for full shifts by using the Putnam's formula in \cite{Putnam1} 
which we will describe below.
Let $A$ be an $N\times N$ irreducible matrix  with entries in $\{0,1\}$.
Let us consider the abelian group
$H(A)$ of the inductive limit
\begin{equation}
\Z^N\otimes\Z^N \quad \overset{A^t\otimes A}{\longrightarrow} \quad
\Z^N\otimes\Z^N \quad \overset{A^t\otimes A}{\longrightarrow}\quad\cdots.
\label{eq:inductiveA}
\end{equation}
Under a natural identification
between $\Z^N\otimes\Z^N $ and
the $N\times N$ matrices $M_N(\Z)$ over $\Z$, 
we set
$H_k(A) =M_N(\Z)$ for $k=1,2,\dots.$
Then the map $A^t\otimes A$ in \eqref{eq:inductiveA}
goes to the map
$\iota_k:H_k(A)\longrightarrow H_{k+1}(A)$ 
defined by
$\iota_k([T,k]) = [ATA,k+1]$ for $[T,k] \in H_k(A)$ 
with $T \in M_N(\Z)$.
Define the homomorphism
$\alpha_k:H_k(A)\longrightarrow H_{k+1}(A)$  by
$\alpha_k([T,k]) = [A^2T,k+1]$ for $[T,k] \in H_k(A)$,
which extends to an endomorphism $\alpha:H(A)\longrightarrow H(A)$.
Putnam showed the following K-theory formula by using the six-term exact sequence
for K-theory of the $C^*$-algebra $\R_A^a$: 
\begin{proposition}[{Putnam \cite[p. 192]{Putnam1}}]\label{prop:Kformula}
\begin{align*}
K_0(\R_A^a) = & \Coker(\id -\alpha: H(A)\longrightarrow H(A)),\\
K_1(\R_A^a) = & \Ker(\id -\alpha: H(A)\longrightarrow H(A)).
\end{align*}
\end{proposition}
We will compute the groups
$K_*(\R_A^a)$ for the $N\times N$ matrix
$
A=
\begin{bmatrix} 
1& \cdots & 1 \\
\vdots &   & \vdots \\
1& \cdots & 1 \\
\end{bmatrix}
$
with all entries being $1$'s,
so that the topological Markov shift
$(\bar{X}_A,\bar{\sigma}_A)$
is the full $N$-shift
written 
$(\bar{X}_N,\bar{\sigma}_N)$.
Let us denote by
$\R_N^a$ the asymptotic Ruelle algebra $\R_A^a$ 
for the matrix $A$.
For a natural number $n$,  
$\Z[\frac{1}{n}]$ means the subgroup
$\{ \frac{m}{n^k} \in \mathbb{R}\mid m,k\in \Z\}$ of the additive group $\mathbb{R}.$
We  provide the following lemma.
\begin{lemma}
There exists an isomorphism
$\xi: H(A) \longrightarrow \Z[\frac{1}{N^2}]$ of abelian groups
 such that the diagram
$$
\begin{CD}
H(A) @>\alpha >> H(A) \\
@V{\xi }VV  @VV{\xi}V \\
\Z[\frac{1}{N^2}] @> \id >> \Z[\frac{1}{N^2}] 
\end{CD}
$$
is commutative.
Hence $\alpha = \id $ on $H(A)$.
\end{lemma}
\begin{proof}
For a matrix $T =[t_{ij}]_{i,j=1}^N\in M_N(\Z)$, define
$s_N(T) = \sum_{i,j=1}^N t_{ij}.$
As $ATA = s_N(T)A,$
the map
$s_N: M_N(\Z)(=H_k(A)) \longrightarrow \Z$
defines a homomorphism
such that 
$
\iota_k([T,k]) =[s_N(T)A, k+1]$ for $T \in M_N(\Z)$.
For $[T,k], [S,k] \in H_k(A)$,
$[T,k]$ and $[S,k]$ define the same element in $H(A)$
if and only if $s_N(T) =s_N(S)$.
Define 
$\tilde{s}_N:H_k(A)\longrightarrow \Z$
by setting
$
\tilde{s}_N([T,k]) = s_N(T)
$
for $[T,k] \in H_k(A)$.
Since
$$
\tilde{s}_N(\iota_k([T,k])) =s_N(s_N(T)A) 
=  s_N(T)N^2
=N^2\tilde{s}_N([T,k]),
$$ 
we have  the sequences of commutative diagrams:
$$
\begin{CD}
H_1(A) @>\iota_1 >> H_2(A) @>\iota_2 >> H_3(A) @>\iota_3 >> \cdots @>>> H(A) \\
@V{\tilde{s}_N }VV  @V{\tilde{s}_N}VV@V{\tilde{s}_N}VV @. @VVV\\
\Z @> \times N^2 >> \Z@> \times N^2 >> \Z @> \times N^2 >> \cdots @>>>
\Z[\frac{1}{N^2}]
\end{CD}
$$
and 
$$
\begin{CD}
H_1(A) @>\iota_1 >> H_2(A) @>\iota_2 >> H_3(A) @>\iota_3 >> \cdots @>>> H(A) \\
@V{\tilde{s}_N }VV  @V{\frac{1}{N^2}\tilde{s}_N}VV@V{\frac{1}{N^4}\tilde{s}_N}VV @. @VVV\\
\bbR @> \id >> \bbR@> \id >> \bbR @> \id >> \cdots @>>>
\bbR.
\end{CD}
$$
Hence we may define an isomorphism
$\xi:H(A)\longrightarrow \Z[\frac{1}{N^2}] \subset \bbR
$
by setting
\begin{equation*}
\xi([T,k]) = \frac{1}{(N^2)^{k-1}}\tilde{s}_N([T,k]) 
            = \frac{1}{N^{2k-2}}s_N(T)
\in \Z[\frac{1}{N^2}]
\quad\text{ for } 
[T,k]\in H_k(A).
\end{equation*}
Since $\alpha([T,k]) = [A^2T, k+1]$
and
$s_N(A^2T) = N^2s_N(T),$
 we have
\begin{equation*}
\xi(\alpha([T,k])) 
=\frac{1}{(N^2)^{k}}s_N(A^2T) 
=\frac{1}{(N^2)^{k-1}}s_N(T) 
=\xi([T,k]),
\end{equation*}
so that the isomorphism
$\xi: H(A) \longrightarrow \Z[\frac{1}{N^2}]$
satisfies 
$\xi\circ\alpha= \xi,$
and hence we have $\alpha = \id$ on $H(A)$.
\end{proof}
As  $\id -\alpha$ is the zero map on $H(A)$
with $\Z[\frac{1}{N^2}] =\Z[\frac{1}{N}]$ in $\mathbb{R}$,
we thus have the  following proposition by the formula of Proposition \ref{prop:Kformula}.
\begin{proposition}[{cf. \cite[Section 3.3]{KilPut}}]\label{prop:K}
$
K_0(\R_N^a) \cong K_1(\R_N^a) \cong H(A) \cong \Z[\frac{1}{N}].
$
\end{proposition}
C. G. Holton proved that if an $N\times N$ matrix $A$ is aperiodic,
then the shift $\bar{\sigma}_N^*$ on the AF-algebra $C^*(G_A^a)$ has the Rohlin property \cite[Theorem 6.1]{Holton}.
For the $N\times N$ matrix 
$A=
\begin{bmatrix} 
1& \cdots & 1 \\
\vdots &   & \vdots \\
1& \cdots & 1 \\
\end{bmatrix},
$
the algebra
$C^*(G_N^a)$ which is the $C^*$-algebra of the groupoid $G_A^a$ is 
the UHF algebra of type $N^\infty$, 
so that the crossed product 
$\R_N^a =C^*(G_N^a)\rtimes_{\bar{\sigma}_N^*}\Z$
is a simple A${\mathbb{T}}$-algebra of real rank zero
with a unique tracial state by \cite[Theorem 1.1]{BKRS}, \cite[Theorem 1.3]{Kishimoto}.  
The unique tracial state on 
$\R_N^a$ is denoted by $\tau_N$.
It arises from the Parry measure on the full $N$-shift
$(\bar{X}_N, \bar{\sigma}_N)$ (Putnam \cite[Theorem 3.3]{Putnam1}).
We may determine the trace values of the $K_0$-group in the following way.
\begin{lemma}\label{lem:dimensionrange}
$\tau_{N*}(K_0(\R_N^a)) = \Z[\frac{1}{N}]$ in $\mathbb{R}$.
\end{lemma}
\begin{proof}
By Corollary \ref{cor:mains10.9}, the algebra
$\R_N^a$ is realized as the fixed point algebra 
of $\R_N^{s,u}$ under the diagonal gauge action. 
It is easy to see that  $\R_N^a$  is generated by linear span of operators
of the form
$
T_{\bar{\xi}}T_{\bar{\eta}}^* \otimes S_\mu S_\nu^* 
$
for
$\mu =(\mu_1,\dots,\mu_m), 
  \nu =(\nu_1,\dots,\nu_n) \in B_*(\bar{X}_A),
\,
\bar{\xi} =(\xi_k,\dots,\xi_1), 
\bar{\eta}=(\eta_l,\dots,\eta_1)
\in B_*(\bar{X}_{A^t})
$
such that $k+m= l+n$.
Since the tracial state $\tau_N$ on $\R_N^a$ comes from the Parry measure on
$\bar{X}_N$, we have
\begin{equation*}
\tau_N(T_{\bar{\xi}}T_{\bar{\eta}}^* \otimes S_\mu S_\nu^*)
=
\begin{cases}
\frac{1}{N^{k+m}} & \text{ if } \bar{\xi} = \bar{\eta},\, \mu = \nu, \\
0 & \text{ otherwise.} 
\end{cases}
\end{equation*}
Through the six-term exact sequence 
$$
\begin{CD}
K_0(C^*(G_N^a)) @>\id -\alpha >> K_0(C^*(G_N^a)) @>\id >> K_0(\R_N^a) \\
@AAA  @.  @VVV\\
K_1(\R_N^a) @< \id<< K_1(C^*(G_N^a)) @< \id -\alpha<< K_1(C^*(G_N^a))
\end{CD}
$$
for the crossed product
$\R_N^a = C^*(G_N^a)\rtimes\Z$
with the fact $ K_1(C^*(G_N^a)) =0$ and $\alpha = \id$,
all elements of $K_0(\R_N^a)$ come from  those of $K_0(C^*(G_N^a)) = H(A)$.
We thus conclude that 
$\tau_{N*}(K_0(\R_N^a)) = \Z[\frac{1}{N}].$
\end{proof}
For two natural numbers $1<M, N \in \N,$ let 
$M =p_1^{k_1}\cdots p_m^{k_m}, \, N =q_1^{l_1}\cdots q_n^{l_n}$
be the prime factorizations of $M, N$ such that 
$p_1<\cdots < p_m,\, q_1<\cdots<q_n$ and $k_1,\dots,k_m, l_1,\dots, l_n \in \N,$ 
respectively.
\begin{proposition}\label{prop:classRuelle}  
Keep the above notation. 
The following assertions are equivalent.
\begin{enumerate}
\renewcommand{\theenumi}{\roman{enumi}}
\renewcommand{\labelenumi}{\textup{(\theenumi)}}
\item
The Ruelle algebras $\R_M^a$ and $\R_N^a$ are isomorphic.
\item
$\Z[\frac{1}{M}] =\Z[\frac{1}{N}]$ as subsets of values of $\mathbb{R}.$
\item
$\{p_1,\dots,p_m\} = \{q_1,\dots,q_n\},$
that is, $m=n$ and $p_1 = q_1,\dots, p_m = q_n$.
\end{enumerate}
\end{proposition}
\begin{proof}
(i) $\Longrightarrow$ (ii):
Since the Ruelle algebras $\R_M^a$ and $\R_N^a$  have unique tracial state, respectively,
the assertion follows from the preceding lemma.

(ii) $\Longrightarrow$ (i):
The algebras $\R_M^a, \R_N^a$ are both A${\mathbb{T}}$-algebras of real rank zero with unique tracial state.
The condition $\Z[\frac{1}{M}] =\Z[\frac{1}{N}]$
implies that their $K$-theoretic dates 
$$
(K_0(\R_M^a), K_0(\R_M^a)_+, [1], K_1(\R_M^a))
=
(K_0(\R_N^a),K_0(\R_N^a)_+, [1], K_1(\R_N^a))
$$
coincide because of Proposition \ref{prop:K} and Lemma \ref{lem:dimensionrange}.
By a general classification theory of simple A$\mathbb{T}$-algebras of real rank zero,
we conclude that the Ruelle algebras $\R_M^a$ and $\R_N^a$ are isomorphic.
 
The equivalence (ii) $\Longleftrightarrow$ (iii) is easy.
\end{proof}
We have the following corollary.
\begin{corollary}\label{coro:acoefull}
Let $M =p_1^{k_1}\cdots p_m^{k_m}$ and $N =q_1^{l_1}\cdots q_n^{l_n}$
be the prime factorizations of $ M, N$ as in the above proposition.
If the sets   
$\{p_1,\dots,p_m\}$ and $\{q_1,\dots,q_n\}$
do not coincide with each other, 
then the two-sided full shifts
$(\bar{X}_M, \bar{\sigma}_M)$ and 
$(\bar{X}_N,\bar{\sigma}_N)$
are not asymptotically continuous orbit equivalent.  
\end{corollary}
\begin{proof}
Suppose that 
$\{p_1,\dots,p_m\} \ne \{q_1,\dots,q_n\}.$
By the above proposition,  
the Ruelle algebras
$\R_N^a, \R_M^a$ are not isomorphic.
Since the isomorphism class of the Ruelle algebra is invariant 
under asymptotic continuous orbit equivalence by Theorem \ref{thm:main3},
we know that 
$(\bar{X}_M,\bar{\sigma}_M)$ and 
$(\bar{X}_N,\bar{\sigma}_N)$
are not asymptotically continuous orbit equivalent.  
\end{proof}

\section{Concluding remarks}
Before ending the paper,
we refer to differences among
asymptotic continuous orbit equivalence,   asymptotic  conjugacy and  topological conjugacy 
of Smale spaces.
It is may be proved that 
topological conjugacy  implies asymptotic  conjugacy, which implies 
asymptotic continuous orbit equivalence.
For an irreducible  Smale space $(X,\phi)$, 
its inverse system $(X,\phi^{-1})$
automatically becomes  an irreducible Smale space by definition.
We then see the following
\begin{proposition}
An irreducible Smale space $(X,\phi)$ is asymptotically continuous orbit equivalent to
its inverse $(X,\phi^{-1}).$
\end{proposition}
\begin{proof}
In Definition \ref{def:acoe}, 
we set $Y = X, \psi = \phi^{-1}$ 
and take $ h =\id, \, c_1 \equiv -1, \, c_2 \equiv -1, \, d_1 \equiv 0, \, d_2\equiv 0.$
We then see that 
 $c_1^n(x) = -n$
for all $x\in X$ and
$c_2^n(y) = -n$ for all $ y \in Y.$
It is direct to see that 
all conditions in Definition \ref{def:acoe} hold for these $c_1, c_2, d_1, d_2$,
so that 
$(X,\phi)$ is asymptotically continuous orbit equivalent to
its inverse $(X,\phi^{-1}).$
\end{proof}
We may easily explain the above situation in terms of $C^*$-algebras.
We actually see that the identity map $\id: X\longrightarrow X$ 
induces an isomorphism $\Phi: \R^a_\phi \longrightarrow \R^a_{\phi^{-1}}$
of $C^*$-algebras such that 
$$
\Phi(C(X)) = C(X)\quad \text{ and } \quad 
\Phi\circ \rho^\phi_t = \rho^{\phi^{-1}}_{-t} \circ \Phi,
$$
because in Theorem \ref{thm:1.1} (iii),
we may have.  
$$
\Ad(U_t(c_{\phi^{-1}})) = \rho_{-t}^{\phi^{-1}}, \qquad
\Ad(U_t(c_{\phi})) = \rho_{t}^{\phi}.
$$
\begin{corollary}\label{cor:11.2} 
There exists a pair $(X,\phi)$ and $(Y,\psi)$ 
of irreducible Smale spaces such that they are asymptotically continuous orbit equivalent
but not topologically  conjugate. 
\end{corollary}
\begin{proof}
As in \cite[Example 7.4.19]{LM},
the matrix 
$A =
\begin{bmatrix}
19 & 5 \\
4 & 1
\end{bmatrix}
$
is not shift equivalent to
its transpose 
$A^t =
\begin{bmatrix}
19 &4 \\
5 & 1
\end{bmatrix}.
$
Let 
$(X,\phi)$ and $(Y,\psi)$
be the shifts of finite type defined by the matrices $A$ and $A^t$, respectively.
Since $(Y,\psi)$ is naturally topologically conjugate to $(X,\phi^{-1}),$
the Smale spaces
$(X,\phi)$ and $(Y,\psi)$ are asymptotically continuous orbit equivalent by the preceding proposition.
As shift equivalence relation of matrices is weaker than strong shift equivalence,
by Williams' theorem \cite{Williams} the shifts of finite type  $(X,\phi)$ and $(Y,\psi)$ are not topologically conjugate.
\end{proof}
In the recent paper:

K. Matsumoto, Topological conjugacy of topological Markov shifts and Ruelle algebras,
\, arXiv:1706.07155,

\noindent
which is a continuation of this paper, the author shows that two-sided topological Markov shifts 
are topologically conjugate if and only if they are asymptotically conjugate.
Hence the example in the proof of Corollary \ref{cor:11.2} shows us that 
there exists a pair $(X,\phi)$ and $(Y,\psi)$ 
of irreducible Smale spaces such that they are asymptotically continuous orbit equivalent
but not asymptotically conjugate. 
For a general irreducible Smale space, however,
we do not know whether or not the asymptotic conjugacy implies topological conjugacy.
This is an open question probably being affirmative.

We finally remark the following.
We know that if two irreducible topological Markov shifts are asymptotically continuous orbit equivalent,
then their asymptotic Ruelle algebras are isomorphic by Theorem \ref{thm:main3}.
Since these asymptotic Ruelle algebras
$\R^a_A$ have unique tracial states $\tau_A$ coming from 
the Parry measures on the shift spaces.
Hence the trace values 
$\tau_{A*}(K_0(\R^a_A))$ are invariant under asymptotic continuous orbit equivalence.
For two matrices
$
A = \begin{bmatrix}
1 & 1 \\
1 & 1 
\end{bmatrix},
\,
B = \begin{bmatrix}
1 & 1 \\
1 & 0 
\end{bmatrix}
$
it is straightforward to see that 
$
\tau_{A*}(K_0(\R^a_A))\ne \tau_{B*}(K_0(\R^a_B))
$      
as subsets of $\mathbb{R},$
because 
$\tau_{B*}(K_0(\R^a_B))$ contains the trace values of the dimension group
of the AF-algebra defined by the matrix $B$.
Hence we know that the two-sided topological Markov shifts
$(\bar{X}_A,\bar{\sigma}_A)$ and 
$(\bar{X}_B,\bar{\sigma}_B)$
are not asymptotically continuous orbit equivalent,
whereas their one-sided topological Markov shifts
$({X}_A,{\sigma}_A)$ and 
$({X}_B,{\sigma}_B)$
are continuous orbit equivalent
as in \cite[Section 5]{MaPacific}.

\bigskip

{\it Acknowledgments:}
The author would like to deeply thank 
Ian F. Putnam for his comments and suggestions on the earlier version of the paper.
He kindly pointed out an error of the earlier version and call the author's  attention
to the papers \cite{Holton}, \cite{KilPut} and \cite{Putnam3}.  
The author also deeply thanks the referees for their careful reading of the paper and many helpful suggestions and useful advices.
Especially Lemma \ref{lem:irreessfree} with its detailed proof 
and
Lemma \ref{lem:irresuessfree}
were due to the referees.
The formulation of Proposition \ref{prop:extendedtensor} was  also suggested by the referees.
The referee's question made Section 11 come up.   
This work was supported by JSPS KAKENHI Grant Number 15K04896.



\begin{thebibliography}{99}



\bibitem{AnanRenault}
{\sc C. Anantharaman-Delaroche and J. Renault},
{\it Amenable Groupoids}, 
L'Enseignement Math\'ematique Gen\'eve, 2000.



\bibitem{Bowen2}{\sc R. Bowen},
{\it Equilibrium states and the ergodic theory of Anosov diffeomorphisms},
Lecture notes in Math. Springer, Berlin 1975, No. 475.

\bibitem{Bowen3}{\sc R. Bowen},
{\it On Axiom A diffeomorphisms},
CBMS Lecture Notes 35, Amer. Math. Soc. Providence, 1978.

\bibitem{BKRS}{\sc O. Bratteli, A. Kishimoto, M. R{\o}rdam and E. St{\o}rmer},
{\it The crossed product of a UHF algebra by a shift},
Ergodic Theory Dynam. Systems {\bf 13}(1993), pp. \, 615--626.



\bibitem{CR}
{\sc T. M. Carlsen and J. Rout},
{\it Diagonal-preserving gauge invariant isomorphisms of graph $C^*$-algebras},
J. Func. Anal. {\bf 273}(2017), pp. \, 2981--2993.


\bibitem{CRS}
{\sc T. M. Carlsen, E. Ruiz and A. Sims},
{\it Equivalence and stable isomorphism of groupoids, and diagonal-preserving 
stable isomorphisms of graph $C^*$-algebras and Leavitt path algebras},
preprint, arXiv: 1602.02602 [math. OA].
















\bibitem{CK}{\sc J. ~Cuntz and W. ~Krieger},
{\it A class of $C^*$-algebras and topological Markov chains},
 Invent.\ Math.\
 {\bf 56}(1980), pp.\ 251--268.


%



\bibitem{Holton}{\sc C. G. Holton},
{\it The Rohlin property for shifts of finite type},
J. Funct. Anal. {\bf 229}(2005), pp. \ 277--299.




\bibitem{KamPut}{\sc J. Kaminker and I. F. Putnam},
{\it K-theoretic duality for shifts of finite type},
Comm. Math. Phys. {\bf 187}(1997),  509--522.

\bibitem{KamPutSpiel}{\sc J. Kaminker, I. F. Putnam and J. Spielberg},
{\it Operator algebras and hyperbolic dynamics},
Operator algebras and quantum field theory (Rome, 1996), 
525--532, Int. Press, Cambridge, MA, 1997.



 \bibitem{KilPut}
 {\sc D. B. Killough and I. F. Putnam}, 
 {\it Ring and module structures 
 on dimension groups associated with a shift of finite type}, 
Ergodic Theory Dynam. Systems {\bf 32}(2012), 1370--1399. 

\bibitem{Kishimoto}{\sc A. Kishimoto},
{\it The Rohlin property for automorphisms of UHF algebras},
J. Reine Angew. Math. 
{\bf 465}(1995), 183--196.









\bibitem{LM}{\sc D. ~Lind and B. ~Marcus},
{\it An introduction to symbolic dynamics and coding},
 Cambridge University Press, Cambridge
(1995).















\bibitem{MaPacific}{\sc K. Matsumoto},
{\it Orbit equivalence of topological Markov shifts and Cuntz--Krieger algebras},
Pacific J.\ Math.\  {\bf 246}(2010), 199--225.

























%





\bibitem{MaJOT2015}{\sc K. Matsumoto},
{\it Strongly continuous orbit equivalence of 
one-sided topological Markov shifts},
J. Operator Theory {\bf 74}(2015), pp. 101--127.





\bibitem{MaProcAMS2017}
{\sc K. Matsumoto},
{\it Uniformly continuous orbit equivalence of Markov shifts and gauge actions on Cuntz--Krieger algebras},
Proc. Amer. Math. Soc. {\bf 145}(2017), pp. \ 1131--1140.  






\bibitem{MaMZ2017}
{\sc K. Matsumoto},
{\it Continuous orbit equivalence, flow equivalence of Markov shifts and circle actions on Cuntz--Krieger algebras},
Math. Z. {\bf 285}(2017), pp. 121--141.  




\bibitem{MMKyoto}{\sc K. Matsumoto and H. Matui},
{\it Continuous orbit equivalence of topological Markov shifts 
and Cuntz--Krieger algebras},
Kyoto J. Math.{\bf 54}(2014), pp.\ 863--878.


\bibitem{MMEDTD} {\sc K. Matsumoto and H. Matui},
{\it Continuous orbit equivalence of topological Markov shifts 
and dynamical zeta functions},
Ergodic Theory Dynam. Systems
{\bf 36}(2016), pp. \ 1557--1581.




\bibitem{MatuiPLMS}{\sc H. Matui}, 
{\it Homology and topological full groups of {\'e}tale groupoids on totally disconnected spaces},
Proc. London Math. Soc. {\bf 104}(2012),  pp.\ 27--56.



\bibitem{Matui2015}
{\sc H. Matui}, 
{\it Topological full groups of one-sided shifts of finite type},
J. Reine Angew. Math.
{\bf 705}(2015), pp. \ 35--84.

%


\bibitem{MRW}
{\sc P. S. Muhly, J. Renault and D. P. Williams},
{\it Equivalence and isomorphism for groupoid $C^*$-algebras},
J. Operator Theory
{\bf 17}(1987),
pp.\ 3--22. 



\bibitem{PP}
{\sc W. Parry and M. Pollicott},
{\it Zeta functions and the periodic orbit structure of hyperbolic dynamics},
Ast{\'e}risque
 {\bf 187-188}(1990).









\bibitem{Putnam1}{\sc I. F. Putnam},
{\it  $C^*$-algebras from Smale spaces},
Canad.\ J.\ Math.\ {\bf 48}(1996), pp.\ 175--195.

\bibitem{Putnam2}{\sc I. F. Putnam},
{\it  Hyperbolic systems and generalized Cuntz--Krieger algebras},
Lecture Notes, Summer School in Operator Algebras, Odense,  August 1996.


\bibitem{Putnam3}{\sc I. F. Putnam},
{\it  Functoriality of the $C^*$-algebras associated with hyperbolic dynamical systems},
J.\ London Math.\ Soc.\ {\bf 62}(2000), pp.\ 873--884.



\bibitem{Putnam4}{\sc I. F. Putnam},
{\it  A homology theory for  Smale spaces},
Memoires of Amer. Math. Soc. {\bf 232}(2014), No. 1094.



\bibitem{PutSp}{\sc I. F. Putnam and J. Spielberg},
{\it  The structure of $C^*$-algebras associated with hyperbolic dynamical systems},
J. Func. Anal. {\bf 163}(1999), pp. \. 279--299.












\bibitem{Renault}{\sc J. Renault},
{\it A groupoid approach to $C^*$-algebras},
Lecture Notes in Math.  793, 
Springer-Verlag, Berlin, Heidelberg and New York (1980).


\bibitem{Renault2}{\sc J. Renault},
{\it  Cartan subalgebras in $C^*$-algebras},
Irish Math. Soc. Bull. 
{\bf 61}(2008), pp.\ 29--63. 




\bibitem{Renault3}{\sc J. Renault},
{\it Examples of masas in $C^*$-algebras},
 Operator structures and dynamical systems, pp.\ 259--265, 
 Contemp. Math., 503, Amer. Math. Soc., Providence, RI, 2009. 







\bibitem{Ruelle1}{\sc D. Ruelle},
{\it Thermodynamic formalism}, Addison-Wesley, Reading (Mass.) (1978).

\bibitem{Ruelle2}{\sc D. Ruelle},
{\it Non-commutative algebras for hyperbolic diffeomorphisms},
Invent. Math.  {\bf 93}(1988), pp.\ 1--13.




\bibitem{Ruelle3}{\sc D. Ruelle},
{\it Dynamical zeta functions and transfer operators},
Notice  Amer.\ Math.\ Soc. {\bf 49}(2002), pp.\ 175--193.














\bibitem{Smale}{\sc S. Smale},
{\it Differentiable dynamical systems},
Bull. Amer. Math. Soc. {\bf 73}(1967), pp. \ 747--817.



\bibitem{Thomsen}
{\sc K. Thomsen},
{\it $C^*$-algebras of homoclinic and heteroclinic structure in expansive dynamics},
Memoirs of Amer. Math. Soc. 
{\bf 206}(2010), No 970.















\bibitem{Williams} {\sc R. F. Williams},
{\it Classification of subshifts of finite type}, 
 Ann.\ Math.\  {\bf 98}(1973), pp.\ 120--153.
 erratum, Ann.\ Math.\
{\bf 99}(1974), pp.\ 380--381.

\end{thebibliography}
\end{document}